\documentclass[final,3p,times]{elsarticle}

\usepackage{epsfig}\usepackage{graphicx}
\usepackage{amsmath,amssymb}
\usepackage{amsthm}

\newtheorem {theorem} {Theorem}
\usepackage{booktabs}
\newcommand{\comment}[1]{}
\newtheorem{cor}{Corollary}[]
\newtheorem{lem}{Lemma}[]

\usepackage[bookmarks,colorlinks,breaklinks]{hyperref}
\hypersetup{
    colorlinks,%
    citecolor=black,%
    filecolor=black,%
    linkcolor=black,%
    urlcolor=black
}
\begin{document}
\title{Variable Bandwidth Diffusion Kernels}
\author[rvt]{Tyrus Berry\corref{cor}}
\ead{thb11@psu.edu}
\author[rvt,rvt2]{John Harlim}
\ead{jharlim@psu.edu}
\cortext[cor]{Corresponding author}
\address[rvt]{Department of Mathematics, the Pennsylvania State University, 109 McAllister Building, University Park, PA 16802-6400, USA}
\address[rvt2]{Department of Meteorology, the Pennsylvania State University, 503 Walker Building, University Park, PA 16802-5013, USA}
\date{\today}

\begin{abstract}
Practical applications of kernel methods often use variable bandwidth kernels, also known as self-tuning kernels, however much of the current theory of kernel based techniques is only applicable to fixed bandwidth kernels.  In this paper, we derive the asymptotic expansion of these variable bandwidth kernels for arbitrary bandwidth functions; generalizing the theory of Diffusion Maps and Laplacian Eigenmaps.  We also derive pointwise error estimates for the corresponding discrete operators which are based on finite data sets; generalizing a result of Singer which was restricted to fixed bandwidth kernels.  Our analysis reveals how areas of small sampling density lead to large errors, particularly for fixed bandwidth kernels.  We explain the limitation of the existing theory to data sampled from compact manifolds by showing that when the sampling density is not bounded away from zero (which implies that the data lies on an open set) the error estimates for fixed bandwidth kernels will be unbounded.  We show that this limitation can be overcome by choosing a bandwidth function inversely proportional to the sampling density (which can be estimated from data) which allows us to control the error estimates uniformly over a non-compact manifold.  We numerically verify these results on non-compact manifolds by constructing the generator of the Ornstein-Uhlenbeck process on a real line and a two-dimensional plane using data sampled independently from the respective invariant measures. We also verify our results on compact manifolds by constructing the Laplacian on the unit circle and the unit sphere and we show that the variable bandwidth kernels exhibit reduced sensitivity to bandwidth selection and give better results for an automatic bandwidth selection algorithm.  
\end{abstract}


\maketitle

\section{Introduction}

Graph Laplacian and kernel based techniques are ubiquitous in machine learning, clustering, classification.  While these practical algorithms have been very successful in various applications, they were not mathematically understood until the development of Laplacian Eigenmaps \cite{BN} and Diffusion Maps \cite{diffusion} as well as other works on the convergence of graph Laplacians to their continuous counterparts \cite{Hein05,Ting2010}.  The theory of \cite{BN,diffusion,SingerEstimate} applies to fixed bandwidth kernels of the form,
\begin{align}K_\epsilon(x,y)=h\left(\frac{\|x-y\|^2}{\epsilon}\right),\label{DMkernel}\end{align} 
where $h$ decays exponentially as the distance $\|x-y\|$ increases.  The novel perspective taken by \cite{BN,diffusion} was that by evaluating the kernel function on all pairs of data points, we can approximate a geometric operator for functions defined on the data set.  
  
Most applications of these kernel based techniques use variable bandwidth kernels, also known as self-tuning kernels (see for example \cite{selftuning}) which have the form,
\begin{align}\label{vbkernel} K^S_{\epsilon}(x,y) = h\left(\frac{||x-y||^2}{\epsilon \rho(x)\rho(y)}\right), \end{align} 
where the superscript $S$ indicates that the bandwidth function, $\rho$, is applied symmetrically to both $x$ and $y$.  In fact, this class of kernels has been routinely used in kernel density estimation problems \cite{RosenblattFBK,ParzenFBK} and the choice of bandwidth function, $\rho$, is  known to be especially important for accurate estimation of the tails of the distribution and regions of sparse sampling, see for example \cite{ScottVBK,ScottVBK2}. Independently, kernels of the form \eqref{vbkernel} have been proposed in \cite{giannakisMajda,GiannakisPNAS} for describing data generated by a dynamical system.  They chose a bandwidth function, $\rho$, based on the distance traveled in state space in a fixed time unit. Combined with time-delay embeddings, this bandwidth function was shown to give a natural scaling which reduces the dependence on the initial observation function. 

Despite the empirical success of the variable bandwidth kernels \eqref{vbkernel}, the asymptotic theory of \cite{BN,diffusion,SingerEstimate} has not been fully extended beyond the fixed bandwidth kernels of the form \eqref{DMkernel}. The first theoretical result on variable bandwidth kernels was established in \cite{Ting2010}, which derived the limiting operator of a sequence of kernels, which includes kernels of the form \eqref{vbkernel}.  However, the results of \cite{Ting2010} do not show the rate of convergence to the limiting operator or give error bounds for the discrete approximations; this is because their proof follows from a very general theorem on the convergence of sequences of Markov processes. The main contribution of this paper is to extend the asymptotic analysis of \cite{diffusion} and the discrete analysis in \cite{SingerEstimate} to give the first rigorous error bounds for the variable bandwidth kernels in \eqref{vbkernel}. These error bounds reveal that the pointwise errors have a complex dependence on the sampling measure and the bandwidth function.  In particular, this result will reveal that for a fixed bandwidth kernel, the error becomes unbounded as the sampling measure approaches zero.  This fact explains why the theory of \cite{BN,diffusion} requires the data to lie on a compact manifold, which is a significant restriction for practical applications when data is sparsely sampled in some regions. Finally, using a particular choice of bandwidth function that is inversely proportional to the sampling density, we are able to control the error in the regions of sparse sampling and extend the theory of \cite{diffusion} to non-compact manifolds for the first time. Intuitively, this special choice of kernel uses a large bandwidth in areas of sparse sampling and a small bandwidth in areas of dense sampling.

Coifman and Lafon \cite{diffusion} introduced the diffusion map as a way to represent data on a low-dimensional Euclidean space. When the data is generated by a dynamical system, the diffusion map algorithm can be used to find low-dimensional representations of the long-time dynamics as shown in \cite{diffcoords,diffusionslowmanifold,diffusionreductioncoords}. These ideas have been generalized to time-delay reconstructions of dynamical systems in the work of \cite{DMDC} and related work in \cite{giannakisMajda,GiannakisPNAS}. Building upon this idea, the variable bandwidth kernel developed in the present paper has been applied by the authors for quantifying uncertainties of stochastic gradient flow systems  \cite{BH14UQ}. In particular, they used the operator estimates developed in this paper as nonparametric models to address nonlinear forecasting, filtering, and response problems.

The remainder of this paper is organized as follows: In Section \ref{resultsSection} we present the main theoretical result as well as practical considerations for the application of this result.  The detailed proofs of this theorem are described in two appendices. In Section \ref{algorithm} we give the details of the numerical algorithm including some important numerical considerations for optimal implementation.  In Section \ref{numerics} we present numerical examples on unbounded manifolds to validate the results of Theorem \ref{mainresult}.  In Section \ref{circleexample} we present numerical examples on compact domains to compare fixed and variable bandwidth kernels for estimating the eigenfunctions of the Laplacian operator. We conclude this paper with a short summary in Section \ref{conclusion}.

\section{Main Result}\label{resultsSection}

The geometric perspective of \cite{BN,diffusion} was to construct a
stochastic matrix whose generator, $L_{\epsilon,\alpha}$, is a discrete representation of a continuous Kolmogorov operator,
\begin{align} \mathcal{L}_{\alpha}f = \Delta f + (2-2\alpha)\nabla f \cdot \frac{\nabla q}{q},\label{ELf} \end{align}
acting on smooth functions $f$. The stochastic matrix $L_{\epsilon,\alpha}$ is constructed by evaluating the homogeneous kernel \eqref{DMkernel} on all pairs of points $\{x_i\}_{i=1}^N\subset\mathbb{R}^n$, sampled from a $d$-dimensional manifold $\mathcal{M} \subset \mathbb{R}^n$ with smooth sampling density $q$. Note that $\Delta$ denotes the Laplacian and $\nabla$ denotes the gradient operator, and both operators are defined with respect to the Riemannian metric that the manifold $\mathcal{M}$ inherits from the ambient space $\mathbb{R}^n$. We should clarify that in this paper, we refer to the Laplacian operator as the negative of Laplace-Beltrami operator for convenience. The parameter $\alpha$ controls the degree to which the sampling distribution is allowed to bias the operator, and a key result of \cite{diffusion} is that setting $\alpha = 1$ removes the bias entirely and recovers the Laplacian operator independent of the sampling density $q$. Moreover, setting $\alpha=1/2$ recovers the backward Kolmogorov operator of a gradient flow with potential $U=-\log{q}$, and setting $\alpha=0$ recovers the normalized graph Laplacian on a graph with isotropic (Gaussian) weights \cite{BN} which approximates the Laplacian operator when the sampling density is uniform. For the case of $\alpha=0$ with uniform sampling, the error estimates of Singer \cite{SingerEstimate} showed that for any sufficiently smooth function $f$ at any point $x_i$ in the data set,
\begin{align} L_{\epsilon,\alpha}f(x_i) \equiv \frac{1}{\epsilon}\left(\frac{\sum_j K_{\epsilon}(x_i,x_j)f(x_j)}{\sum_j K_{\epsilon}(x_i,x_j)} - f(x_i)\right) = \mathcal{L}_\alpha f(x_i) + \mathcal{O}\left(\epsilon,\frac{||\nabla f(x_i)||}{\sqrt{N} \epsilon^{1/2+d/4}} \right) \label{singer}\end{align}
with high probability. The discrete operator $L_{\epsilon,\alpha}$ is closely related to the graph Laplacian with edge weights $w_{ij} = K_{\epsilon}(x_i,x_j)$ given by the kernel $K_{\epsilon}$ in \eqref{DMkernel} (see \cite{BN,diffusion}).

Our main contribution in this paper can be summarized as follows:

\begin{theorem}\label{mainresult} Let $q\in L^1(\mathcal{M}) \cap \mathcal{C}^3(\mathcal{M})$ be a density that is bounded above on an embedded $d$-dimensional manifold $\mathcal{M} \subset \mathbb{R}^n$ without boundary and let $\{x_i\}_{i=1}^N$ be sampled independently with distribution $q$.  
Let $K^S_{\epsilon}(x,y)$ be a variable bandwidth kernel of the form \eqref{vbkernel} with bandwidth function $\rho$ and shape function $h:[0,\infty) \to [0,\infty)$ with exponential decay at infinity.
For a smooth real-valued function $f \in L^2(\mathcal{M},q)\cap \mathcal{C}^3(\mathcal{M})$ and an arbitrary point $x_i \in \mathcal{M}$, define the discrete functionals,
\begin{align}
F_i(x_j) &= \frac{K^S_{\epsilon}(x_i,x_j)f(x_j)}{q^S_{\epsilon}(x_i)^\alpha q^S_{\epsilon}(x_j)^\alpha},\quad\quad
G_i(x_j) = \frac{K^S_{\epsilon}(x_i,x_j)}{q^S_{\epsilon}(x_i)^\alpha q^S_{\epsilon}(x_j)^\alpha}, \nonumber
\end{align}
where $q^S_{\epsilon}(x_i) = \sum_l K^S_{\epsilon}(x_i,x_l)/\rho(x_i)^{d}$ is proportional to a kernel density estimate of the sampling density $q$.  Then, with high probability,
\begin{align}
L^S_{\epsilon,\alpha}f(x_i) &\equiv  \frac{1}{\epsilon m \rho(x_i)^{2}}\left(\frac{\sum_{j}F_i(x_j)}{\sum_{j}G_i(x_j)}-f(x_i)\right) \nonumber \\ &= \mathcal{L}_{\alpha,\rho} f(x_i) + \mathcal{O}\left(\epsilon, \frac{q(x_i)^{1/2}\rho(x_i)^{-d/2}}{\sqrt{N}\epsilon^{2+d/4}} ,\frac{||\nabla f(x_i)||q(x_i)^{-(1/2-2\alpha+2d\alpha)}\rho(x_i)^{-(d/2+1)}}{\sqrt{N}\epsilon^{1/2+d/4}} \right), \label{errorbound}
\end{align}
for some finite valued constant $m$, where 
\begin{align}\mathcal{L}_{\alpha,\rho} f \equiv\Delta f + 2(1-\alpha)\nabla f \cdot \frac{\nabla q}{q}+ (d+2) \nabla f \cdot \frac{\nabla \rho}{\rho}.\label{Larho}\end{align}
\end{theorem}

 We should note while this result is consistent with that of \cite{Ting2010}, which showed the convergence of the variable bandwidth kernels to the operator in \eqref{Larho}, Theorem~\ref{mainresult} also provides error bounds for both the continuous and discrete operators, generalizing the asymptotic theory \cite{diffusion,SingerEstimate} to the variable bandwidth kernels of the form \eqref{vbkernel}. The detailed proof of this theorem is given in two appendices below. In \ref{VBK}, we derive the first component of the error bounds in \eqref{errorbound}. In \ref{discreteOps}, we derive the second and third components of the error bounds in \eqref{errorbound}.

For a particular choice of bandwidth function, this error bound will reveal some interesting implications when the data is sparsely sampled as we will discuss below. Choosing $q= \rho^\beta+\mathcal{O}(\epsilon)$, we deduce,

\begin{cor}\label{corollary} Given the same hypothesis as in Theorem~\ref{mainresult}, let the bandwidth function $\rho = q^{\,\beta} + \mathcal{O}(\epsilon)$. Then, with high probability,
\begin{align}\label{errorestimate}
L^S_{\epsilon,\alpha,\beta}f(x_i)  = \mathcal{L}_{\alpha,\beta} f(x_i) + \mathcal{O}\left(\epsilon, \frac{q(x_i)^{(1-d\beta)/2}}{\sqrt{N}\epsilon^{2+d/4}} ,\frac{||\nabla f(x_i)||q(x_i)^{-c_2}}{\sqrt{N}\epsilon^{1/2+d/4}} \right), 
\end{align}
where the discrete operator $L^S_{\epsilon,\alpha,\beta}$ is $L^S_{\epsilon,\alpha}$ in \eqref{errorbound} where bandwidth function $\rho$ is replaced with $q^\beta$, and the continuous operator $\mathcal{L}_{\alpha,\rho}$ in \eqref{Larho} becomes,
\begin{align}\mathcal{L}_{\alpha,\beta} f \equiv\Delta f + c_1\nabla f \cdot \frac{\nabla q}{q},\label{generator_ab}\end{align}
with $c_1 = 2-2\alpha + d\beta + 2\beta$ and $c_2 = 1/2-2\alpha+2d\alpha +d\beta/2+\beta$.
\end{cor}

The key to applying Corollary~\ref{corollary} is in the choices of $\alpha$ and $\beta$.  The first consideration for choosing $\alpha$ and $\beta$ is that on a non-compact manifold $q$ may become arbitrarily close to zero.  In order to bound the error terms on a non-compact manifold, the exponent terms in the error components in \eqref{errorestimate} must satisfy $(1-d\beta)/2 > 0$ and $c_2<0$.  Intuitively, we expect $\beta<0$ to have the best results since this increases the bandwidth in areas of sparse sampling and decreases it in areas of dense sampling.  Some natural choices for $\beta$ are $-1/2$, $-1/d$ and $-1/(d/2+1)$.  One advantage of the choice $\beta = -1/2$ is that we find $c_1 = 1-2\alpha - d/2$ and so simply by taking $\alpha$ sufficiently less than zero we can guarantee $c_1>0$ even when the dimension is unknown.  This is important for gradient flow systems where the fundamental properties (such as the invariant measure) are dramatically altered if $c_1 < 0$.  In all the examples in this paper, the variable bandwidth kernel will use $\beta=-1/2$ and the fixed bandwidth kernel will use the same algorithm with $\beta=0$.

The second consideration for choosing $\alpha$ and $\beta$ is the desired operator, which is controlled by $c_1$, and in this paper we will be interested in two operators.  First, we will be interested in the case of gradient flow systems for which $c_1=1$ and therefore $\beta=-1/2$ implies that $\alpha = -d/4$.  Second, we will consider the case of finding the Laplacian, for which $c_1=0$ and so $\beta=-1/2$ implies $\alpha = 1/2-d/4$.  When $c_1=0$ the operator $\mathcal{L}_{\alpha,\beta}$ in \eqref{generator_ab} is the Laplacian on the manifold and is \emph{independent of the sampling density}.  This fact was first realized by \cite{diffusion} for the fixed bandwidth case, where $\alpha=1$, $\beta=0$.  This is an important consideration for practitioners, since it is often valuable for an algorithm to be independent of changes in the sampling measure.  Alternatively, some applications may require the sampling to bias the results, and in these cases increasing $c_1>0$ gives a natural way to increase the bias of the sampling on the operator since the invariant measure of the operator is proportional to $q^{c_1}$ where $q$ is the sampling measure.

Notice that the bandwidth function used in Corollary~\ref{corollary} is $\rho = q^{\beta} + \mathcal{O}(\epsilon)$ which implies that we do not require the sampling density to be exactly known.  For practical applications we may use any kernel density estimate to find an order-$\epsilon$ approximation of $q$ for the purposes of defining the bandwidth function $\rho$.  However, we note that the normalization term $q^S_{\epsilon}$ may not be replaced with an alternate density estimate because the result in Theorem~\ref{mainresult} carefully accounts for the higher order terms in the asymptotic expansion of $q^S_{\epsilon}$.

The error bound in Corollary~\ref{corollary} has three components. The $\mathcal{O}(\epsilon)$ component is due to the error between the true operator $\mathcal{L}_{\alpha,\beta}$ and the operator $L_{\epsilon,\alpha,\beta}^S$ in \eqref{errorestimate} with the summations replaced by the expectations, where $\mathbb{E}[f] \equiv \int_{\cal M} f(z)q(z)dV(z)$. In other words, for fixed $\epsilon$, in the limit of large $N$ the error will be of order $\epsilon$ assuming the first term dominates.  The error term $\mathcal{O}\left(\frac{q(x_i)^{(1-d\beta)/2}}{\sqrt{N}\epsilon^{2+d/4}}\right)$ is due to the need to obtain an order-$\epsilon^2$ estimate of $q_{\epsilon}^S$.  While this term dominates for $q = \mathcal{O}(1)$, as $q\to 0$ the third component of the error may become dominant.  
The error term $\mathcal{O}\left(\frac{||\nabla f(x_i)||q(x_i)^{-c_2}}{\sqrt{N}\epsilon^{1/2+d/4}}\right)$ is due to the bias error between the ratio of discrete sums $\sum_j F_i(x_j)/\left(\sum_j G_i(x_j)\right)$ and the continuous expectations $\mathbb{E}[F_i]/\mathbb{E}[G_i]$.  Depending on the choices of $\alpha$ and $\beta$, this term can dominate the error in areas of sparse sampling where $q$ is small.  In particular, if $q$ is not bounded away from zero, as $N$ increases the data will begin to sample areas of small density, and this final error term can actually increase as the amount of data increases when $c_2 > 0$.

The $\alpha$ normalization in the functionals $F_i, G_i$ in Theorem~\ref{mainresult} is a de-biasing parameter, which is equivalent to the diffusion maps $\alpha$ normalization so that when $\beta=0$ we recover the operator $\mathcal{L}_{\alpha,0}=\mathcal{L}_\alpha$ in \eqref{ELf} since $c_1 = 2-2\alpha$.  Notice that $\beta=0$ is the one case which does not require knowledge of the intrinsic dimension $d$ of the manifold $\mathcal{M}$.  However, when $\beta=0$, $\alpha>0$, and $d\in\mathbb{N}$, we have $c_2 = 1/2+2\alpha(d-1)>0$ which means that the error may be unbounded as $q\to 0$.  This crucial observation explains why using the fixed bandwidth kernel in \eqref{DMkernel} may produce large error estimates when the sampling measure is not bounded away from zero. By taking $\beta < 0$ we can make $c_2\leq 0$ which implies the pointwise errors are uniformly bounded and we will recover the continuous operator in the limit of large data.  Taking $\beta<0$ will require knowledge of the intrinsic dimension $d$ of $\mathcal{M}$, see \cite{HeinDimension,MaggioniDimension} for some methods and considerations for estimating the intrinsic dimension, and we suggest a new method in Section \ref{circleexample}.

A related issue on non-compact manifolds is that $||\nabla f||$ may be unbounded.  In particular, in Section \ref{numerics} we will consider the Hermite polynomials which are eigenfunctions of the Kolmogorov operator of a stochastically forced gradient flow with a quadratic potential on the real line.  As long as $q$ has sufficiently fast decay at infinity and $c_2<0$, the term $q^{-c_2}$ will control the growth of $||\nabla f||$ to allow for uniform pointwise error bounds on the data set.  Of course, it would require an infinite amount of data to construct the entire operator on an unbounded domain, however, for a finite amount of data we can correctly estimate the operator pointwise with bounded error over the entire data set by taking $\beta<0$ to sufficiently force $c_2<0$.  

We note that the results of \cite{diffusion} suggest that Theorem \ref{mainresult} can be extended to manifolds with a compact boundary and that a Neumann boundary condition is implicit to the kernel based approximation.  We do not consider manifolds with boundary here because non-compact manifolds may have non-compact boundaries and the results of \cite{diffusion} strongly rely on the compactness of the boundary.  We also note that Theorem \ref{mainresult} only gives pointwise convergence to the operator $\mathcal{L}_{\alpha,\beta}$ in \eqref{generator_ab} when applied to smooth functions, and this does not imply spectral convergence.  In fact, the spectral convergence of the continuous expectations to the limiting operator for fixed bandwidth kernels was shown in \cite{BNspectral} and the spectral convergence of the discrete operator to the continuous expectations was shown in \cite{von2008consistency}.  In \ref{VBK}, we numerically verify our proof of Theorem~\ref{mainresult} by comparing the pointwise estimates, obtained by evaluating the operator on smooth functions (see Figures~\ref{leftrightcomp} and \ref{symmetricNonUniform}). In the numerical examples in Sections~\ref{numerics} and \ref{circleexample} we will compare the eigenfunctions of the limiting operator to the discrete estimates given by the eigenvectors of the discrete operator.  Our examples show good agreements which suggest that generalizing the spectral convergence results of \cite{BNspectral,von2008consistency} is possible, but extending these results to variable bandwidth kernels is beyond the scope of this paper.


\section{Details of the numerical implementation}\label{algorithm}

Given a data set $\{x_i\}_{i=1}^N \subset \mathbb{R}^n$ sampled independently from a density $q(x)$ on a $d$-dimensional Riemannian manifold $\mathcal{M} \subset \mathbb{R}^n$, the algorithm of this section will produce an $N \times N$ sparse matrix $L^S_{\epsilon,\alpha,\beta}$ in \eqref{errorestimate} which approximates the Kolmogorov operator in \eqref{generator_ab}. For example, assume that the data is generated by Brownian motion on a manifold in a potential $U(x)$, that is,
\begin{align}\label{SDE} dx = -c_1 \nabla U(x)dt + \sqrt{2} dW_t, \end{align}
where $W_t$ is a Brownian motion on the manifold $\mathcal{M}$ and $U:\mathcal{M}\to \mathbb{R}$ is smooth potential.  The invariant measure of this system is given by $q(x) \propto \exp(-c_1 U(x))$ and we will assume that the data are independently sampled from this distribution.  Letting $\Delta$ be the Laplacian (with negative eigenvalues) on $\mathcal{M}$, the generator of the stochastic process \eqref{SDE} is the backward Kolmogorov operator $\mathcal{L}_{\alpha,\beta}$ in \eqref{generator_ab}.  Typically we will be interested in the cases $c_1 = 0$, which approximates the Laplacian (the generator for Brownian motion on $\mathcal{M}$), and $c_1 = 1$ which approximates the generator of the stochastically forced gradient flow in \eqref{SDE}.  

In order to make use of the result in Corollary~\ref{corollary}, we require the bandwidth function $\rho$ to be a power of the sampling density, $\rho = q^{\beta} + \mathcal{O}(\epsilon)$.  While it is possible to use a fixed bandwidth kernel to estimate $q$ up to order-$\epsilon$, the results of \ref{samplingErr1} suggest that we cannot take $\epsilon$ very small unless $N$ is large.  The standard theory of variable bandwidth kernel density estimation \cite{RosenblattFBK,ParzenFBK,ScottVBK,ScottVBK2} offers multiple competitive algorithms, however for simplicity we will use an ad hoc method based on the distance to the nearest neighbors.  To estimate $q$ for the purposes of defining $\rho$, we first define an ad hoc bandwidth function $\rho_0(x_i) = \left(\frac{1}{k_0-1}\sum_{j=2}^{k_0} ||x_i-x_{\textup{I}(i,j)}||^2 \right)^{1/2}$, where $\textup{I}(i,j)$ is the index of the $j$-th nearest neighbor of $x_i$ from the data set (note that we leave out the nearest neighbor $\textup{I}(i,1)$, which is always the point $x_i$ itself).  In the numerical examples in the next section we use $k_0=8$ nearest neighbors and we found that the results are not very sensitive to the choice of $k_0$ (we tested values up to $k_0=64$ with similar results).  We then define $\epsilon_0^{1/2} \equiv \frac{1}{N}\sum_{i=1}^N \rho_0(x_i)$ and  $\tilde\rho_0 \equiv \rho_0/\epsilon_0^{1/2}$, so that $\tilde \rho_0 = \mathcal{O}(1)$ and use a symmetric kernel with bandwidth $\rho_0$ to estimate the density as,
\begin{align}\label{preestimate} q_0(x_i) &\equiv \frac{(2\pi)^{-d/2}}{\rho_0(x_i)^{d} N}\sum_{l=1}^N \exp\left(\frac{-||x_i-x_l||^2}{2\rho_0(x_i)\rho_0(x_l)} \right) \nonumber \\ &= \frac{(2\pi\epsilon_0)^{-d/2}}{\tilde\rho_0(x_i)^{d} N}\sum_{l=1}^N \exp\left(\frac{-||x_i-x_l||^2}{2\epsilon_0\tilde\rho_0(x_i)\tilde\rho_0(x_l)} \right) = q(x_i) + \mathcal{O}\left(\epsilon_0,\frac{\sqrt{q(x_i)}}{N^{1/2}\epsilon_0^{d/4}\tilde\rho_0(x_i)^{d/2}}\right),
\end{align}
where the estimate follows from \eqref{GSeq1} and \eqref{errorEst2} with high probability. We can then use $\rho \equiv q_0^{\beta} = q^{\beta} + \mathcal{O}(\epsilon_0)$ as the bandwidth function in the kernel $K_{\epsilon}^S$ below.  Balancing the two error terms in \eqref{preestimate} we find, $\epsilon_0 = \mathcal{O}\left(N^{-1/(1+d/4)}\right)$.  Notice that $\epsilon_0$ is significantly smaller than $\epsilon$ as required by balancing the error terms in Corollary~\ref{corollary}, so that $\rho = q^{\beta} + \mathcal{O}(\epsilon)$ as required in Corollary~\ref{corollary}.

Using the bandwidth function $\rho$ estimated as above, we now evaluate the kernel $K_{\epsilon}^S$ on all pairs from the data set, and normalize following Theorem \ref{mainresult} to form $L_{\epsilon,\alpha,\beta}^S$ as,
\begin{align}
K^S_{\epsilon}(x_i,x_j) &= \exp\left\{\frac{-||x_i-x_j||^2}{4\epsilon \rho(x_i)\rho(x_j)}\right\} &\hspace{10pt} q^S_{\epsilon}(x_i) &= \sum_{j=1}^N \frac{K_{\epsilon}(x_i,x_j)}{\rho(x_i)^d}   \nonumber \\
K^S_{\epsilon,\alpha}(x_i,x_j) &= \frac{K^S_{\epsilon}(x_i,x_j)}{q^S_{\epsilon}(x_i)^{\alpha}q^S_{\epsilon}(x_j)^{\alpha}} &\hspace{10pt} q^S_{\epsilon,\alpha}(x_i) &= \sum_{j=1}^N K^S_{\epsilon,\alpha}(x_i,x_j) \nonumber \\ 
\hat K^S_{\epsilon,\alpha}(x_i,x_j) &= \frac{K^S_{\epsilon,\alpha}(x_i,x_j)}{q^S_{\epsilon,\alpha}(x_i)} &\hspace{10pt} L^S_{\epsilon,\alpha,\beta}(x_i,x_j) &= \frac{\hat K^S_{\epsilon,\alpha}(x_i,x_j)-\delta_{ij}}{\epsilon\rho(x_i)^2}. \nonumber
\end{align}
Note that the kernel $K^S_{\epsilon,\alpha}$ is symmetric, however, due to the left normalization, the kernel $\hat K^S_{\epsilon,\alpha}$ is not symmetric and the need to normalize by $\rho(x_i)^2$ further degrades the symmetry in $L^S_{\epsilon,\alpha,\beta}$.  Since we are interested in the eigenvalues and eigenvectors of $L^S_{\epsilon,\alpha,\beta}$, we instead construct a symmetric matrix which is given by conjugation of $L^S_{\epsilon,\alpha,\beta}$.  

Let $D_{ii} = q^S_{\epsilon,\alpha}(x_i)$ and $P_{ii} = \rho(x_i)$ be diagonal $N\times N$ matrices and define the symmetric matrix, $K_{ij} = K^S_{\epsilon,\alpha}(x_i,x_j)$.  Let $L_{ij} = L^S_{\epsilon,\alpha,\beta}(x_i,x_j)$ be the desired normalized Laplacian matrix.  Note that,  $L =P^{-2}(D^{-1}K - I)/\epsilon$ and since $P$ and $D$ are diagonal, we can form the conjugation of $L$ by the diagonal matrix $S = PD^{1/2}$ to find,
\[ S LS^{-1} = \frac{1}{\epsilon} PD^{1/2}P^{-2}(D^{-1}K - I) D^{-1/2}P^{-1} = \frac{1}{\epsilon}P^{-1}(D^{-1/2}KD^{-1/2} -I)P^{-1} = \frac{1}{\epsilon}(S^{-1}K S^{-1} - P^{-2}) . \]
So we define the symmetric matrix $\hat L \equiv \frac{1}{\epsilon}(S^{-1}K S^{-1} - P^{-2})$ with entries,
\[ \hat L_{ij} = \frac{1}{\epsilon \rho(x_i)\rho(x_j)}\left( \frac{K_{\epsilon,\alpha}^S(x_i,x_j)}{\sqrt{q_{\epsilon,\alpha}^S(x_i)q_{\epsilon,\alpha}^S(x_j)}} - \delta_{ij} \right). \]
In order to find the eigenvectors of $L$, we first find the eigen-decompostion of $\hat L = \hat U\Lambda \hat U^\top$, and then note that  setting $U = S^{-1}\hat U$ we have,
\[ LU = S^{-1}\hat L S U = S^{-1}\hat L S S^{-1} \hat U = S^{-1}\hat L \hat U = S^{-1}\Lambda \hat U = \Lambda U, \]
since $S$ is diagonal.  Thus, the columns of $U$ are the desired eigenvectors of $L$ with associated eigenvalues given by $\Lambda$.

In order to make a comparison between the true eigenfunctions and the eigenvectors which approximate them, we must make sure to scale the eigenvectors appropriately.  Since the eigenvectors approximate the eigenfunctions evaluated on the data set itself, they are sampled according to the density $q$.  For an eigenvector $\vec\phi = (\phi_1,...,\phi_N)^\top$, where $\phi_i$ approximates an eigenfunction $\phi$ evaluated at $x_i$, $\phi(x_i)$, we can estimate the normalization factor as a Monte-Carlo integral given by,
\[ ||\phi||_{L^2(q)} = \left(\int \phi(x)^2 q(x)dx \right)^{1/2} = \left( \lim_{N\to\infty}\frac{1}{N} \sum_{i=1}^N \phi(x_i)^2 \right)^{1/2}. \]
This implies that we should normalize the vector $\vec\phi$ so that $||\vec\phi||_{\mathbb{R}^N} = \left(\sum_{i=1}^N \phi_i^2\right)^{1/2} = \sqrt{N}$.  A further complication is the possibility of repeated eigenvalues, especially for a symmetric domain and potential.  The numerical approximations to the eigenfunctions which correspond to a repeated eigenvalue can be any orthogonal transformation of the true eigenfunctions with the given eigenvalue.  For visual comparison we compute this orthogonal transformation using the known eigenfunctions and apply it to the numerical eigenfunctions for each repeated eigenvalue.

The matrices formed above will all be $N\times N$ where $N$ is the number of data points.  When $N$ becomes large this quickly leads to large memory requirements.  However, due to the exponential decay in the initial kernel $K_{\epsilon}^S(x_i,x_j)$, we can replace these values by zero when $x_j$ is far away from $x_i$.  In the examples below we will use the typical algorithm of taking the $k$ nearest neighbors of each point $x_i$ and allowing those to be the only nonzero values in the matrix.  Notice that the kernel matrix formed using only the $k$-nearest neighbors may not be symmetric since the nearest neighbor relationship is not reflexive.  If we are forming a kernel matrix $K$ which should be symmetric using only the $k$-nearest neighbors, we always immediately replace $K$ with $(K+K^\top)/2$ which is symmetric and still sparse.  This sparse representation will only use $\mathcal{O}(Nk)$ memory rather than $\mathcal{O}(N^2)$ and will give similar results for $\epsilon$ small enough that the truncated entries are already very close to zero.  Since the initial kernel matrix is sparse, all the remaining matrices are simply multiplication and subtraction of sparse matrices by diagonal matrices and hence all the matrices constructed above will be sparse.  We then use a sparse eigenvalue solver to find the desired number of eigenvectors of $\hat L$, with eigenvalues closest to $\lambda_0 = 0$.

\begin{figure}
\centering
\includegraphics[width=0.4\textwidth]{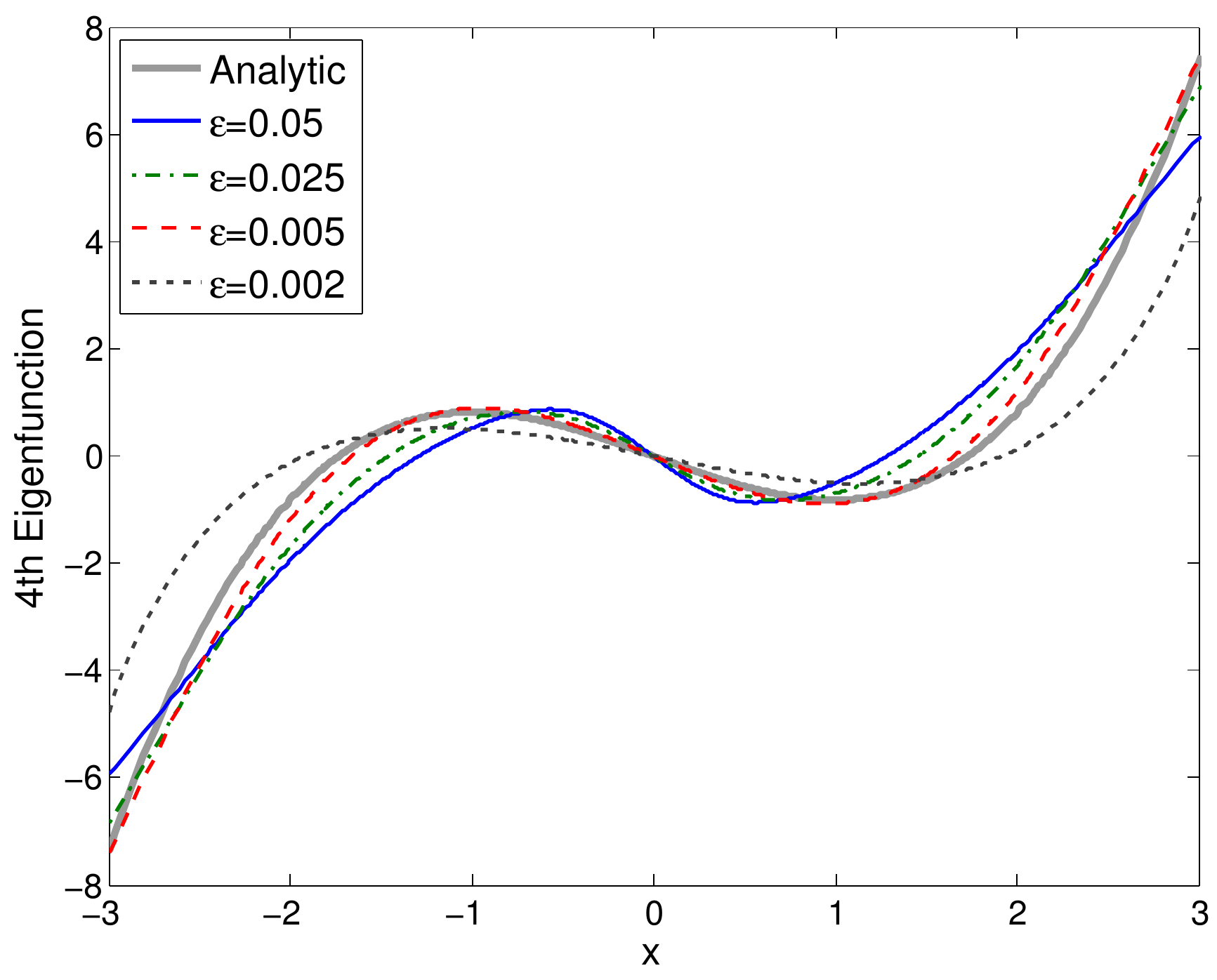}
\includegraphics[width=0.4\textwidth]{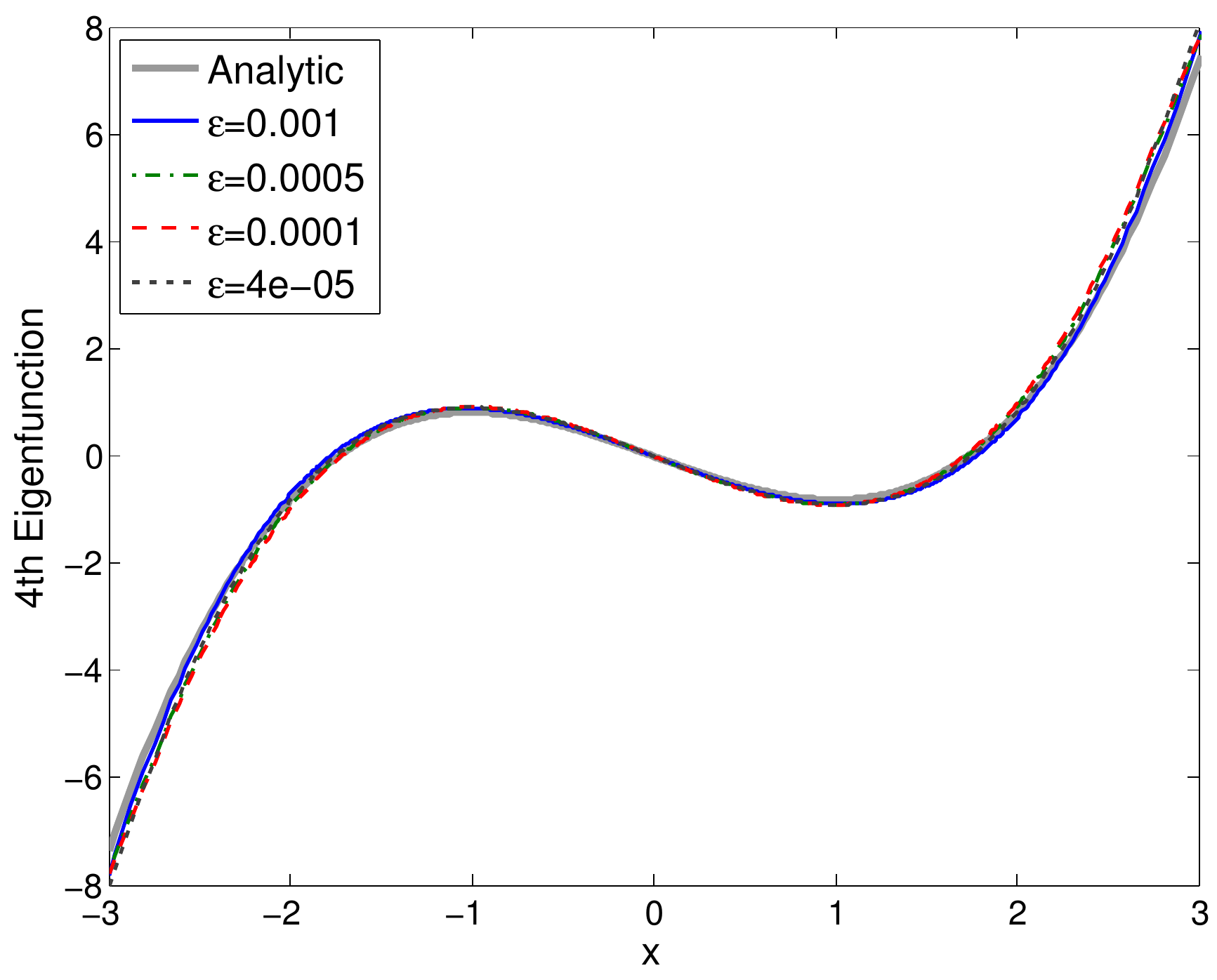}
\caption{\label{sensitivityFig} 
Fixed bandwidth sensitivity (left, $\alpha=1/2, \beta=0$) compared to variable bandwidth sensitivity (right, $\alpha=-1/4,\beta=-1/2$) for the fourth eigenfunction of the generator of the Ornstein-Uhlenbeck process with 2000 data points sampled according to \eqref{nicedist}.  The $\epsilon$ values for the fixed bandwidth kernel are exactly $50$ times the corresponding values for the variable bandwidth kernel due to the difference in scaling.}
\end{figure}

\section{Application to Ornstein-Uhlenbeck processes on non-compact manifolds}\label{numerics}

In this section we show the improvement in operator estimation which is made possible by using variable bandwidth kernels rather than fixed bandwidth kernels on unbounded manifolds. As our first example, we consider the backward Kolmogorov operator for the Ornstein-Uhlenbeck process on the real line.  This process is driven by Brownian motion in a quadratic potential field $U(x) = \frac{1}{2} x^2$ with invariant measure given by a standard normal distribution, $q(x) \propto \exp(-U(x)) = \exp(-x^2/2)$.  

\begin{figure}
\centering
\includegraphics[width=0.4\textwidth]{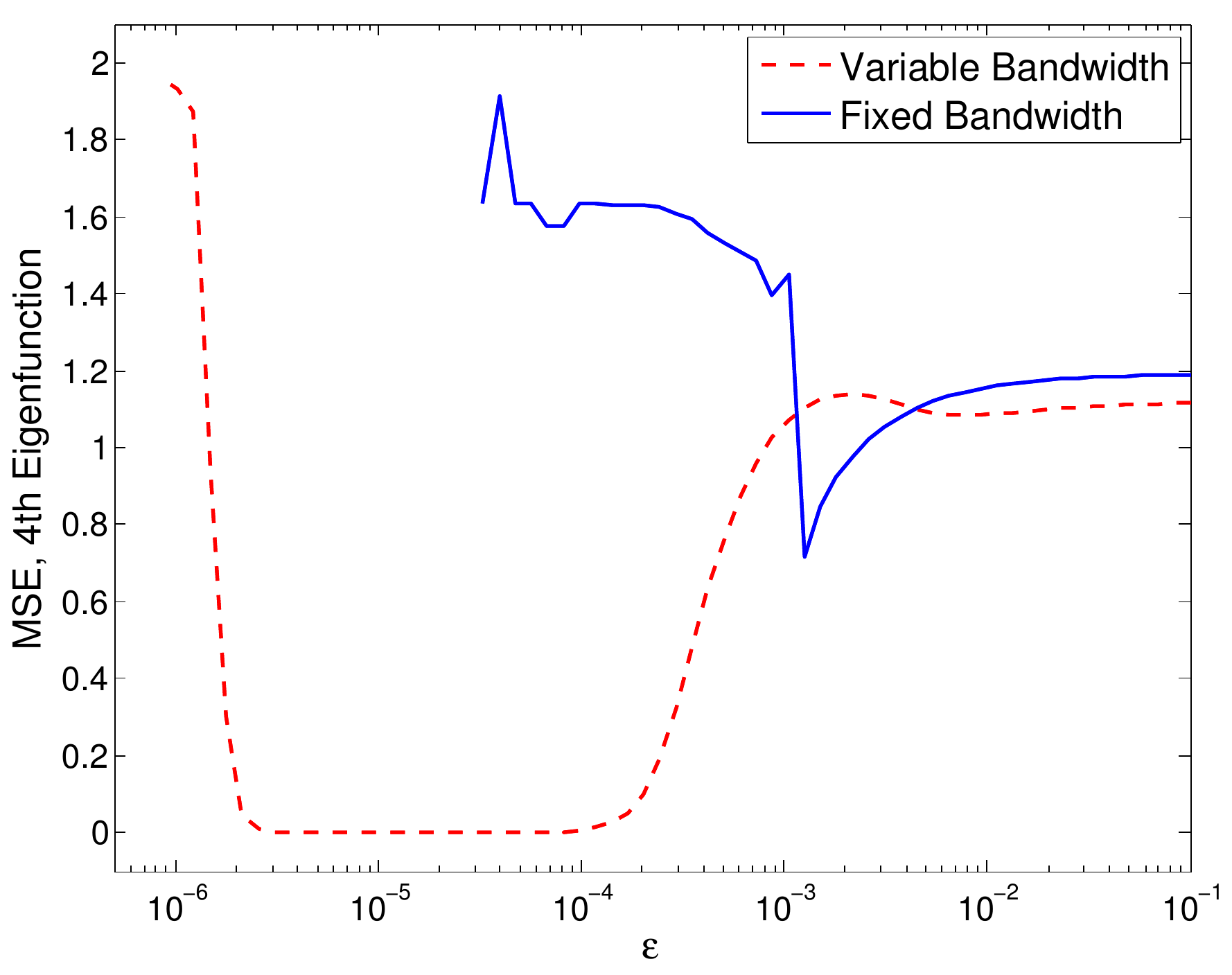}\includegraphics[width=0.4\textwidth]{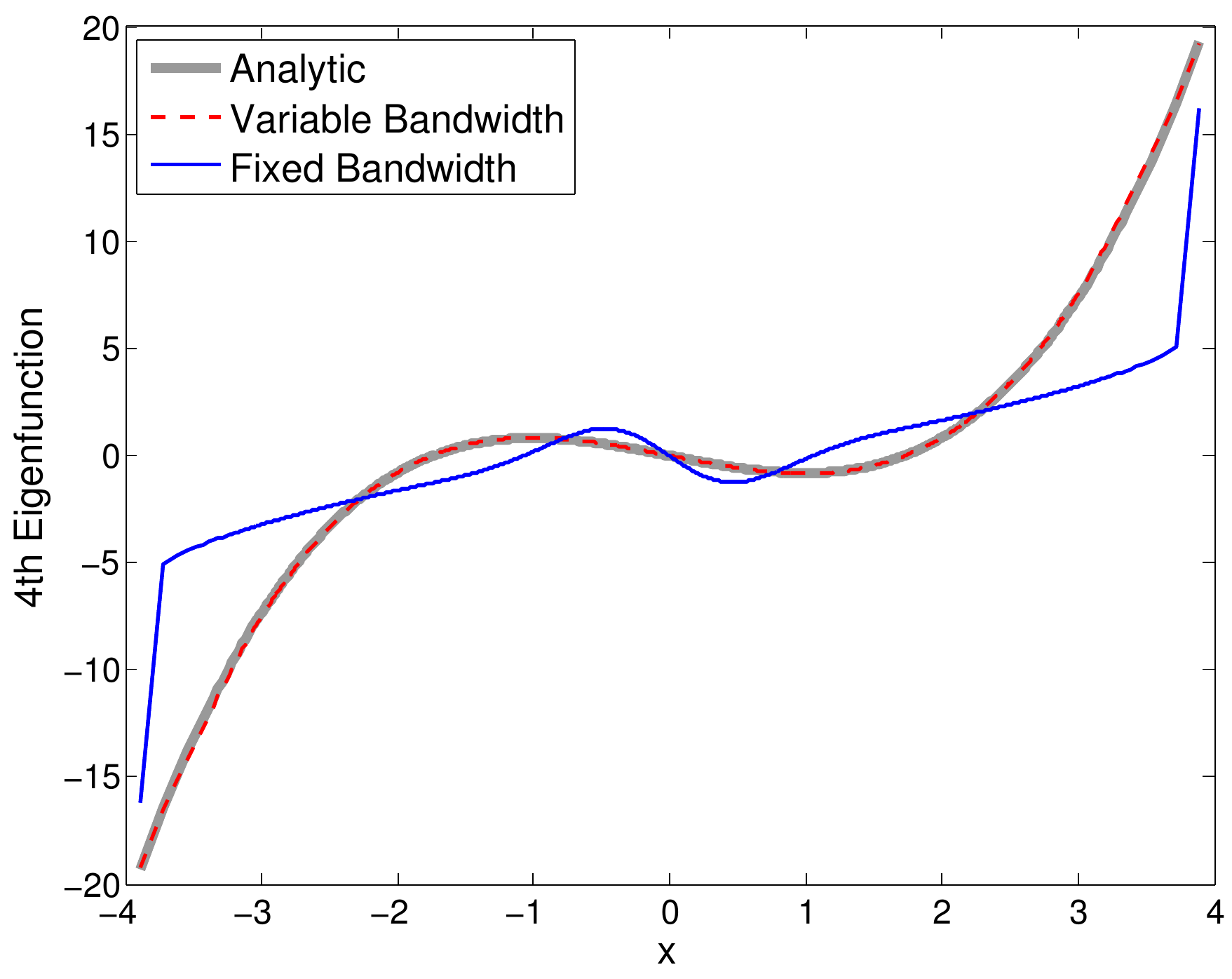}
\caption{\label{sensitivityLargeData} Variable bandwidth sensitivity (red dashed curve, $\alpha=-1/4, \beta=-1/2$) compared to fixed bandwidth sensitivity (blue solid curve, $\alpha=1/2,\beta=0$) for the fourth eigenfunction of the generator of the Ornstein-Uhlenbeck process with 20000 data points sampled according to \eqref{nicedist}.  Left: The mean squared error between the analytic fourth eigenfunction and the kernel based approximations as a function of $\epsilon$.  Right: The analytic fourth eigenfunction is compared to the kernel based approximations with $\epsilon$ chosen to minimize the mean squared error for each kernel.}
\end{figure}

In \ref{dmextension} we show that the continuous theory of \cite{diffusion} easily extends to non-compact manifolds such as $\mathbb{R}$ for functions $f$ which are integrable with respect to the sampling measure $q$. The problem arises when we try to approximate the integral operator $G_{\epsilon}^S$ in \eqref{a1} with a discrete sum evaluated on a random data set.  We first demonstrate that even for a `nice' sample set, the fixed bandwidth kernel can only approximate the operator $\mathcal{L}_{\alpha,0}$ in \eqref{generator_ab} when $N$ is small and $\epsilon$ is carefully tuned.  Rather than sampling randomly from a standard normal distribution, we first generate a `nice' sample set with $N=2000$ points by setting $\delta = (N+1)^{-1}$ and generating a uniform grid $\{\tilde x_i = \delta i\}_{i=1}^{N}$. We then apply the inverse of the cumulative distribution function to the uniform grid so that, 
\begin{align} \label{nicedist} x_i = \sqrt{2} \textup{erf}^{-1}(2 \tilde x_i -1), \end{align}
and the histogram or kernel density estimate of the $\{x_i\}$ will converge to the standard normal density.  The advantage of the `nice' data set is that it is not randomly sampled so there is no possibility of outliers or finite sampling deviations from the true density which are present in any random sample set.  In Figure \ref{sensitivityFig}, we compare the fourth eigenvector of the operator $L_{\epsilon,\alpha,\beta}^S$ with the analytic eigenfunction $H_3(x) = \frac{1}{\sqrt{6}}(x^3-3 x)$ of the generator of the Ornstein-Uhlenbeck process. Notice that for a careful selection of $\epsilon$ the eigenvector approximation from the fixed bandwidth kernel ($\alpha = 1/2,\beta =0$) agrees with the analytic eigenfunction.  In contrast, the variable bandwidth kernel gives a much better approximation over a larger range of $\epsilon$ values.  This demonstrates the first valuable aspect of variable bandwidth kernels, which is that they reduce the sensitivity to $\epsilon$.  Moreover, when the true eigenfunction is unknown, it is difficult to determine which value of $\epsilon$ in the fixed bandwidth kernel is giving the best approximation.  The fact that the variable bandwidth kernel gives a stable result, which \emph{persists} across a large range of bandwidth choices, suggests that it may be possible to choose the bandwidth automatically based on stationarity of the solution.

In the limit of large data, the difficulty in applying the fixed bandwidth kernel becomes more severe.  In Figure \ref{sensitivityLargeData} we compare the variable bandwidth and fixed bandwidth approximations of the fourth eigenfunction for $N=20000$ data points sampled according to \eqref{nicedist}.  In this case, there is no value of $\epsilon$ which gives reasonable results for the fixed bandwidth kernel, as shown in the left panel where we show the RMSE as a function of $\epsilon$ for a large range of $\epsilon$ values, and in the right panel where we plot the eigenfunction which minimized the RMSE.  On the other hand, the variable bandwidth kernel once again returns excellent and stable results over a wide range of values of $\epsilon$.  While it may seem counterintuitive that the fixed bandwidth kernel performs worse with more data, the error bound in Corollary~\ref{corollary} suggests exactly this effect.  Since $c_2>0$ for the fixed bandwidth kernel, the third error bound in Theorem \ref{mainresult} diverges as the sampling density $q$ approaches zero.  When the data set is small, it is unlikely that many of the samples are in areas of small sampling density.  As the size of the data set is increased, the data set contains more points in areas of small sampling density and the minimum value of $q(x_i)$ decreases, causing the error bound to diverge. 

\begin{figure}
\centering
\includegraphics[width=0.4\textwidth]{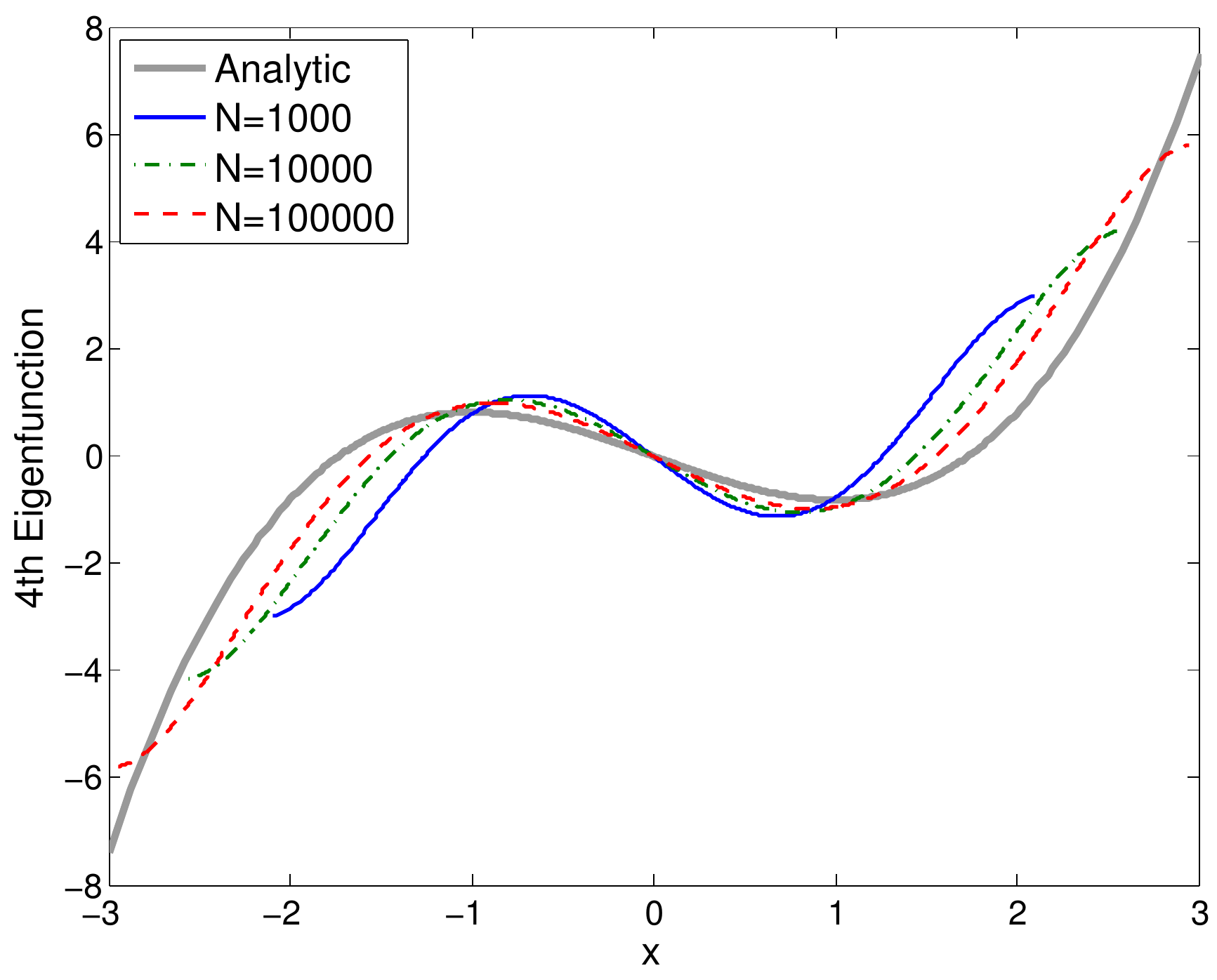}\includegraphics[width=0.4\textwidth]{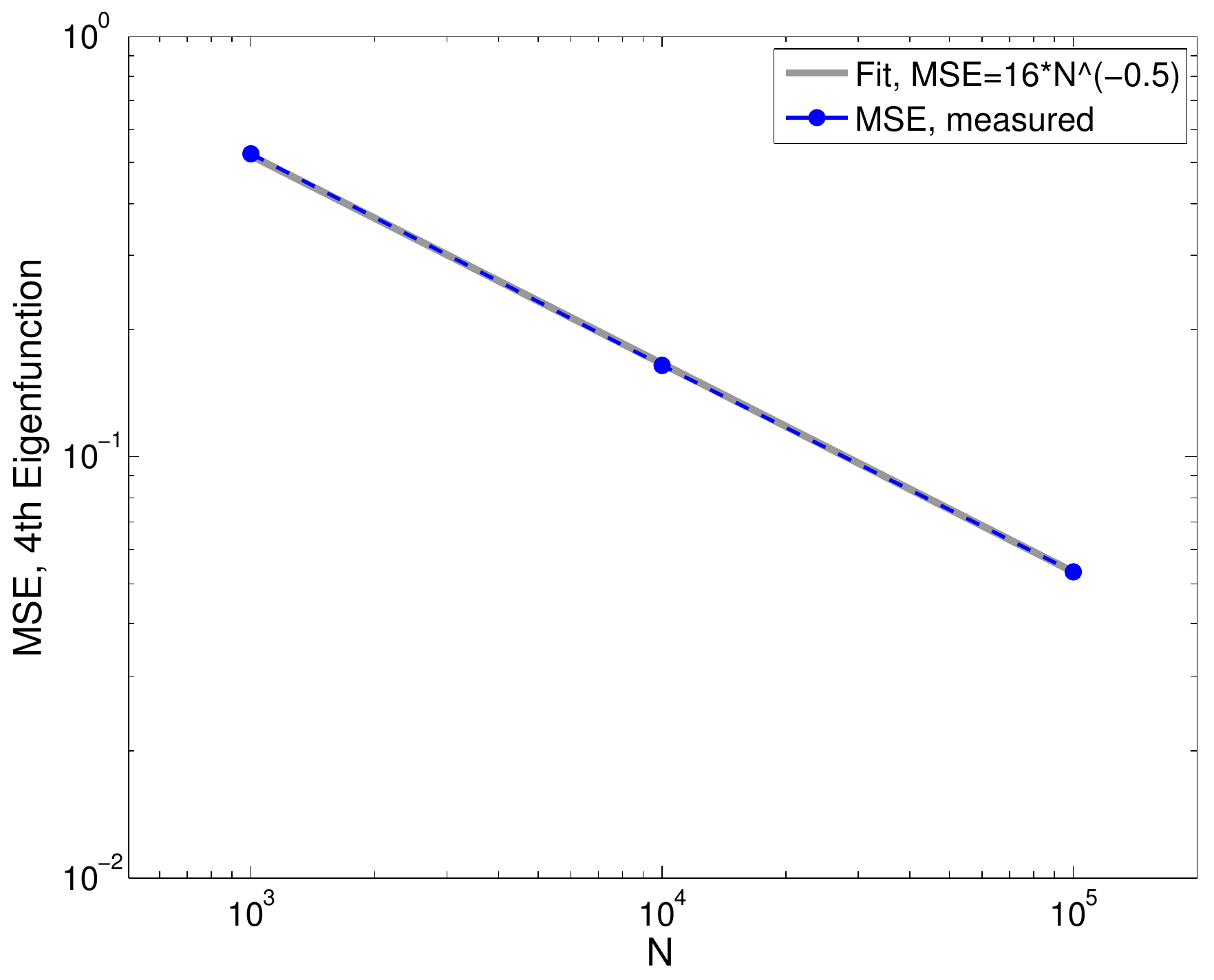}
\caption{\label{outlierTest} Left: Comparison of the analytic fourth eigenfunction of generator for the Ornstein-Uhlenbeck process with the discrete approximations produced with the fixed bandwidth kernel ($\alpha=1/2,\beta=0$) with various amounts $N$ of data where $\sqrt{N}$ outlier points are removed prior to the analysis.  Right: Mean squared error (MSE) over data points $-2\leq x \leq 2$ for $N=1000, 10000$, and $100000$ along with power law fit.  If the power law persists, achieving the MSE of $0.002$ of the variable bandwidth kernel which used only 1000 data points, would require approximately $6\times 10^{7}$ data points for the fixed bandwidth kernel.}
\end{figure}

One may hope to solve the issue of the error bound increasing in areas of small sampling by simply removing the outliers.  Indeed such a procedure was suggested in \cite{singerOutlierRemoval} and while this can improve results for small data sets, it has many unintended consequences.  The artificial modification of the data set implicitly creates an absorbing boundary condition at a \emph{virtual boundary} between the remaining data and the removed outliers.  This virtual boundary significantly affects the spectral properties of the discrete operator $L_{\epsilon}^S$ when $\beta = 0$.  The structural error in the operator approximation due to the virtual boundary leads to extremely poor convergence properties for the eigenvalues and eigenfunctions of $L^S_{\epsilon}$ for $\beta = 0$.  In Figure \ref{outlierTest} we show that even with $N=100000$ data points distributed `nicely' according to \eqref{nicedist}, removing only 316 outlier points (those of smallest probability) has a significant effect on the eigenfunction approximation.  This strategy may still have problems as $N$ increases, but even assuming that the trend in Figure \ref{outlierTest} persists for large $N$, achieving the mean squared error of a variable bandwidth kernel applied to $1000$ data points would require approximately $6\times 10^{7}$ data points with a fixed bandwidth kernel using the outlier removal strategy.

\begin{figure}
\centering
\includegraphics[width=0.4\textwidth]{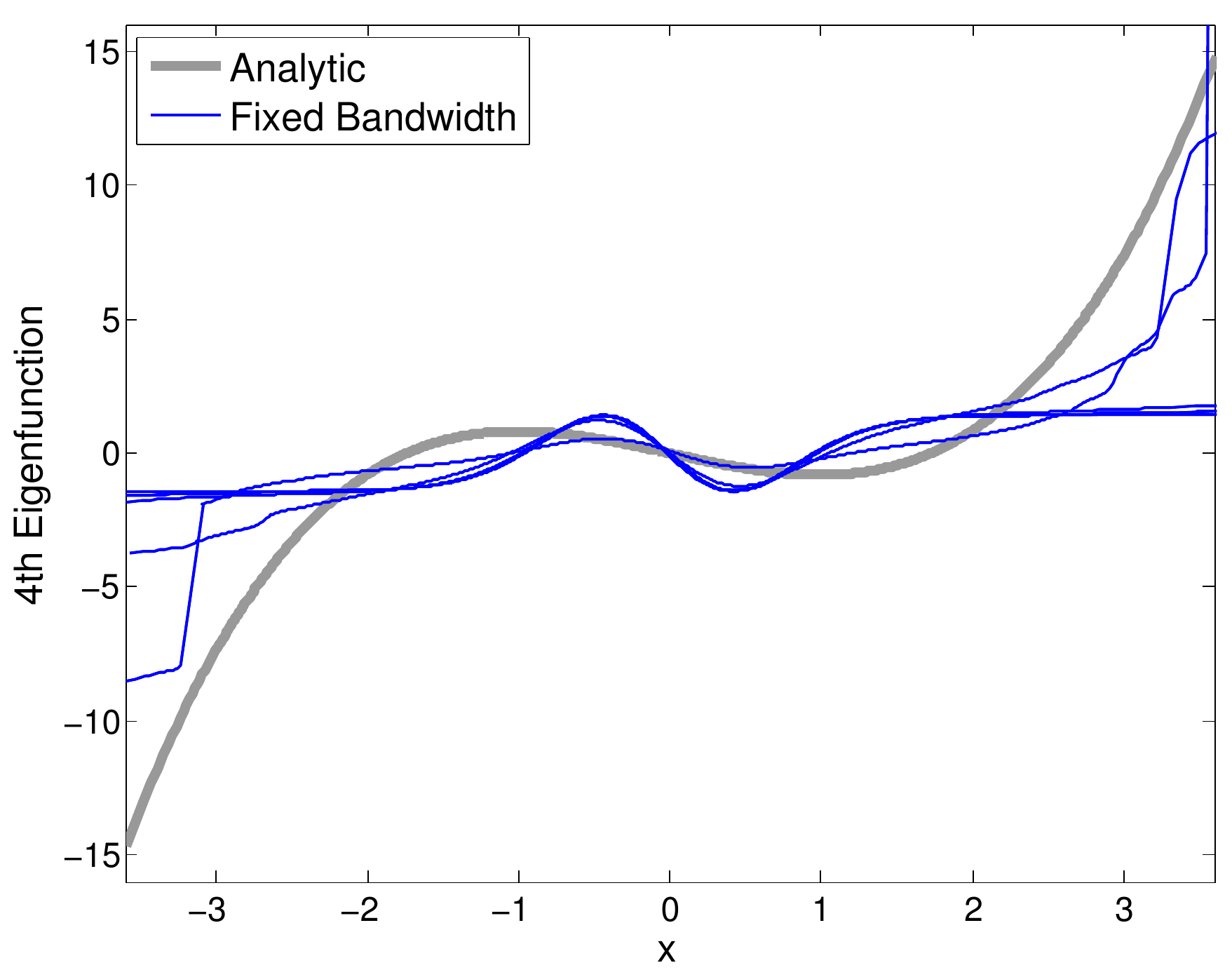}
\includegraphics[width=0.4\textwidth]{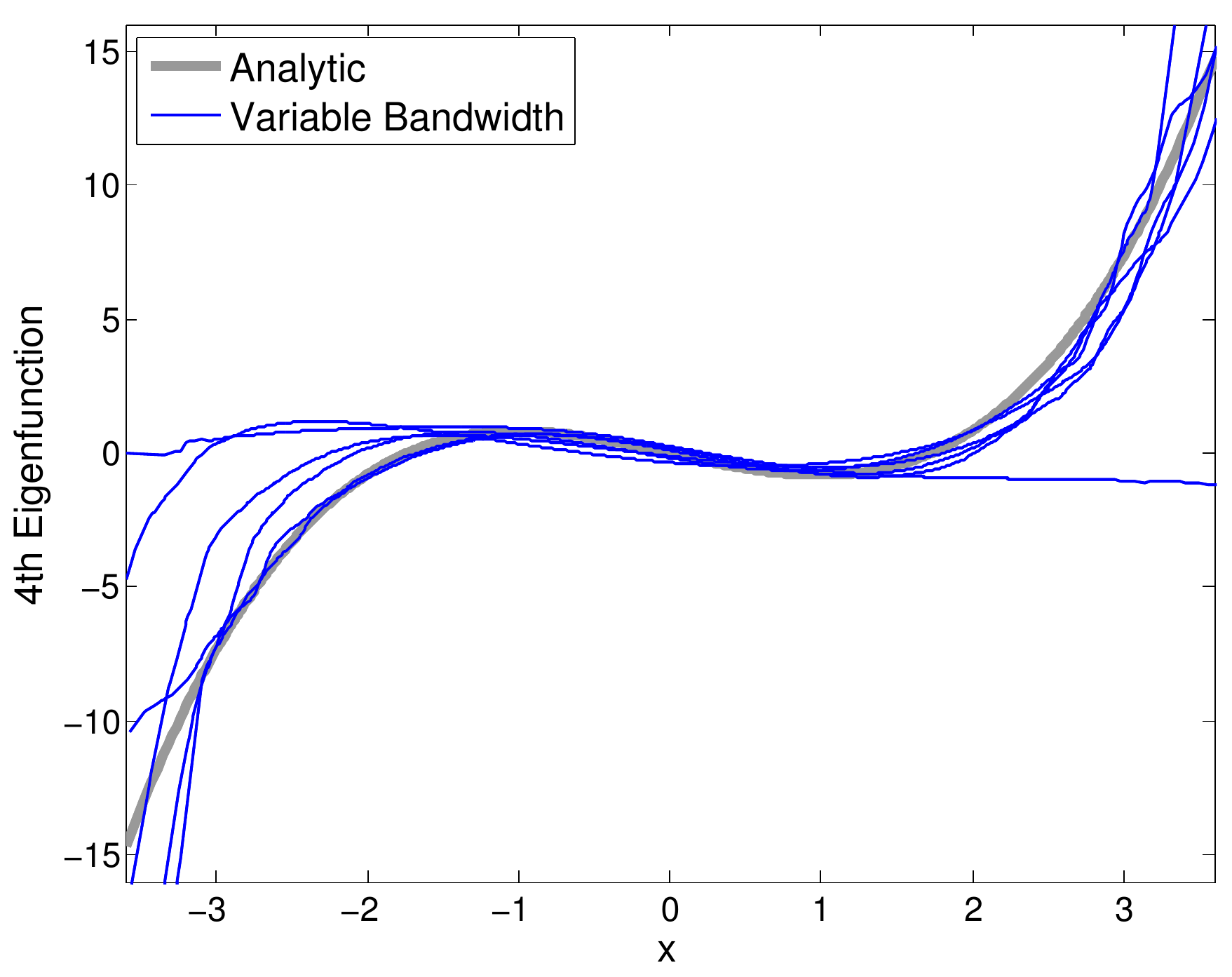}
\caption{\label{randomsamples} Comparison of fixed bandwidth (left) and variable bandwidth (right) kernels for estimating the fourth eigenfunction of the generator for the Ornstein-Uhlenbeck process using 10 randomly generated data sets with $N=20000$ points each. Left: For each of the 10 data sets, the fixed bandwidth kernel ($\alpha=1/2, \beta=0$) is used to approximate the fourth eigenfunction using 65 separate values of $\epsilon$ ranging from $10^{-5}$ to $1$ and $\epsilon$ is chosen to minimize the mean squared error between the analytic fourth eigenfunction and the kernel based approximation.  Right: The same analysis performed on the same 10 data sets with the variable bandwidth kernel ($\alpha=-1/4, \beta=-1/2$). }
\end{figure}

We now consider the case of randomly sampled data, where $x_i$ are independently sampled from a standard normal distribution, and demonstrate that fixed bandwidth kernels have even more significant limitations in this context.  In Figure \ref{randomsamples} we compare the fixed and variable bandwidth approximations with 10 randomly generated data sets of length $N=20000$.  Notice that none of the 10 fixed bandwidth approximations shown in Figure \ref{randomsamples} agree well with the correct eigenfunction.  In contrast, the variable bandwidth kernel shows significant improvement with the increase in data in accordance with Theorem \ref{mainresult}.  Note that one of the data sets resulted in a particularly poor approximation even with the variable bandwidth kernel, but this does not contradict Theorem \ref{mainresult} since the error bounds are pointwise and only obtained with high probability.  In other words, for fixed $N$ a particularly bad random sample can lead to poor estimates even with the variable bandwidth.

In our second example, we tested the variable bandwidth kernel on a two-dimensional plane (an unbounded manifold) by applying the method to data sampled independently from the invariant measure of a two-dimensional Ornstein-Uhlenbeck process.  This process is given by Brownian motion in the potential well $U(x,y) = (x^2 + y^2)/2$ and the invariant measure is a two-dimensional Gaussian distribution with covariance matrix equal to the identity matrix.  In Figure \ref{randomsamples2D} we show the analytic 4th eigenfunction of the generator of this process which is $\phi_4(x,y)=xy$ along with the estimates of this eigenfunction produced by the variable bandwidth ($\alpha=-1/2, \beta=-1/2$) and fixed bandwidth ($\alpha=1/2,\beta=0$) kernels. While the estimate with variable bandwidth kernel looks reasonably closed to the analytic eigenfunction, the estimates with fixed bandwidth kernel are not accurate at all despite empirical tuning of $\epsilon$ over a wide range of values.

\begin{figure}
\centering
\includegraphics[width=0.3\textwidth]{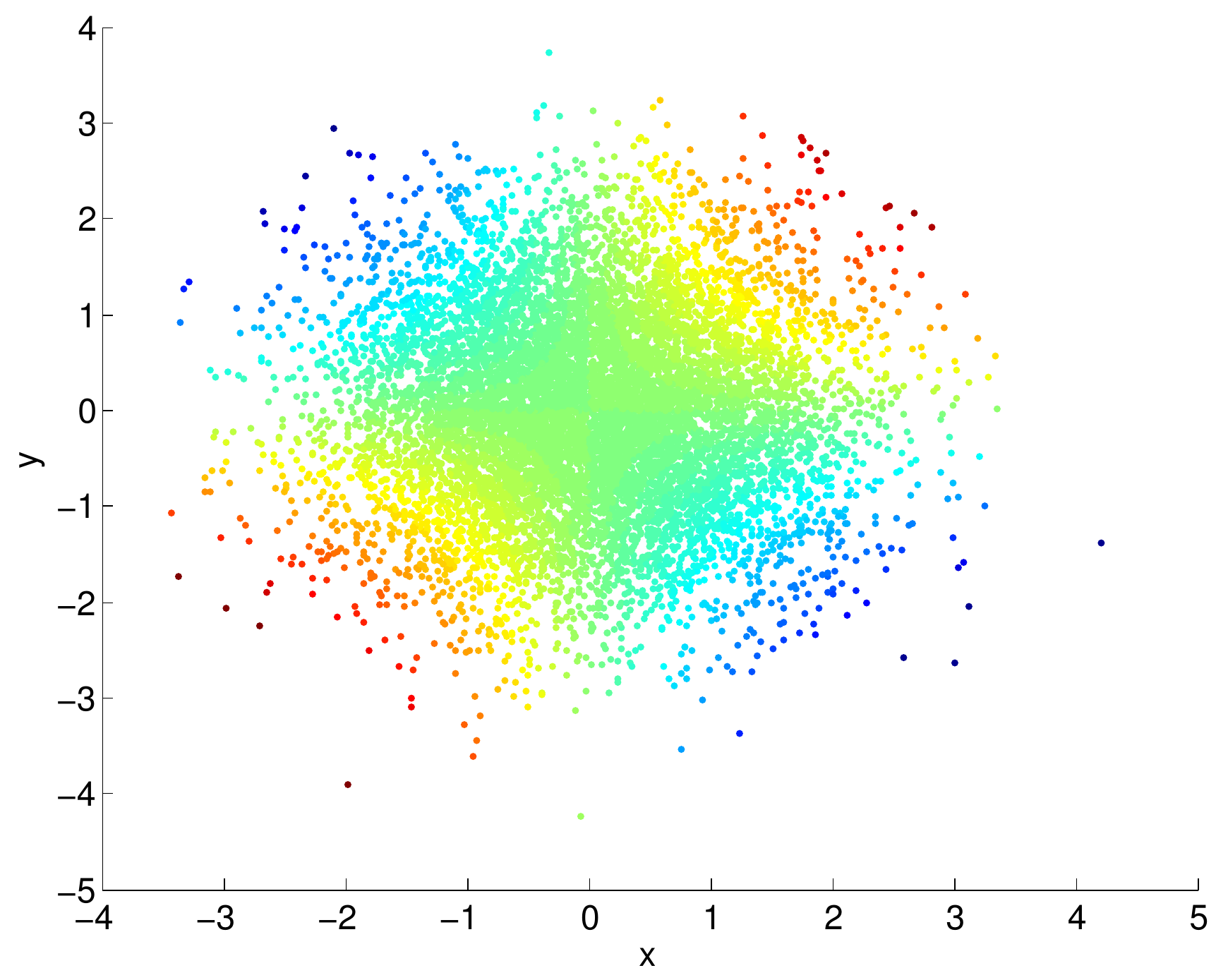}
\includegraphics[width=0.3\textwidth]{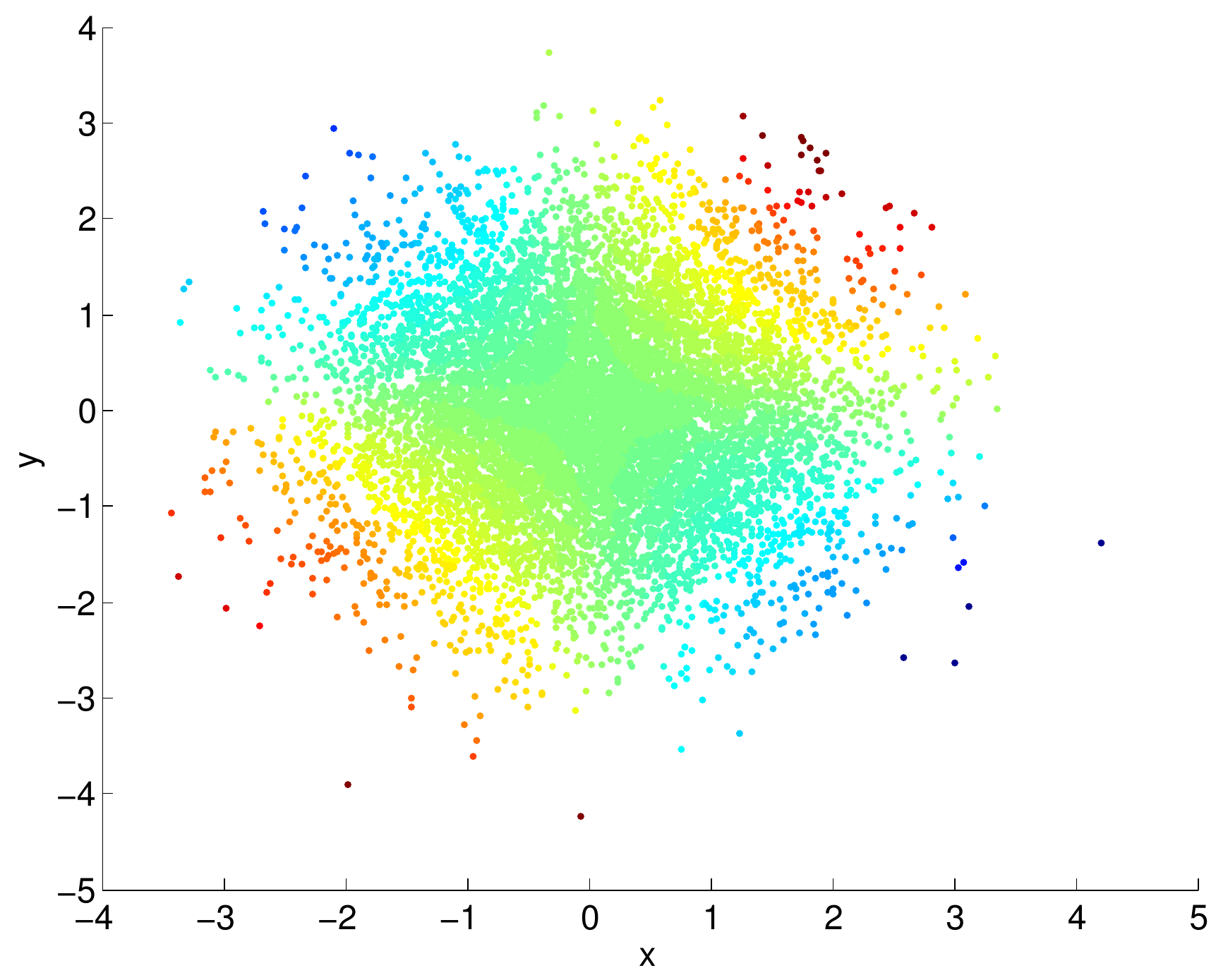}
\includegraphics[width=0.3\textwidth]{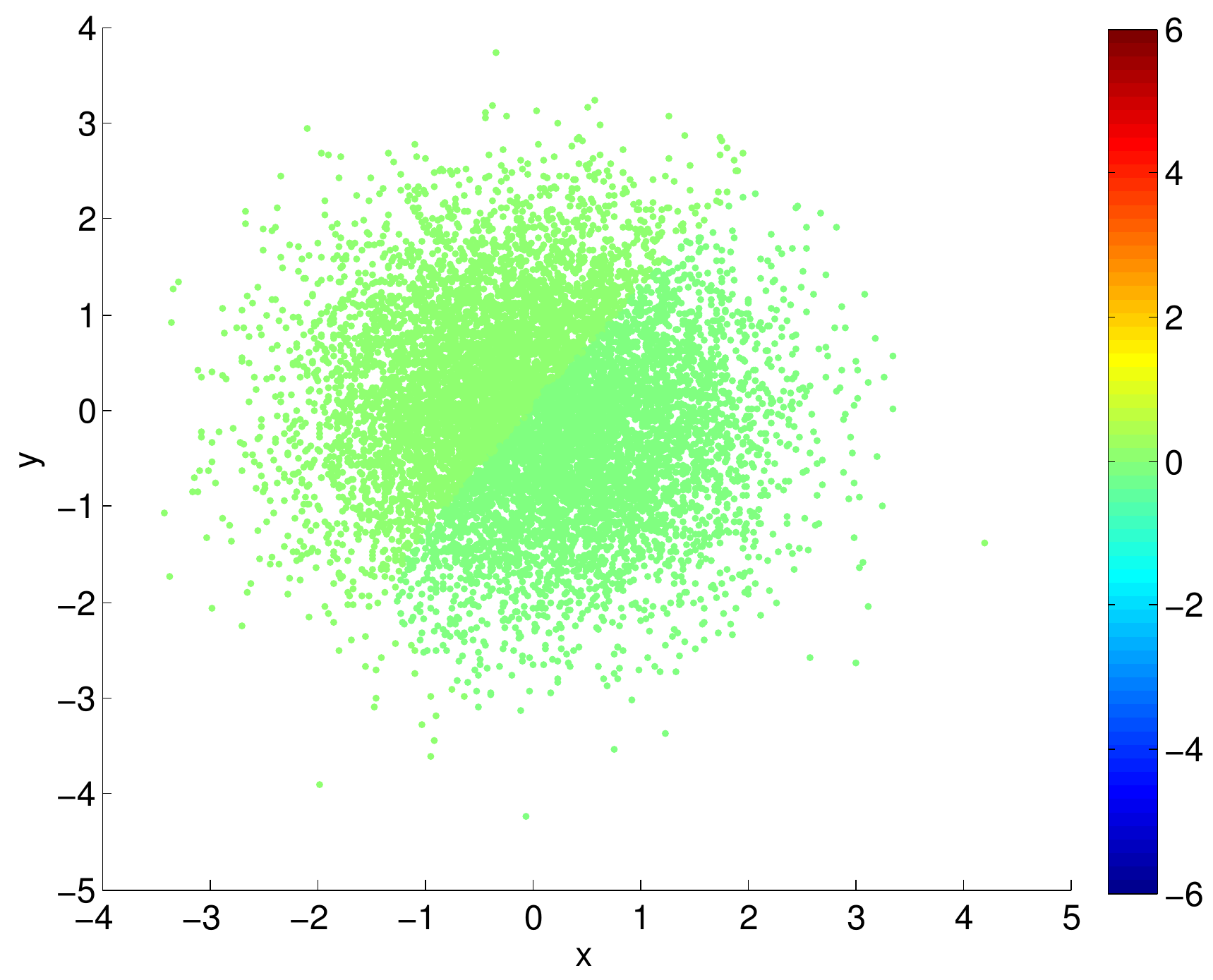}
\caption{\label{randomsamples2D} Comparison of the analytic (left) fourth eigenfunction of the 2-dimensional Ornstein-Uhlenbeck generator $\phi_4(x,y) = xy$ with the variable bandwidth (middle) kernel estimate and the fixed bandwidth (right) estimate.  The data set consist of 10000 points drawn independently from the invariant measure of the process, which is simply a standard Gaussian distribution in $\mathbb{R}^2$ with covariance structure given by the identity matrix.  For both kernels $\epsilon$ was chosen to minimize the squared error between the estimated and the analytic eigenfunctions.}
\end{figure}


\section{Application to non-uniformly sampled compact manifolds}\label{circleexample}

In this section we demonstrate that even on a compact manifold, when the sampling density $q$ is non-uniform, the variable bandwidth kernels still have many advantages.   As a first example, we consider a unit circle parameterized by $\theta \in [0,2\pi)$ with sampling density $q(\theta) = \frac{1}{4\pi}(2+\cos(\theta))$.  We can produce a grid of $N=1500$ points which have this distribution by numerically inverting the cumulative distribution function $F(\theta) = \frac{1}{4\pi}(2\theta + \sin(\theta))$ and applying it to a uniform grid $\{t_i = \delta i\}_{i=1}^{1500}$ where $\delta=1/1501$ so that $\theta_i = F^{-1}(t_i)$.  We then produce a data set in $\mathbb{R}^2$ via the standard embedding $x_i = (\cos(\theta_i),\sin(\theta_i))^\top$.  Since the manifold is compact and the sampling density is bounded away from zero, we can approximate the Laplacian on the circle $\Delta = \frac{d^2}{d\theta^2}$ from the data set $\{x_i\}$ with either the fixed bandwidth kernel ($\alpha=1,\beta=0$) or the variable bandwidth kernel ($\alpha=1/4,\beta=-1/2$).  However, each algorithm requires tuning the nuisance parameter $\epsilon$, which can be very difficult for large data sets where the computational times become restrictive.  In Figure \ref{circlesensitivity} we show that the variable bandwidth gives better results over a much larger range of $\epsilon$ values, which is an important consideration for practical applications.  We then check the sensitivity to randomness in the data set by perturbing each $\theta_i$ by a uniform random variable from $[0,0.5]$ and reducing modulo $2\pi$.  The analysis is repeated for the randomly distributed data set in Figure \ref{circlesensitivity} and again the variable bandwidth kernel yields a better approximation with reduced sensitivity on the bandwidth $\epsilon$ near the optimal choice.  Of course, the weakness of the variable bandwidth kernel is that it requires knowledge of the intrinsic dimension $d$, which can be costly and difficult to estimate especially for noisy data.  These results suggest that it may be worth the cost of estimating the dimension in order to recover a significantly improved approximation.

We also note that the variable bandwidth kernel seems to work well with an automated $\epsilon$-tuning method introduced in \cite{coifman2008TuningEpsilon}.  It was noted in \cite{coifman2008TuningEpsilon} that for a fixed bandwidth kernel in \eqref{DMkernel},
\comment{$K_{\epsilon}(x_i,x_j) = \exp(-||x_i-x_j||^2/(4\epsilon))$,}in the limit as $\epsilon \to 0$ the kernel approaches $0$, and in the limit as $\epsilon\to \infty$ the kernel approaches $1$ for all pairs of data points.  Moreover, they note that when $\epsilon$ is well tuned, the kernel localizes the data set so that,
\begin{align}\label{tuningfunction} S(\epsilon) = \frac{1}{N^2}\sum_{i,j} K_{\epsilon}(x_i,x_j) \approx \frac{1}{{\rm vol}(\mathcal{M})^2}\int_{\mathcal{M}}\int_{T_{x}\mathcal{M}} K_{\epsilon}(x,y) \, dy \, dV(x) \approx \int_{\mathcal{M}}  \frac{(4\pi\epsilon)^{d/2}}{{\rm vol}(\mathcal{M})^2} \,  dV(x) = \frac{(4\pi\epsilon)^{d/2}}{{\rm vol}(\mathcal{M})}, \end{align}
where $d$ is the dimension of the manifold and $T_{x}\mathcal{M} \cong \mathbb{R}^d$ is the tangent space at $x$.  Numerically, when $\epsilon$ is very large, $S(\epsilon)$ will approach $1$ and when $\epsilon$ is very small, $S(\epsilon)$ will approach $1/N$.  In \cite{coifman2008TuningEpsilon} they suggested that we can choose $\epsilon$ by evaluating $S(\epsilon)$ for a large range of values and searching for the region where $\log(S(\epsilon))$ grows linearly with respect to $\log(\epsilon)$.  In other words, for a good choice of $\epsilon$, the function $S(\epsilon)$ should be locally well approximated by a power law $S(\epsilon) \propto \epsilon^a$ where,
\[ a = \frac{d(\log S)}{d(\log \epsilon)} \approx \frac{\log(S(\epsilon+h))-\log(S(\epsilon))}{\log(\epsilon+h)-\log(\epsilon)} \]
is approximately the local slope near $\epsilon$ in a plot of $\log(S(\epsilon))$ versus $\log(\epsilon)$.
 A natural extension of this strategy is to evaluate $S(\epsilon)$ for $\epsilon_i = 2^i$ over a range of values such as $i=-30,-29,...,9,10$ and to maximize the slope $a_i = \frac{\log(S(\epsilon_{i+1}))-\log(S(\epsilon_i))}{\log(\epsilon_{i+1})-\log(\epsilon_i)}$.  We applied this procedure to both the fixed bandwidth kernel and the variable bandwidth kernel in Figure \ref{circlesensitivity} and found that both slopes were maximized for $\epsilon = 2^{-5} \approx 0.031$ and this value is highlighted in filled circles in the middle panel of Figure \ref{circlesensitivity}.  It is interesting to note that in the region of linearity, $\log(S(\epsilon)) = \frac{d}{2}\log(\epsilon) + \frac{d}{2}\log(4\pi)$ and the maximal value of the slope for both kernels in Figure \ref{circlesensitivity} is in fact $d/2 = 1/2$.  This suggests that it may be possible to estimate the dimension of the manifold as the maximum value of the slope, one possibility would be to used a fixed bandwidth kernel to estimate the dimension and then to use this estimate for the variable bandwidth kernel.  A more exhaustive comparison of this approach to other methods of estimating the dimension, such as \cite{HeinDimension,MaggioniDimension}, especially in the presence of noisy samples which do not lie exactly on the manifold, is beyond the scope of this paper.

\begin{figure}[h]
\centering
\includegraphics[width=0.333\textwidth]{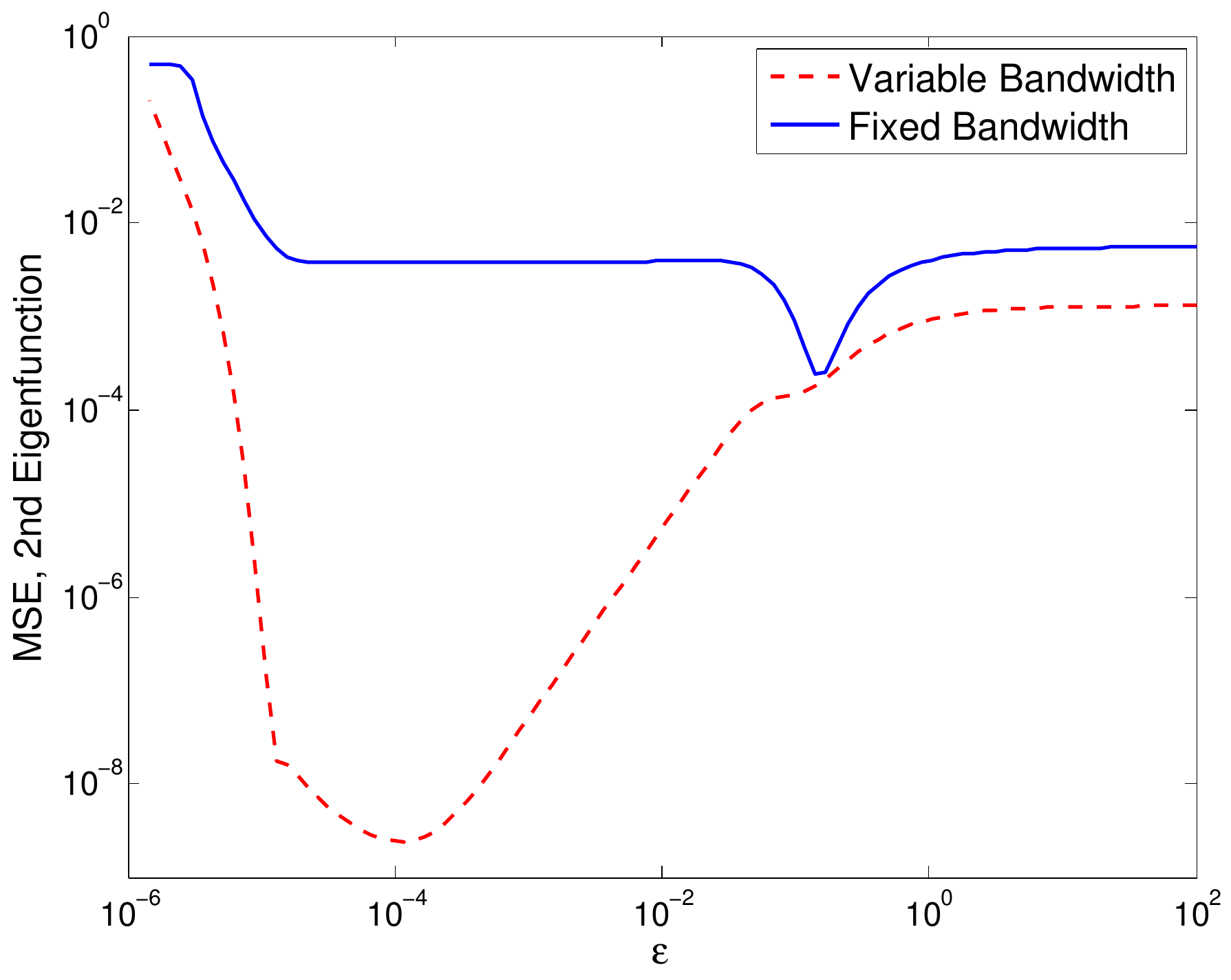}\includegraphics[width=0.33\textwidth]{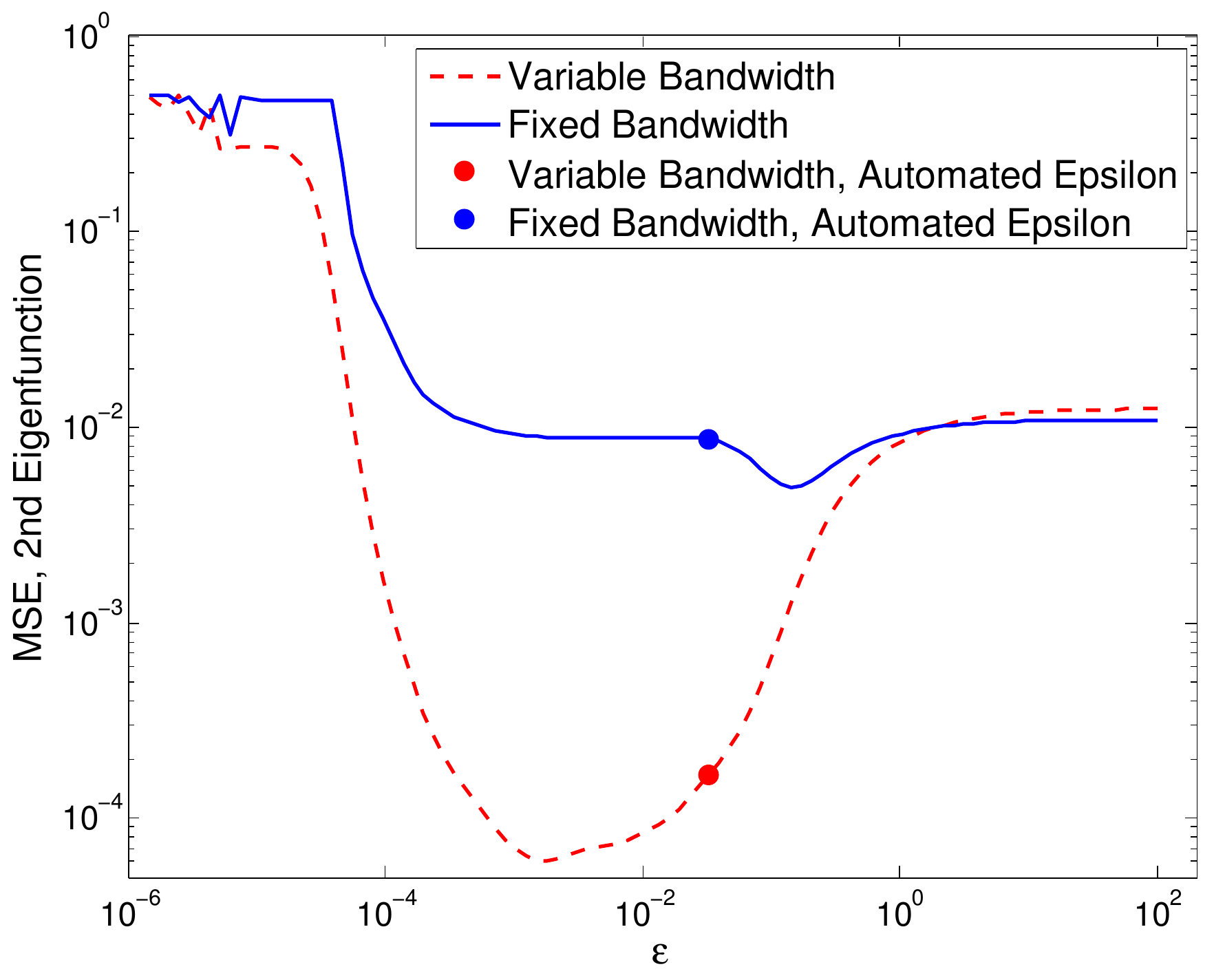}\includegraphics[width=0.333\textwidth]{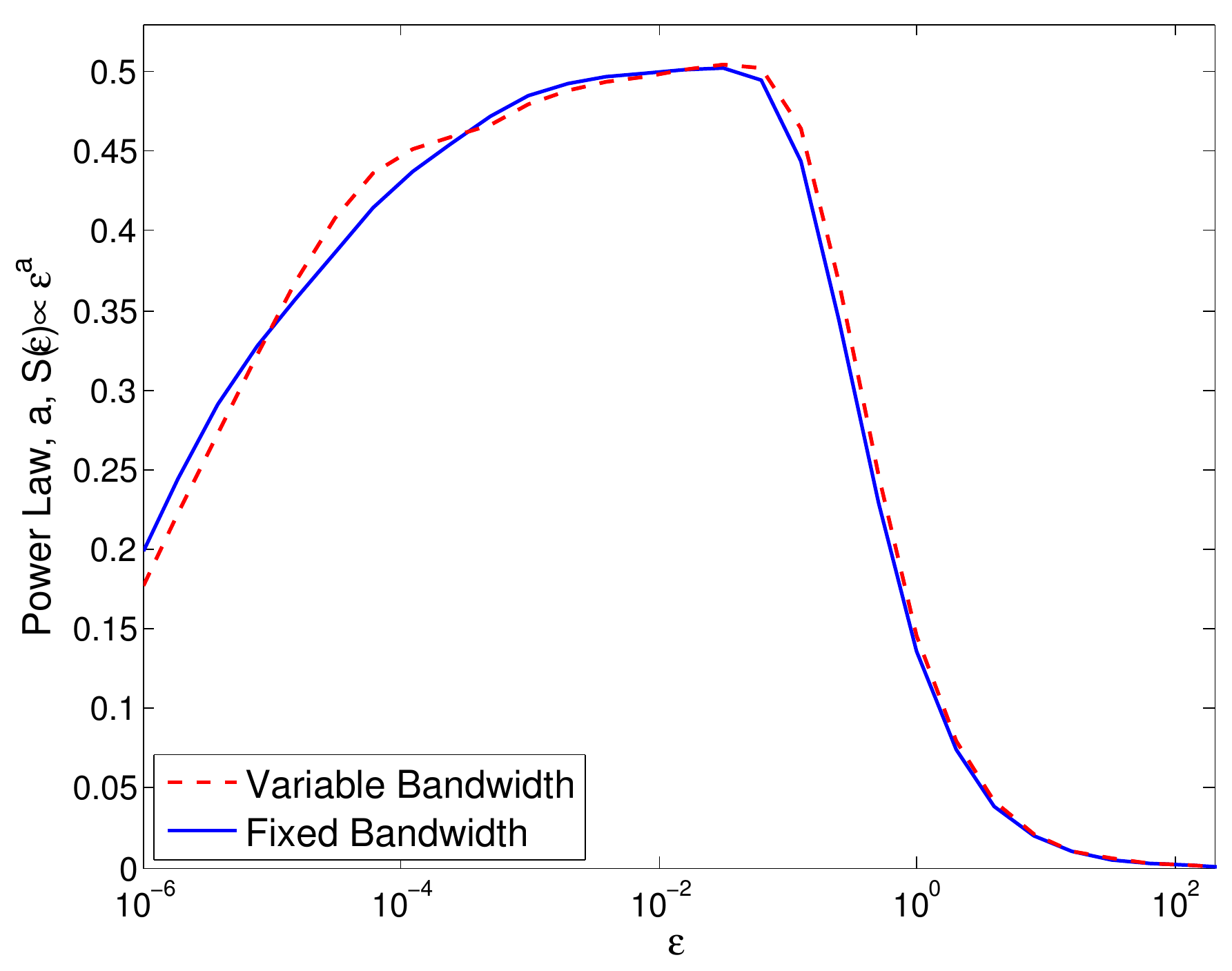}
\caption{\label{circlesensitivity} Sensitivity of the approximation of the second eigenfunction ($\sin(\theta)$) of the Laplacian on a unit circle sampled according to the density $q(\theta) = \frac{1}{4\pi}(2+\cos(\theta))$ with a variable bandwidth kernel (red dashed curves) and a fixed bandwidth kernel (blue solid curves).  Left: Mean squared error between the analytic eigenfunction and the respective approximations for the data set $\{x_i = (\cos(\theta_i),\sin(\theta_i))^\top\}$ where $\theta_i = F^{-1}(i/1501)$ for $i=1,...,1500$ and $F(\theta)= \frac{1}{4\pi}(2\theta + \sin(\theta))$ is the cumulative distribution function for $q$.  Middle: Same plot on a randomized data set where each $\theta_i$ was perturbed by a uniform random variable in $[0,0.5]$ ($\theta_i$ is then reduced modulo $2\pi$). We also show $\epsilon$ resulting from the automated choice of $\epsilon$ using the method of \cite{coifman2008TuningEpsilon} (see the filled circles).  Right: As a function of $\epsilon$, we plot the exponent, $a$, from the local power law in the tuning function $S(\epsilon) \propto \epsilon^a$ estimated as the slope $\frac{\log(S(\epsilon_{i+1}))-\log(S(\epsilon_i))}{\log(\epsilon_{i+1})-\log(\epsilon_i)}$.  Notice that the exponent has a maximum near $1/2=d/2$ as suggested by the theory of \cite{coifman2008TuningEpsilon}. }
\end{figure}

As a final example, consider a unit sphere in $\mathbb{R}^3$, for which the first three nontrivial eigenfunctions of the Laplacian are proportional to the standard coordinate functions, namely $\phi_1 \propto x, \phi_2\propto y$, and $\phi_3\propto z$.  To generate a non-uniform sampling on the sphere we generated 3000 random points in $\mathbb{R}^3$ from a Gaussian distribution with a randomly chosen $3\times 3$ covariance matrix and then projected these points onto the unit sphere (see Figure \ref{spheresensitivity}, top left panel).  We then applied both the variable bandwidth kernel (with $\alpha=0$ and $\beta=-1/2$) and the fixed bandwidth kernel (with $\alpha=1$ and $\beta=0$) for a large range of values of $\epsilon$ and found the mean squared error between the estimated eigenfunctions and the true analytic eigenfunctions.  Of course, due to the symmetry of the sphere, the first three nontrivial eigenfunctions have the same eigenvalue and any orthogonal linear transformation of the three eigenfunctions is also a valid basis for the same eigen-space.  To make a valid comparison it is necessary to find the best orthogonal linear transformation between the estimated eigenfunctions and the true eigenfunctions.  For a better visual comparison, we will simply find the best linear transformation between the estimated eigenfunctions and the original coordinates (since the original coordinates are proportional to the analytic eigenfunctions).  We also applied the automatic tuning method of \cite{coifman2008TuningEpsilon} and chose the value of $\epsilon$ that maximized the slope $\frac{\log(S(\epsilon_{i+1}))-\log(S(\epsilon_i))}{\log(\epsilon_{i+1})-\log(\epsilon_i)}$ computed from \eqref{tuningfunction}.  In Figure \ref{spheresensitivity} we show that the variable bandwidth kernel is less sensitive to the parameter $\epsilon$ and moreover, the automated tuning algorithm of \cite{coifman2008TuningEpsilon} produces the nearly optimal value of $\epsilon$.  In contrast, the automated choice of $\epsilon$ for the fixed bandwidth kernel produces very poor results, shown in the bottom middle panel of Figure \ref{spheresensitivity}.  As expected, on a compact manifold such as this, the fixed bandwidth kernel can yield good results, however it requires a carefully tuned $\epsilon$ (see bottom right of Figure~\ref{spheresensitivity}). We should note that we repeated this experiment several time for different randomly generated initial covariance matrices and the results shown in Figure \ref{spheresensitivity} are typical when the sampling on the sphere is significantly non-uniform.  When the sampling is close to uniform the performance of fixed bandwidth kernel is very close to that of the variable bandwidth kernel.  Finally, we note that the maximum exponent in the  local power law, $S(\epsilon) \propto \epsilon^a$, for the fixed bandwidth kernel is approximately $0.9$ which would yield a dimension estimate of $d = 2a = 1.8$ which is close to the true dimension.  

\begin{figure}[h]
\centering
\includegraphics[width=0.31\textwidth]{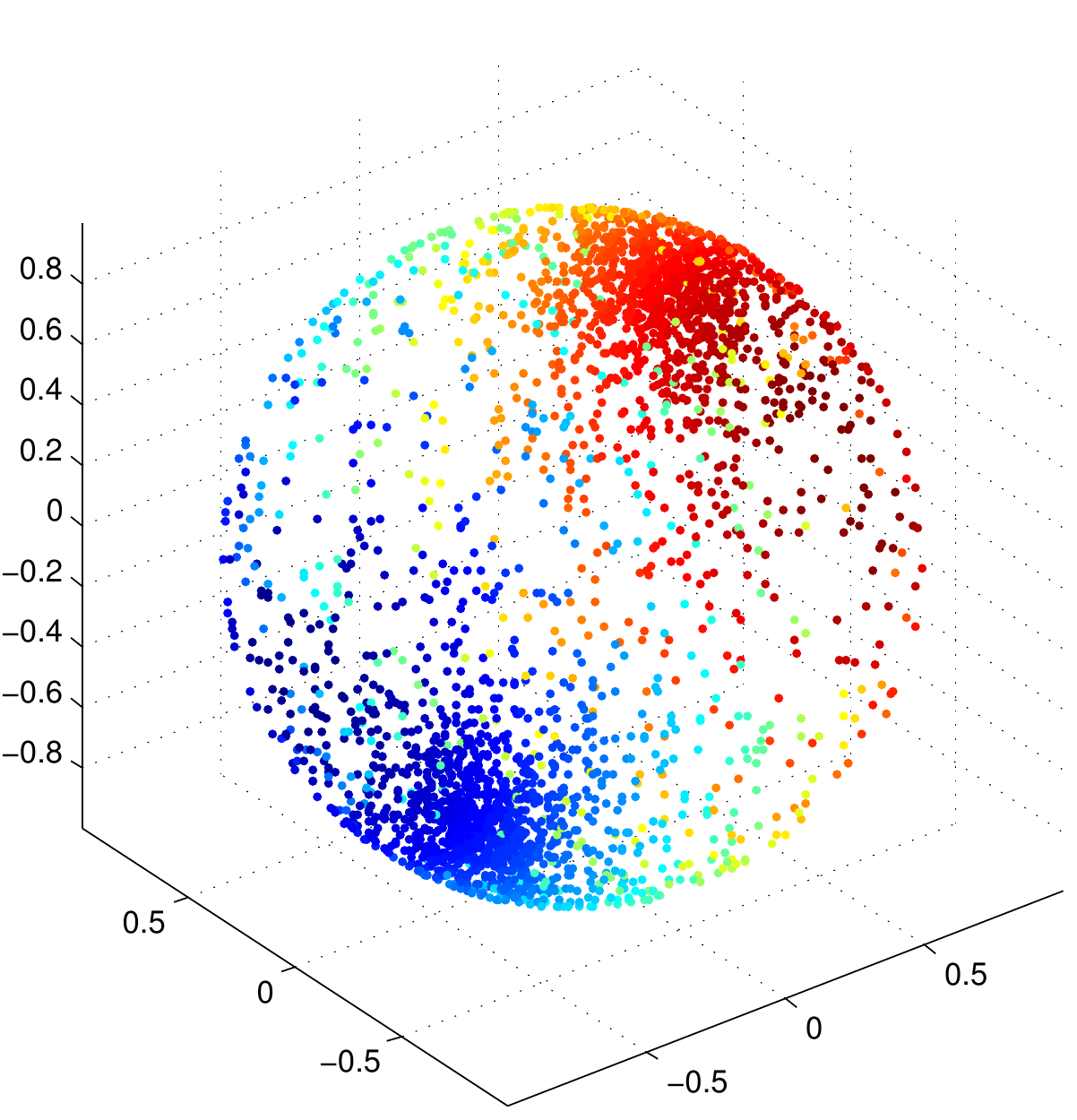}\includegraphics[width=0.336\textwidth]{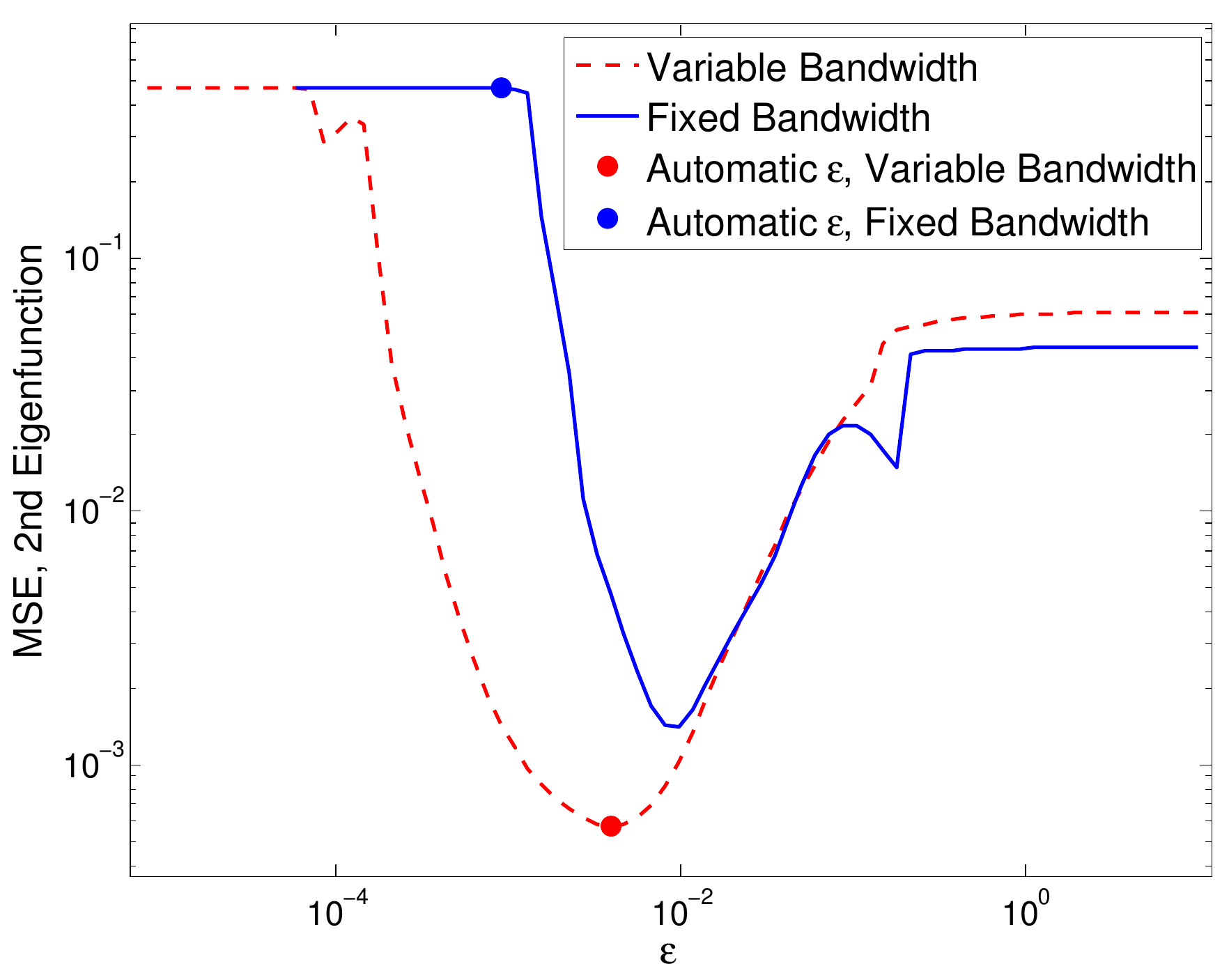}\includegraphics[width=0.34\textwidth]{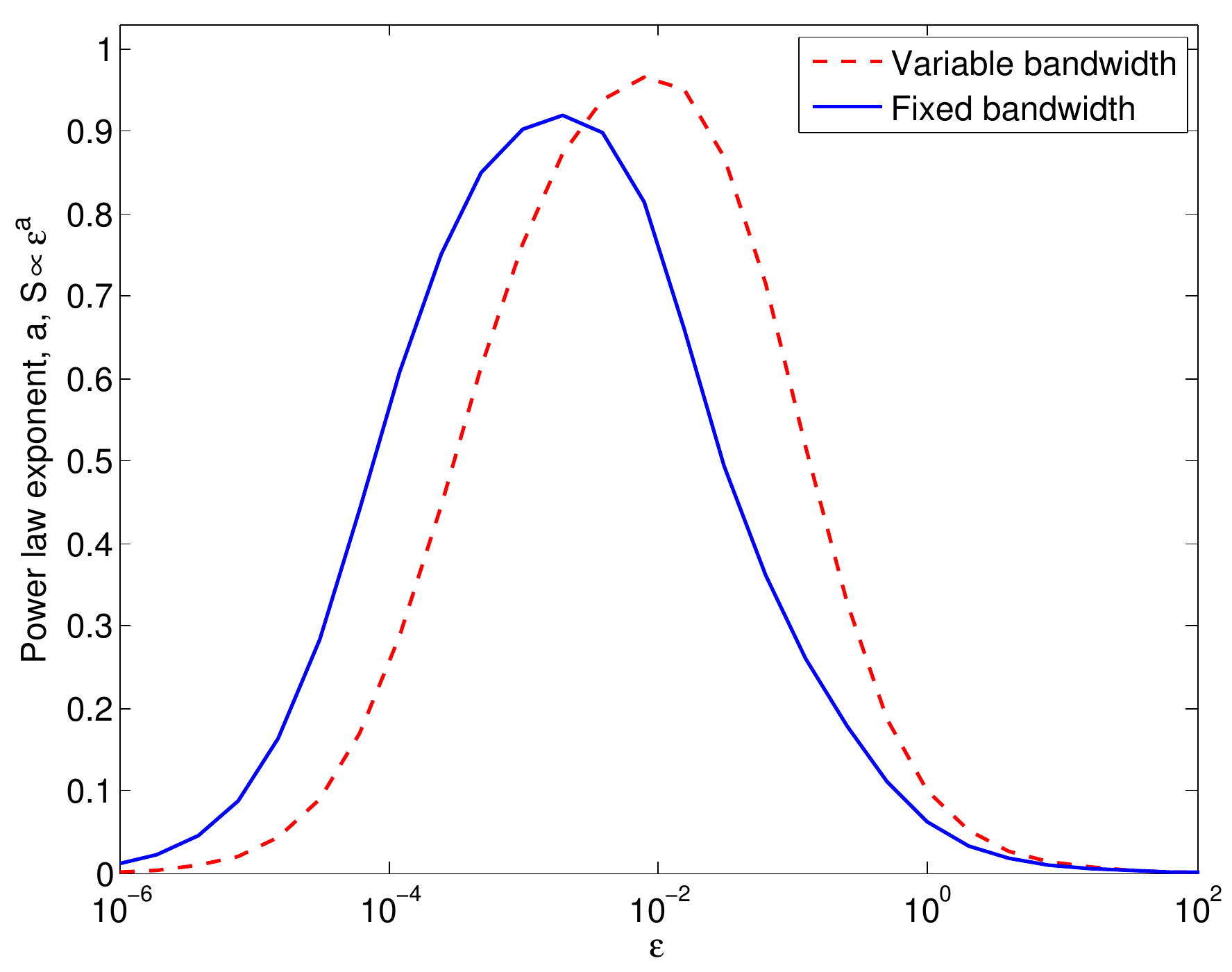}

\includegraphics[width=0.31\textwidth]{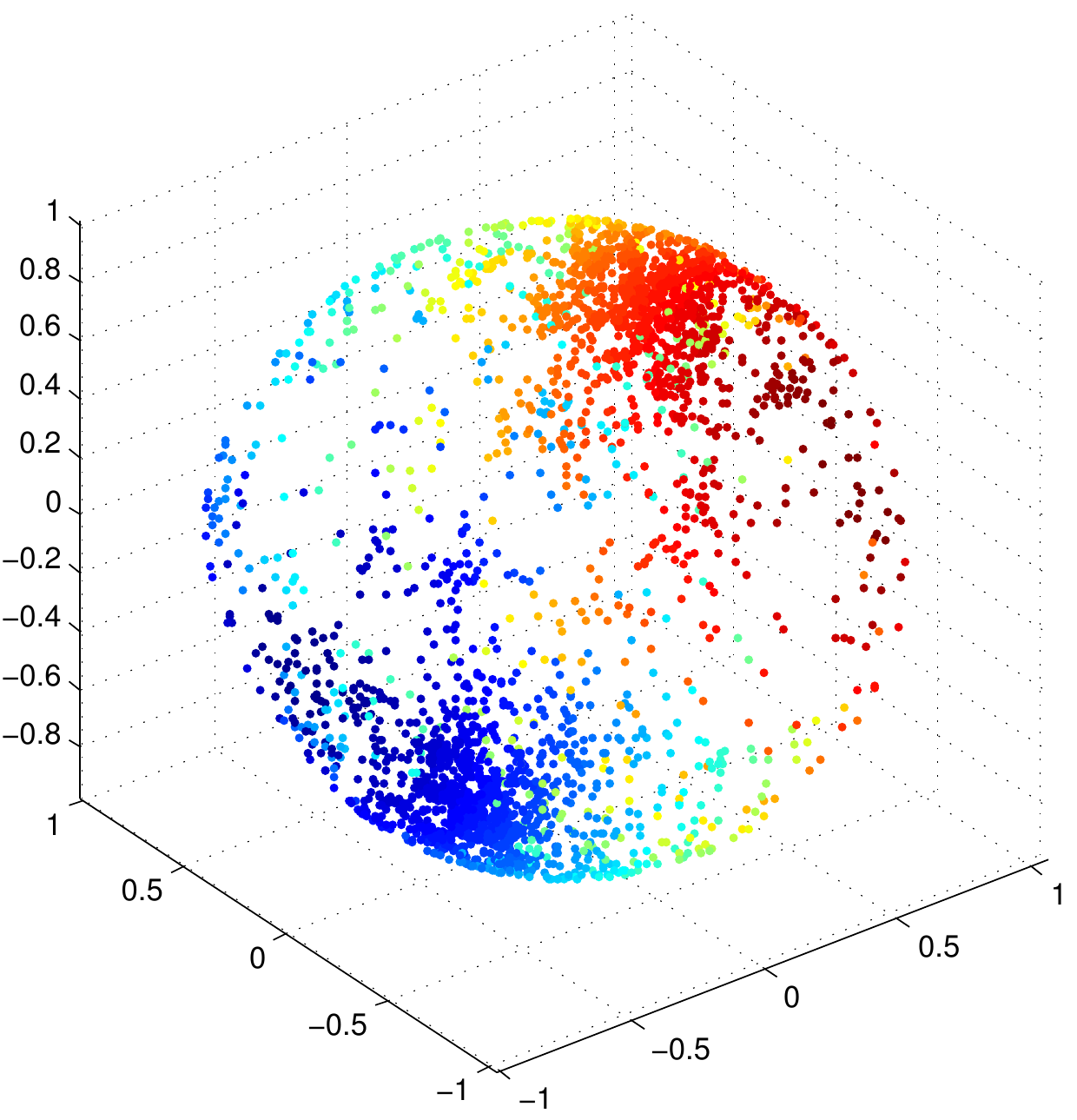}\hspace{.02\textwidth}\includegraphics[width=0.31\textwidth]{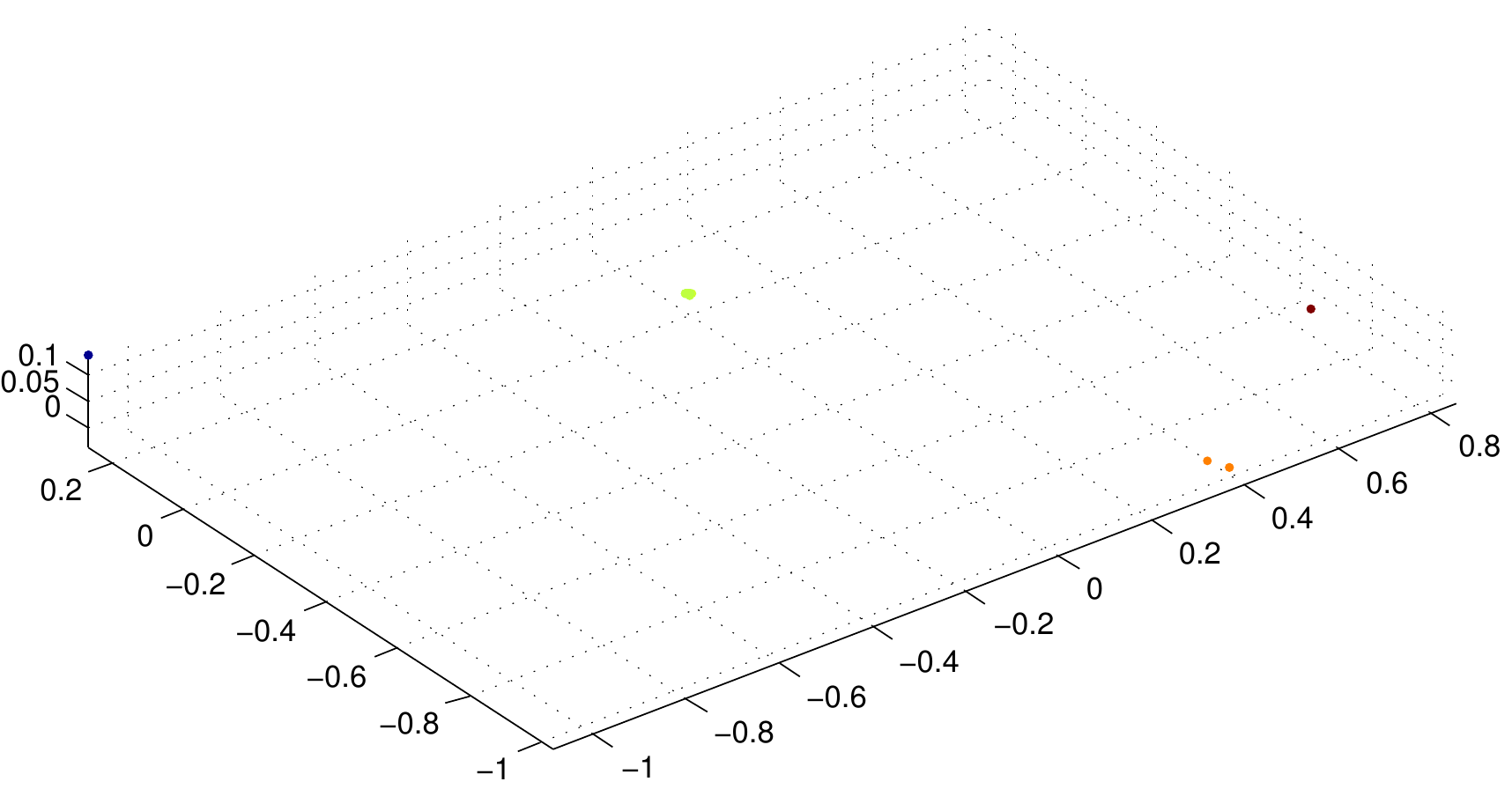}\hspace{.02\textwidth}\includegraphics[width=0.31\textwidth]{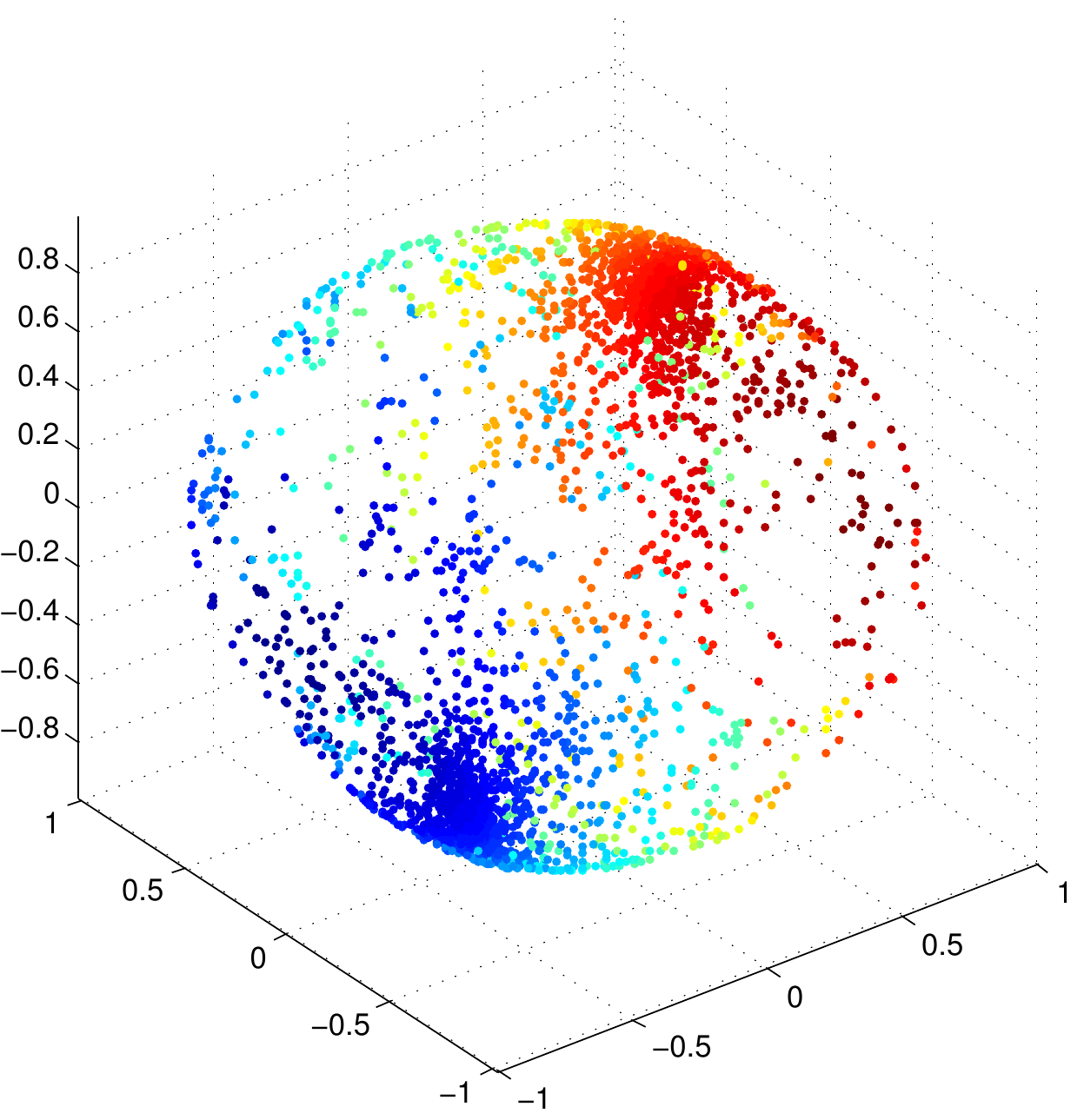}\hspace{.02\textwidth}

\caption{\label{spheresensitivity} Sensitivity of the approximation of the first three eigenfunctions ($\phi_1 \propto x, \phi_2\propto y$, and $\phi_3\propto z$) of the Laplacian on a unit sphere in $\mathbb{R}^3$ sampled by generating 3000 random points in $\mathbb{R}^3$ from a Gaussian distribution with a randomly generated $3\times 3$ covariance matrix.  These 3000 points where then projected onto the unit sphere. Top, left: The original data set. Top, middle: Mean squared error in the second eigenfunction (after normalization by an orthogonal transformation) as a function of $\epsilon$ for the fixed and variable bandwidth kernels highlighting the automated choices.  Top, right: The exponent, $a$, in the local power law fit $S(\epsilon) \propto \epsilon^a$ near each value of $\epsilon$, the automated choice of $\epsilon$ is found by maximizing this exponent.  Bottom: For each method, we use linear least squares to map the first three nontrivial eigenfunctions onto the original data set, this removes any scaling or orthogonal rotations due to the symmetry of the sphere.  Left: Variable bandwidth with automatically chosen $\epsilon$. Middle: Fixed bandwidth with automatically chosen $\epsilon$; note that this result is typical when $\epsilon$ chosen is so small that the points in the sparse regions of the manifold are completely disconnected from the rest of the manifold. Right: Fixed bandwidth with optimal value of $\epsilon$ from an exhaustive search by comparing with the true analytic eigenfunctions. }
\end{figure}


\section{Conclusion}\label{conclusion}

The theory developed above shows how to use variable bandwidth kernels to approximate the Laplacian and the generators of gradient flow systems.  We developed the general asymptotic expansion for the integral operators associated to the continuous variable bandwidth kernels of the form \eqref{vbkernel}.  This expansion reveals that a bandwidth function, $\rho$, changes the limiting operator to include a gradient term $(d+2)\nabla f \cdot\frac{\nabla \rho}{\rho}$.  In the case of uniform sampling, this implies that the Laplacian Eigenmaps algorithm, with a kernel of the form \eqref{vbkernel} will no longer produce a Laplacian operator, but instead will produce the generator of a gradient flow with potential field $U(x) = -(d+2)\log(\rho)$.  As shown in \eqref{Larho}, in the case of non-uniform sampling, we can remove the effect of the sampling, using $\alpha=1$ as in Diffusion Maps \cite{diffusion}.  This allows us to approximate the generator for any gradient flow system with known potential function $U$, by using the bandwidth function $\rho = \exp(-U/(d+2))$.  Alternatively, by choosing the bandwidth function to be a power of the sampling density, we can recover a result similar to that of \cite{diffusion} but with a constant $c_1$ which depends on both the $\alpha$ normalization and the exponent $\beta$ used in the bandwidth function.  Given a data set sampled from the invariant measure of a gradient flow system, this result allows us to use a large class of variable bandwidth kernels to approximate the generator of the gradient flow system from which the data originates without any prior knowledge of the potential function.

In practice we are interested in approximating the integral operators $\mathcal{L}_{\alpha,\beta}$ in \eqref{generator_ab} from data.  Often the data will be random samples of the density $q$, in which case we use discrete sums as Monte-Carlo approximations to the integral operators.  By extending the analysis of Singer in \cite{SingerEstimate}, we showed that the bias error estimates may be unbounded as the sampling density $q$ approaches zero.  Recall that in kernel density estimation problems, variable bandwidth kernels are known to give faster convergence rate and reduced sensitivity to the choice of bandwidth \cite{ScottVBK,ScottVBK2}, however it is still possible to use fixed bandwidth kernels to estimate densities.  The theory developed in Section~\ref{resultsSection} and the numerical example demonstrated in Section \ref{numerics} reveal that for operator approximation problems with densities that are not bounded away from zero, variable bandwidth kernels are not simply an improvement but are \emph{necessary} for convergence.  Fixed bandwidth kernels can be used for density estimation because the sampling density always appears in the numerator of the error bound whereas for operator approximation the density appears in the denominator of the error bound when $\beta=0$. 

From our numerical simulations, we found that on non-compact domains, the operators found by applying fixed bandwidth kernels do not always converge to the limits suggested by the continuous theory of \cite{diffusion} due to the large errors in the discrete operator estimates and that this limitation can be overcome by variable bandwidth kernels with an appropriate bandwidth function.  Moreover, for compact domains the estimation with variable bandwidth kernels are more robust to choices of the nuisance parameter $\epsilon$ which supports their popularity in practice. Indeed, we numerically found that the automated $\epsilon$-tuning method proposed in \cite{coifman2008TuningEpsilon} works quite well with this kernel in the sense that it detects the range of $\epsilon$ where the mean squared errors are empirically small.

\comment{The main drawback of the variable bandwidth kernel algorithm} 
The main drawback of our choice of bandwidth function is that it requires knowledge of the intrinsic dimension of the embedded manifold.  The automated $\epsilon$-tuning method proposed in \cite{coifman2008TuningEpsilon} seems to suggest one way to estimate this intrinsic dimension as we speculated above, however this can be difficult in the presence of noise and alternative techniques are given in \cite{HeinDimension,MaggioniDimension}.  The dimension is required in two ways.  First the equations for $c_1$ and $c_2$ depend on the dimension, and for $\beta \neq 0$ the dimension is required to find $c_1$ which determines the limiting operator.  Second, the estimate $q_{\epsilon}^S$ of the sampling density $q$, which is used to de-bias the operator using the $\alpha$ normalization of \cite{diffusion}, requires a true variable bandwidth density estimate.  Since the order zero term in the expansion of $G_{\epsilon}^S(q)$ is $m_0\rho^{d}q$, we need to divide by $\rho^d$ to recover $q$, leading to the definition $q_{\epsilon}^S(x_j) = \sum_{l} K_{\epsilon}^S(x_j,x_l)/\rho(x_j)^d$.  An alternative approach which we also considered was to perform the normalization without the division by $\rho^d$.  We conducted a thorough analysis of this alternative normalization and found different formulas for $c_1$ and $c_2$.  However in our analysis of this alternative formulation we found that the constraint $c_2<0$ would require $\frac{-1}{d} < \beta < 0$, which restricts $\beta$ significantly when $d$ is large, and means that once again the dimension must be known in order to choose $\beta$ in practice. 

Remaining limitations to the theory developed here are that the results to do not apply to non-compact manifolds with boundary and all of our convergence results only show pointwise convergence to the limiting operators applied to smooth functions. The theory of \cite{diffusion} shows that Neumann boundary conditions are implicit to the diffusion maps construction on compact manifolds, however their result relies strongly on the boundary of the manifold being compact and we suspect that it will require a different method of proof on non-compact manifolds where the boundary may be non-compact.  Moreover, a strong result of \cite{BNspectral} shows that for a fixed bandwidth kernel the pointwise convergence can be extended to spectral convergence, and extending their results to variable bandwidth kernels is another remaining challenge.  The variable bandwidth kernels considered here could also be generalized to allow anisotropy, meaning that the Euclidean norm could be replaced by a more general norm which may also vary at different points on the manifold.  Anisotropic kernels have been used successfully in kernel density estimation \cite{ScottVBK,ScottVBK2} and the first results on these kernels for operator estimation on compact manifolds can be found in \cite{SingerAnisotropy2008,KushnirAnisotropy}.  More recently, for compact manifolds it was shown that a large class of anisotropic kernels which simply have exponential decay in the distance correspond to changes in the Riemannian metric on the manifold \cite{BS14}.  Extending these results to non-compact manifolds with boundary, along with spectral convergence would complete the connection between kernel methods and Riemannian differential geometry.

\section*{Acknowledgments}
The research of J.H. is partially supported by the Office of Naval Research Grants N00014-13-1-0797, MURI N00014-12-1-0912 and the National Science Foundation DMS-1317919. T. B. is supported under the ONR MURI grant N00014-12-1-0912.

\bibliographystyle{plain}
\bibliography{VBbib}

\appendix


\section{Convergence rates of the continuous operators}\label{VBK}

The goal of this appendix is to determine the rate of convergence of the generator of the integral operator, 
\begin{align}
G^S_{\epsilon}f = \epsilon^{-d/2}\int_{\mathcal{M}} K_{\epsilon}^S(x,y)f(y)\, dV(y), \label{a1}
\end{align}
associated with the kernel $K^S_\epsilon$ in \eqref{vbkernel} to $\mathcal{L}_{\alpha,\rho}$ in \eqref{Larho} as $\epsilon\rightarrow 0$. To achieve this goal, we need to determine the asymptotic expansion of $G^S_{\epsilon}$ with respect to $\epsilon$.  

In order to find this expansion, we first extend a key technical lemma (cf. Lemma \ref{diffmaplemma}) of \cite{diffusion} to non-compact manifolds in \ref{dmextension}. Ultimately, this expansion can be found by describing
\comment{we will apply change of variable to describe} $G^S_{\epsilon}$ in \eqref{a1} as an integral operator associated to a non-symmetric kernel,
\begin{align} 
K^R_{\epsilon}(x,y) \equiv h\left(\frac{||x-y||^2}{\epsilon\rho(y)}\right), \label{rightkernel}
\end{align}
which requires us to find the asymptotic expansion of this ``right-formulation". It turns out that the asymptotic expansion of the right-formulation with respect to the kernel in \eqref{rightkernel} can be formulated in the weak sense, which will reduce this challenging expansion to the more simple case of finding the asymptotic expansion for the integral operator corresponding to the kernel,
\begin{align} K^L_{\epsilon}(x,y) \equiv h\left(\frac{||x-y||^2}{\epsilon\rho(x)}\right),\label{leftkernel}\end{align}
which we will call the ``left-formulation". Thus, we will first discuss the left-formulation in \ref{leftform}. Subsequently, we use this left-formulation to derive the right-formulation in \ref{rightform}. Finally, in \ref{symmform}, we combine these two results to find the expansion of the symmetric-formulation in \eqref{a1}.  
For the sake of clarity, we will compute all these expansions assuming uniform sampling and then we will extend the final expansion of $G^S_\epsilon$ in \eqref{a1} to non-uniform sampling in \ref{nonuniform}.  We note that the left and right formulations can also be extended to non-uniform sampling using the same technique.  


\subsection{Extending the asymptotic expansion of diffusion maps}\label{dmextension}

In this section, we extend the following fundamental lemma for kernel operator estimation to non-compact manifolds.
\begin{lem}[Expansion of Fixed Bandwidth Kernels, Coifman and Lafon \cite{diffusion}]\label{diffmaplemma} Let $f$ be a smooth real-valued function on an embedded $d$-dimensional compact manifold $\mathcal{M} \subset \mathbb{R}^n$ and let $h:[0,\infty) \to [0,\infty)$ have fast decay, then we have
\[ G_{\epsilon}f(x) = \epsilon^{-d/2}\int_{\mathcal{M}} h\left(\frac{||x-y||^2}{\epsilon}\right)f(y)\, dV(y) = m_0 f(x) + \epsilon m_2(\omega(x)f(x) + \Delta f(x)) + \mathcal{O}(\epsilon^2) \]
where $m_0 = \int_{\mathbb{R}^d} h(||z||^2)dz$ and $m_2  = \frac{1}{2}\int_{\mathbb{R}^d} z_1^2 h(||z||^2)dz$ are constants determined by $h$, and $\omega$ depends on the induced geometry of $\mathcal{M}$.  Note that $dV(y)$ is the volume form on $\mathcal{M}$ and the operator $\Delta$ is the (negative definite) Laplacian on $\mathcal{M}$, and these are both defined with respect to the Riemannian metric inherited from the ambient space.
\end{lem}
The above lemma assumes uniform sampling on a compact manifold since we can estimate the operator $G_{\epsilon}f$ by the discrete sum,  
\[ \lim_{N\to\infty} \frac{1}{N}\sum^N_{i=1} K_{\epsilon}(x,x_i)f(x_i) = \int_{\mathcal{M}} K_{\epsilon}(x,y)f(y)q(y)\, dV(y) = \frac{1}{\textup{vol}(\mathcal{M})}\int_{\mathcal{M}} K_{\epsilon}(x,y)f(y) \, dV(y) = \frac{\epsilon^{d/2}}{\textup{vol}(\mathcal{M})} G_{\epsilon}f(x), \]
where the second equality follows from assuming $q$ is sampled uniformly with respect to the volume form $dV(y)$ on the manifold $\mathcal{M}$. To see this, let $q(y)=c$ be uniform with respect to $dV(y)$.  Since $q$ is a density, we have $1 = \int_{\mathcal{M}}q(y)\, dV(y) = \int_{\mathcal{M}} c \, dV(y) = c \textup{vol}(\mathcal{M})$ and so $q(y) = c = \textup{vol}(\mathcal{M})^{-1}$ which explains why the Monte-Carlo sum estimates $\frac{\epsilon^{d/2}}{\textup{vol}(\mathcal{M})}G_{\epsilon}f$ instead of $G_{\epsilon}f$.  

Of course, if $\mathcal{M}$ is not compact the integral in $G_{\epsilon}$ may diverge, however, for practical applications which sample a non-compact manifold, the sampling density is typically not uniform, and especially on unbounded manifolds we will assume fast decay of the sampling density at infinity.  Here, we generalize Lemma~\ref{diffmaplemma} to non-compact manifolds as follows,

\begin{lem}[Expansion of Fixed Bandwidth Kernels on non-Compact Manifolds]\label{mainlemma} Let $\mathcal{M} \subset \mathbb{R}^n$ be an embedded $d$-dimensional manifold without boundary and let $q:\mathcal{M}\to (0,\infty)$ be bounded above such that $q\in L^1(\mathcal{M}) \cap \mathcal{C}^3(\mathcal{M})$.  Let $h:[0,\infty) \to [0,\infty)$ have fast decay in the sense that there exists $a,\sigma$ such that $h(x) < a\exp(-x/\sigma)$.  Then for all $f \in L^2(\mathcal{M},q) \cap \mathcal{C}^3(\mathcal{M})$ we have,
\begin{align} G_{\epsilon}(fq)(x) = \epsilon^{-d/2}\int_{\mathcal{M}} h\left(\frac{||x-y||^2}{\epsilon}\right)f(y)q(y)\, dV(y) = m_0 f(x)q(x) + \epsilon m_2(\omega(x)f(x)q(x) + \Delta (fq)(x)) + \mathcal{O}(\epsilon^2) \label{gefq}\end{align}
where $m_0 = \int_{\mathbb{R}^d} h(||z||^2)dz$ and $m_2  = \frac{1}{2}\int_{\mathbb{R}^d} z_1^2 h(||z||^2)dz$ are constants determined by $h$ and $q$, and $\omega$ depends on the induced geometry of $\mathcal{M}$.  The operator $\Delta$ is the (negative definite) Laplacian on $\mathcal{M}$ with the induced metric.
\end{lem}
\begin{proof}
We note that in the proof of Lemma~\ref{diffmaplemma} in \cite{diffusion} compactness is only used to bound the integral outside of a neighborhood of radius $\epsilon^{\gamma}$ for $0<\gamma<1/2$ around $x$.  Since $f \in L^2(\mathcal{M},q)$ we have,
\begin{align}
\left| \epsilon^{-d/2}\int_{y\in\mathcal{M},||x-y|| > \epsilon^{\gamma}} h\left(\frac{||x-y||^2}{\epsilon}\right)f(y)q(y)\, dV(y) \right| &\leq ||fq^{1/2}||_{L^2(\mathcal{M})}\left(  \epsilon^{-d/2}\int_{y\in\mathcal{M},||x-y|| > \epsilon^{\gamma}} h\left(\frac{||x-y||^2}{\epsilon}\right)^2 q(y)\, dV(y) \right)^{1/2} \nonumber \\
&= ||f||_{L^2(\mathcal{M},q)} \left(  \epsilon^{-d/2}\int_{z\in\hat{\mathcal{M}},||z|| > \epsilon^{\gamma-1/2}} h\left(||z||^2\right)^2 \epsilon^{d/2}q(x-\epsilon^{1/2}z)\, d\hat V(z) \right)^{1/2} \nonumber \\
&\leq ||f||_{L^2(\mathcal{M},q)} ||q||_{\infty} \left(  \int_{z\in\hat{\cal M},||z|| > \epsilon^{\gamma-1/2}} a e^{-2||z||^2/\sigma} \, d\hat V(z) \right)^{1/2} \nonumber \\
&\leq ||f||_{L^2(\mathcal{M},q)} ||q||_{\infty} \left(  \int_{||z|| > \epsilon^{\gamma-1/2}} a e^{-2||z||^2/\sigma} \, dz \right)^{1/2} 
= \mathcal{O}(\epsilon^2),
\end{align}
where $x-y = \sqrt{\epsilon}z$ so $dV(y) = \epsilon^{d/2} d\hat{V}(z)$, where $d\hat{V}(z)$ is the volume form of the transformed manifold $\hat{\cal M}$, and the last inequality follows from extending the integral from the manifold to the entire ambient space, and the final equality follows from the exponential decay of the integral of the tail of a Gaussian distribution, since $\epsilon^{\gamma-1/2} \to \infty$ as $\epsilon\to 0$.  This allows us to localize the integral in $G_{\epsilon}$ in \eqref{gefq} to an $\epsilon^{\gamma}$ neighborhood of $x$ so that,
\[ G_{\epsilon}(fq)(x) =  \epsilon^{-d/2}\int_{\mathcal{M}} h\left(\frac{||x-y||^2}{\epsilon}\right)f(y)q(y)\, dV(y) =  \epsilon^{-d/2}\int_{y\in\mathcal{M},||x-y|| < \epsilon^{\gamma}} h\left(\frac{||x-y||^2}{\epsilon}\right)f(y)q(y)\, dV(y) + \mathcal{O}(\epsilon^2). \]
  The remainder of the proof proceeds by local asymptotic expansion of $G_\epsilon$ following  \cite{diffusion} with no further modifications.
\end{proof}

Lemma~\ref{mainlemma} shows that the central asymptotic expansion of \cite{diffusion} extends trivially to non-compact manifolds by assuming an integrable sampling distribution $q$ which is bounded above.  Lemma \ref{mainlemma} requires that $f \in L^2(\mathcal{M},q)$ so the expansion is only valid for functions whose growth is controlled by the decay of the sampling $q$.  In particular, for unbounded manifolds such as $\mathbb{R}$ (see Section \ref{numerics}) the eigenfunctions of the operator $\mathcal{L}_{\alpha,\beta}$ in \eqref{generator_ab} may have polynomial growth and thus we will typically assume $q$ has exponential decay at infinity.  


\subsection{Left formulation of uniformly sampled data}\label{leftform}

Let $\rho(x)$ be a positive function on the manifold and define the variable bandwidth Gaussian kernel, with bandwidth $\rho$ to be $K_{\epsilon}^L$ in \eqref{leftkernel}. For simplicity we first assume uniform sampling on a compact manifold, since \ref{dmextension} shows that these expansions directly generalize by simply applying the operator $G_{\epsilon}$ in \eqref{gefq} to the product $fq$ under appropriate assumptions on $f$ and $q$.  In \ref{nonuniform} we will return to non-uniform sampling using this strategy. Under the uniform sampling assumption, the effect of the bandwidth function, $\rho$, is to weight the Laplacian. To show this, we define the following change of variables, $\hat y = \mathcal{F}(y) = \frac{y-x}{\sqrt{\rho(x)}} + x$. Then the integral operator,
\begin{align} G^L_{\epsilon}f(x) &= \epsilon^{-d/2} \int_{\mathcal{M}} h\left( \frac{||x-y||^2}{\epsilon\rho(x)} \right)f(y)\, dV(y) \nonumber \\
&= \epsilon^{-d/2}\rho(x)^{d/2} \int_{\mathcal{F}(\mathcal{M})} h\left( \frac{||x-\hat y||^2}{\epsilon} \right)f\left(\sqrt{\rho(x)}(\hat y-x)+x\right)d\hat V(\hat y) 
\end{align}
where $\left|\frac{dy}{d\hat y}\right| = \rho^{d/2}$.  We can now apply Lemma~\ref{diffmaplemma} to the integral expression with the function $\hat f(\hat y) = f\left(\sqrt{\rho(x)}(\hat y - x) + x\right)$ so that,
\begin{align}\label{rhoExpansion} G^L_{\epsilon}f(x) 
&= \rho(x)^{d/2} G_{\epsilon}\hat f(x) 
= \rho(x)^{d/2} \left(m_0\hat f(x) + m_2\epsilon(\omega(x)\hat f(x) + \Delta \hat f(x))\right),
\end{align}
where $m_0$ and $m_2$ are constants determined by the shape function $h$,
\[ m_0 = \int_{\mathbb{R}^d} h\left( ||z||^2 \right)dz,\quad\quad m_2 = \frac{1}{2}\int_{\mathbb{R}^d}z_1^2 h\left( ||z||^2 \right)dz, \] 
which are the same expressions as for the fixed bandwidth kernel in Lemma \ref{diffmaplemma}.  Note that the transformation $\kappa(\hat y) = \sqrt{\rho(x)}(\hat y-x)+x$ inside the function $f(\sqrt{\rho(x)}(\hat y-x)+x)$ corresponds to a change of metric in the tangent space $T_x\mathcal{M}$.  The map $\kappa(\hat y)$ is a local diffeomorphism such that $\kappa(x) = x$ and $D\kappa(\hat y) = \sqrt{\rho(x)}I_{d\times d}$.  In this small neighborhood of $x$, the Laplacian can be written locally as,
\begin{align}\label{rhoLaplacian} \Delta \hat f(x) &\equiv \left. \frac{1}{\sqrt{|g|}}\partial_{\hat y_i} \left(g^{ij}\sqrt{|g|} \partial_{\hat y_j} (f \circ \kappa)(\hat y)  \right) \right|_{\hat y=x} 
=\left. \frac{1}{\sqrt{|g|}}\partial_{\hat y_i} \left( g^{ij}\sqrt{|g|} (\partial_{\hat y_j} f) \circ \kappa (\hat y)\, \partial_{\hat y_j}\kappa(\hat y) \right) \right|_{\hat y=x} \nonumber \\
&=\left. \frac{\sqrt{\rho(x)}}{\sqrt{|g|}}\partial_{\hat y_i} \left( (g^{ij}\sqrt{|g|} \partial_{\hat y_j} f) \circ \kappa(\hat y) \right) \right|_{\hat y=x} =\left. \frac{\sqrt{\rho(x)}}{\sqrt{|g|}}\partial_{\hat y_i} \left(g^{ij}\sqrt{|g|} (\partial_{\hat y_j} f) \right) \circ \kappa(\hat y)\, \partial_{\hat y_i} \kappa(\hat y) \right|_{\hat y=x} \nonumber \\
&=\left. \frac{\rho(x)}{\sqrt{|g|}}\partial_{\hat y_i} \left(g^{ij}\sqrt{|g|} (\partial_{\hat y_j} f) \right) \circ \kappa(\hat y)  \right|_{\hat y=x} = \frac{\rho(x)}{\sqrt{|g|}}\partial_{\hat y_i} \left(g^{ij}\sqrt{|g|} (\partial_{\hat y_j} f(x) \right)  \nonumber \\
&= \rho(x)\Delta f(x). 
\end{align}
Notice that the above result is simply two applications of the chain rule, combined with the fact that $\kappa(x)=x$ and $D\kappa(\hat y) = \sqrt{\rho(x)}$.  Combining \eqref{rhoLaplacian} with \eqref{rhoExpansion} we have the following expansion for the $G_{\epsilon}^L$ based on the variable bandwidth kernel,
\begin{align}\label{VBExpansion} G^L_{\epsilon}f(x) &=  \rho(x)^{d/2} \left(m_0 f(x) + m_2\epsilon(\omega(x) f(x) + \rho(x) \Delta  f(x))\right) +\mathcal{O}(\epsilon^2)\nonumber \\
&= m_0 \rho(x)^{d/2} f(x) \left(1 + \epsilon m\left(\omega(x) + \rho(x) \frac{\Delta  f(x)}{f(x)}\right)\right)+\mathcal{O}(\epsilon^2),
\end{align}
where $m\equiv m_2/m_0$.  We now apply a left-normalization, dividing by $G^L_{\epsilon}1(x)$ outside the operator so we have,
\[ \frac{G^L_{\epsilon}f(x)}{G^L_{\epsilon}1(x)} = f(x) + \epsilon m \rho(x)\Delta f(x) +\mathcal{O}(\epsilon^2). \]
Finally, we can extract the order-$\epsilon$ term, defining the operator $L^L_{\epsilon}$ by,
\begin{align} L_{\epsilon}^Lf(x) = \frac{1}{\epsilon m \rho(x)}\left(\frac{G^L_{\epsilon}f(x)}{G^L_{\epsilon}1(x)} - f(x)\right) = \Delta f(x) + \mathcal{O}(\epsilon). \label{lle}\end{align}
In Figure~\ref{leftrightcomp}, we numerically verify this asymptotic expansion with a simple example on a periodic domain. Of course we can only easily approximate the operator $G^L_{\epsilon}$ in \eqref{VBExpansion} via a Monte-Carlo integral approximation, 
\[ \lim_{N\to\infty} \frac{1}{N}\sum_{i=1}^N K_{\epsilon}^L(x,x_i)f(x_i)  = \frac{1}{\textup{vol}(\mathcal{M})}\int_{\mathcal{M}}K_{\epsilon}^L(x,y)f(y)\, dV(y) = \frac{\epsilon^{d/2}}{\textup{vol}(\mathcal{M})}G^L_{\epsilon}f(x) , \] 
which is only valid when the sampling density of $\{x_i\}$ is uniform on $\mathcal{M}$.  We will return to the case of non-uniform sampling in \ref{nonuniform}.  Note that the factor $\epsilon^{d/2}/\textup{vol}(\mathcal{M})$ does not affect the operator $L_{\epsilon}^L$ in \eqref{lle} due to the left-normalization.  We now turn to the right formulation $K_{\epsilon}^R(x,y) = K_{\epsilon}^L(y,x)$ in \eqref{rightkernel} and we will emphasize the significant difference between these deceptively similar kernels.


\subsection{Right formulation of uniformly sampled data}\label{rightform}

In this section we continue to assume uniform sampling and consider the right formulation of the variable bandwidth kernel $K^R_{\epsilon}$ in \eqref{rightkernel}, which contains a variable bandwidth dependent only on $y$.  As we will see, the dependence on $\rho(y)$ will not be the same as that of the variable bandwidth given by $\rho(x)$ considered in the previous section. In this case, it is not possible to simply change variables to eliminate the $\rho(y)$ term from the kernel, since any change of variables which involves both $y$ and $\rho(y)$ would not be explicitly invertible.  Moreover, the Jacobian of such a change of variables would involve the gradient of the bandwidth function $\rho(y)$, which gives some insight into the difference between bandwidth functions which depend on $y$ rather than $x$.

To find an expansion of the following operator,
\[ G^R_{\epsilon}f(x) = \epsilon^{-d/2}\int_{\mathcal{M}} h\left(\frac{||x-y||^2}{\epsilon\rho(y)}\right)f(y)\, dV(y), \]
we consider a weak formulation,
\[ \left<g, G^R_{\epsilon}f\right> = \epsilon^{-d/2}\int_{\mathcal{M}}\int_{\mathcal{M}} h\left(\frac{||x-y||^2}{\epsilon\rho(y)}\right)f(y)g(x) \,dV(y) \, dV(x), \]
 for any arbitrarily smooth function $g$.
We now simply exchange the order of integrations, and expand the inner integral using our previous result in \eqref{VBExpansion},
\begin{align} \left<g, G^R_{\epsilon}f\right> &= \int_{\mathcal{M}} \left(\epsilon^{-d/2}\int_{\mathcal{M}} h\left(\frac{||x-y||^2}{\epsilon\rho(y)}\right)g(x)\, dV(x) \right) f(y) \, dV(y) \nonumber \\
&= \int_{\mathcal{M}} \rho(y)^{d/2}\left[m_0 g(y) + \epsilon m_2 (\omega(y)g(y) + \rho(y)\Delta g(y)) \right] f(y) \, dV(y) + \mathcal{O}(\epsilon^2) \nonumber \\
&= \int_{\mathcal{M}}  m_0 \rho(y)^{d/2}f(y)g(y) + \epsilon m_2 \omega(y)\rho(y)^{d/2}f(y)g(y)\, dV(y) + \epsilon m_2 \int_{\mathcal{M}} f(y)\rho(y)^{d/2 + 1}\Delta g(y) \, dV(y) + \mathcal{O}(\epsilon^2) \nonumber \\
&= \int_{\mathcal{M}}  m_0 \rho(y)^{d/2}f(y)g(y) + \epsilon m_2 \omega(y)\rho(y)^{d/2}f(y)g(y)\, dV(y) + \epsilon m_2 \int_{\mathcal{M}} g(y)\Delta \left(f(y)\rho(y)^{d/2 + 1}\right) \, dV(y) + \mathcal{O}(\epsilon^2) \nonumber
\end{align} 
where the last equality follows from the symmetry of $\Delta$, meaning $\left<f \rho^{d/2+1},\Delta g\right> = \left< \Delta \left(f\rho^{d/2+1}\right),g\right>$.  Recombining the two integrals, we summarize the above calculation as,
\[ \left<g, G^R_{\epsilon}f\right> = \int_{\mathcal{M}} g(y) \big(m_0 \rho(y)^{d/2}f(y) + \epsilon m_2 ( \omega(y)\rho(y)^{d/2}f(y)-\Delta (f(y)\rho(y)^{d/2 + 1}))\big) \, dV(y) + \mathcal{O}(\epsilon^2). \]
Finally, since $g(y)$ was arbitrary, we conclude that,
\[ G^R_{\epsilon}f = m_0 \rho^{d/2}f + \epsilon m_2 ( \omega\rho^{d/2}f+\Delta (f\rho^{d/2 + 1})) + \mathcal{O}(\epsilon^2), \]
where the equality is in the weak sense, and thus these operators will be equal on the space of smooth functions.  Expanding $\Delta (f\rho^{d/2 + 1})$ we find,
\[ \Delta (f\rho^{d/2 + 1}) = f\Delta(\rho^{d/2+1}) + \rho^{d/2+1}\Delta f + (d+2)\rho^{d/2}\nabla \rho \cdot \nabla f, \]
which allows us to write the operator expansion as,
\begin{align}\label{rightVMexpansion} G^R_{\epsilon}f = \rho^{d/2} \left(m_0 f + \epsilon m_2 ( \tilde\omega f + \rho\Delta f + (d+2)\nabla\rho \cdot \nabla f)\right) + \mathcal{O}(\epsilon^2), \end{align}
where $\tilde\omega = \omega-\rho^{-d/2}\Delta(\rho^{d/2+1})$.  If we apply the left normalization, dividing by $G^R_{\epsilon}1(x)$ outside the operator, we find that,
\[ \frac{G^R_{\epsilon}f(x)}{G^R_{\epsilon}1(x)} = f(x) + \epsilon m\big( \rho(x)\Delta f(x) + (d+2)\nabla \rho(x)\cdot \nabla f(x) \big) + \mathcal{O}(\epsilon^2). \]
Finally, we can extract the order-$\epsilon$ term, defining the operator $L^R_{\epsilon}$ by,
\[ L^R_{\epsilon}f \equiv \frac{1}{\epsilon m \rho} \left(\frac{G^R_{\epsilon}f}{G^R_{\epsilon}1} - f\right) = \Delta f + (d+2)\frac{\nabla \rho}{\rho}\cdot \nabla f + \mathcal{O}(\epsilon) . \]
Note that when the variable bandwidth is a function of $y$, the operator $L^R_\epsilon$ takes the form of a Kolmogorov operator for diffusion in potential well given by $U(x) = -(d+2)\log \rho(x)$.  We verify this result in Figure \ref{leftrightcomp}.  We now turn to the symmetric formulation with kernel $K^S_{\epsilon}(x,y)$ in \eqref{vbkernel} which will require this result.


\subsection{Symmetric bandwidth for uniformly sampled data}\label{symmform}

Continuing with our assumption of uniform sampling, we now return to the symmetric kernel $K^S_\epsilon$ in \eqref{vbkernel} and the associated operator, 
\[ G^S_{\epsilon}f(x) = \epsilon^{-d/2}\int_{\mathcal{M}}K^S_{\epsilon}(x,y)f(y)\, dV(y) = \epsilon^{-d/2}\int_{\mathcal{M}}h\left(\frac{||x-y||^2}{\epsilon\rho(x)\rho(y)}\right)f(y) \, dV(y). \]
In order to expand this expression, we will first change variables to eliminate the $\rho(x)$ term and then we will apply the expansion in \eqref{rightVMexpansion}.  Define the change of variables $\hat y = {\cal F}(y) = x - \frac{x-y}{\sqrt{\rho(x)}}$ so that $y = x-\sqrt{\rho(x)}(x-\hat y)$ and 
\[ G^S_{\epsilon}f(x) = \epsilon^{-d/2} \rho(x)^{d/2}\int_{\cal F(\cal M)}h\left(\frac{||x-\hat y||^2}{\epsilon \rho(x-\sqrt{\rho(x)}(x-\hat y))}\right)f(x-\sqrt{\rho(x)}(x-\hat y)) \, d\hat V( \hat y). \]
Letting $\hat \rho(\hat y) = \rho(x-\sqrt{\rho(x)}(x-\hat y))$ and $\hat f(\hat y) = f(x-\sqrt{\rho(x)}(x-\hat y))$ we have,
\[ G^S_{\epsilon}f(x) = \epsilon^{-d/2} \rho(x)^{d/2}\int_{\cal F(\cal M)}h\left(\frac{||x-\hat y||^2}{\epsilon \hat\rho(\hat y)}\right)\hat f(\hat y) \, d\hat V(\hat y). \]
Applying the expansion \eqref{rightVMexpansion} from the previous section we have,
\[ G^S_{\epsilon}f = \rho^{d/2}\hat \rho^{d/2} \left(m_0 \hat f + \epsilon m_2 ( \tilde\omega \hat f + \hat \rho\Delta \hat f + (d+2)\nabla \hat \rho \cdot \nabla \hat f)\right) + \mathcal{O}(\epsilon^2). \]
Note that $\hat \rho(x)=\rho(x)$ and $\hat f(x) = f(x)$ so that 
\[ G^S_{\epsilon}f = \rho^{d} \left(m_0 f + \epsilon m_2 ( \tilde\omega f +  \rho\Delta \hat f + (d+2)\nabla \hat \rho \cdot \nabla \hat f)\right) + \mathcal{O}(\epsilon^2). \]
Furthermore, $\nabla \hat f(x) = \left. \nabla f(x+\sqrt{\rho(x)}(x-\hat y)) \right|_{\hat y =x} = \sqrt{\rho(x)}\nabla f(x)$ and $\nabla \hat \rho(x) = \sqrt{\rho(x)}\nabla \rho(x)$ and $\Delta \hat f(x) = \rho(x)\Delta f(x)$ so we have,
\begin{align}\label{symmetricExpansion} G^S_{\epsilon}f = \rho^{d} \left(m_0 f + \epsilon m_2 ( \tilde\omega f +  \rho^2\Delta f + (d+2)\rho \nabla \rho \cdot \nabla f)\right) + \mathcal{O}(\epsilon^2). \end{align}
Applying left-normalization we find the operator $L^S_{\epsilon}$ given by, 
\begin{align}\label{symmetricOperator} L^S_{\epsilon}f(x) \equiv \frac{1}{\epsilon m \rho(x)^2}\left(\frac{G^S_{\epsilon}f(x)}{G^S_{\epsilon}1(x)}-f(x)\right) = \Delta f + (d+2)\frac{\nabla \rho}{\rho}\cdot \nabla f + \mathcal{O}(\epsilon). \end{align}
We verify this formula in Figure \ref{leftrightcomp} on a unit circle in $\mathbb{R}^2$ and a flat torus in $\mathbb{R}^4$ with data points spaced uniformly in the latent spaces $\theta \in [0,2\pi)$ and $(\theta,\phi) \in [0,2\pi)^2$ respectively.

\noindent {\bf Remark.} For general variable bandwidth with non-symmetric variable bandwidth Kernels,
\[K^U_{\epsilon}(x,y) = h\left(\frac{||x-y||^2}{\epsilon \rho_1(x)\rho_2(y)}\right), \]
it is not difficult to check that under the uniform sampling assumption, 
\[ L^U_{\epsilon}f(x) \equiv  \frac{1}{\epsilon m\, \rho_1(x)\rho_2(x)}\left(\frac{G^U_{\epsilon}f(x)}{G^U_{\epsilon}1(x)}-f(x)\right) = \Delta f + (d+2)\frac{\nabla \rho_2}{\rho_2}\cdot \nabla f + \mathcal{O}(\epsilon), \]
where $G^U_{\epsilon}\equiv \epsilon^{-d/2} \int_{\cal M} K^U_{\epsilon}(x,y)f(y)\,dV(y)$.

\begin{figure}
\centering
\includegraphics[width=0.4\textwidth]{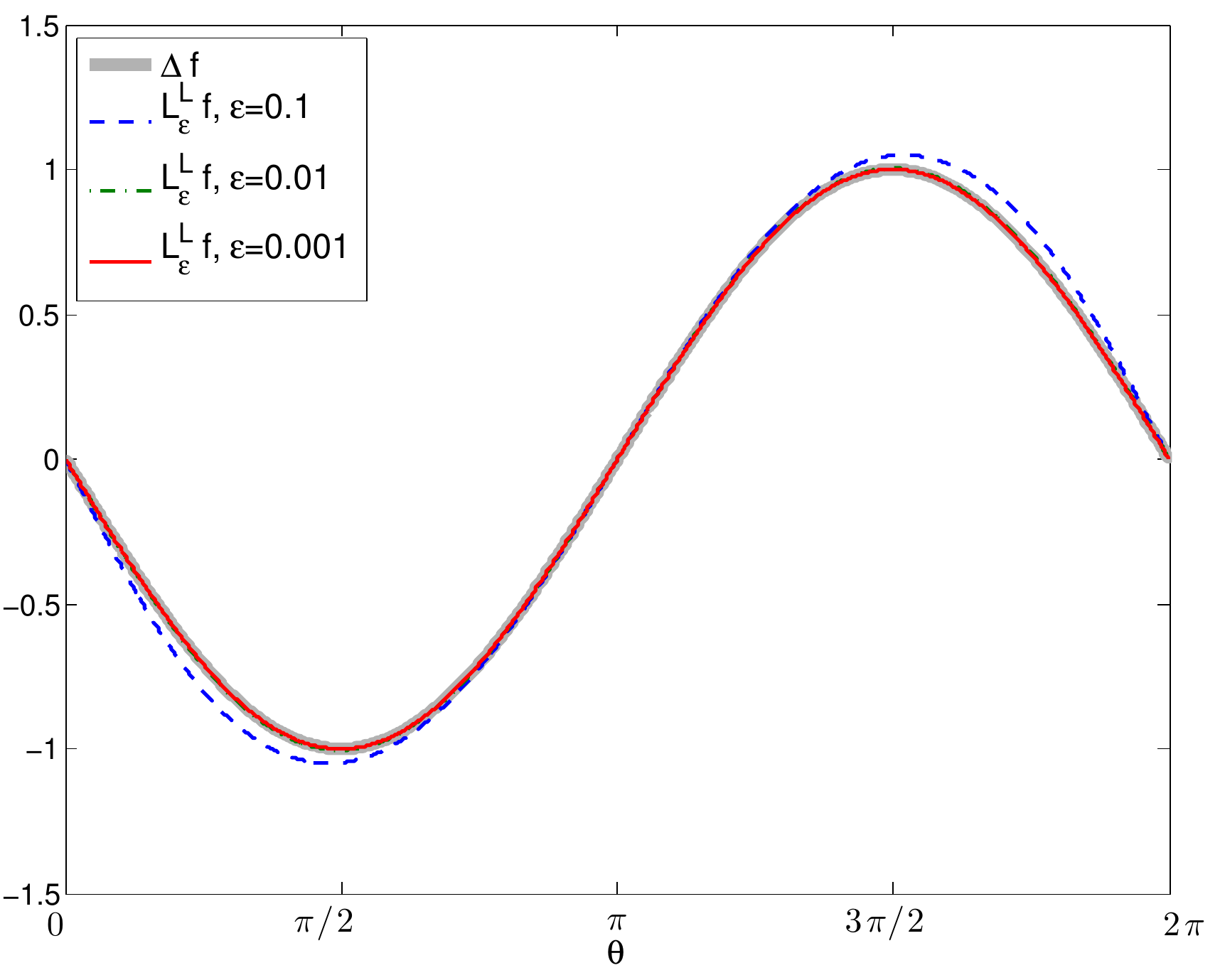}\includegraphics[width=0.4\textwidth]{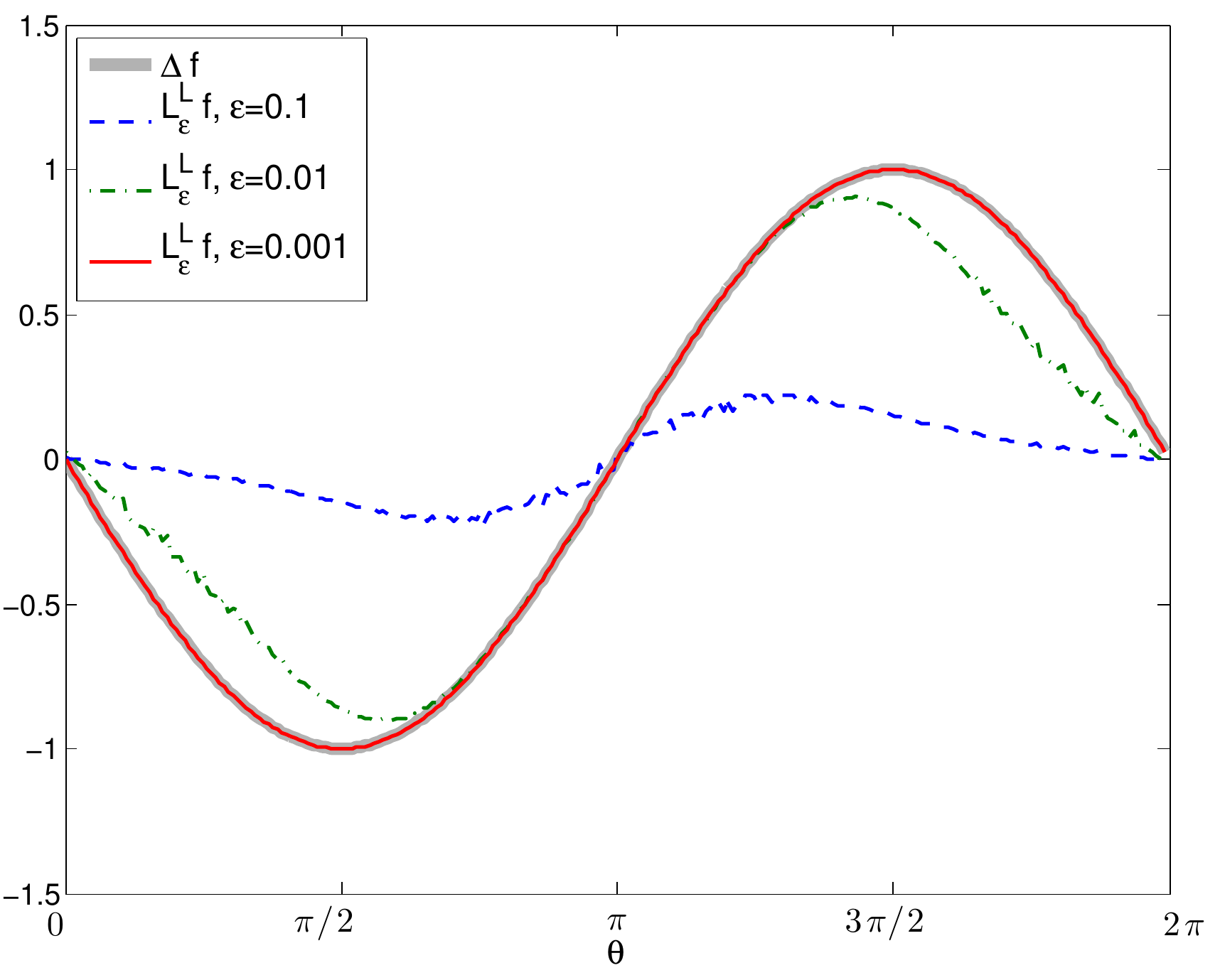}

\includegraphics[width=0.4\textwidth]{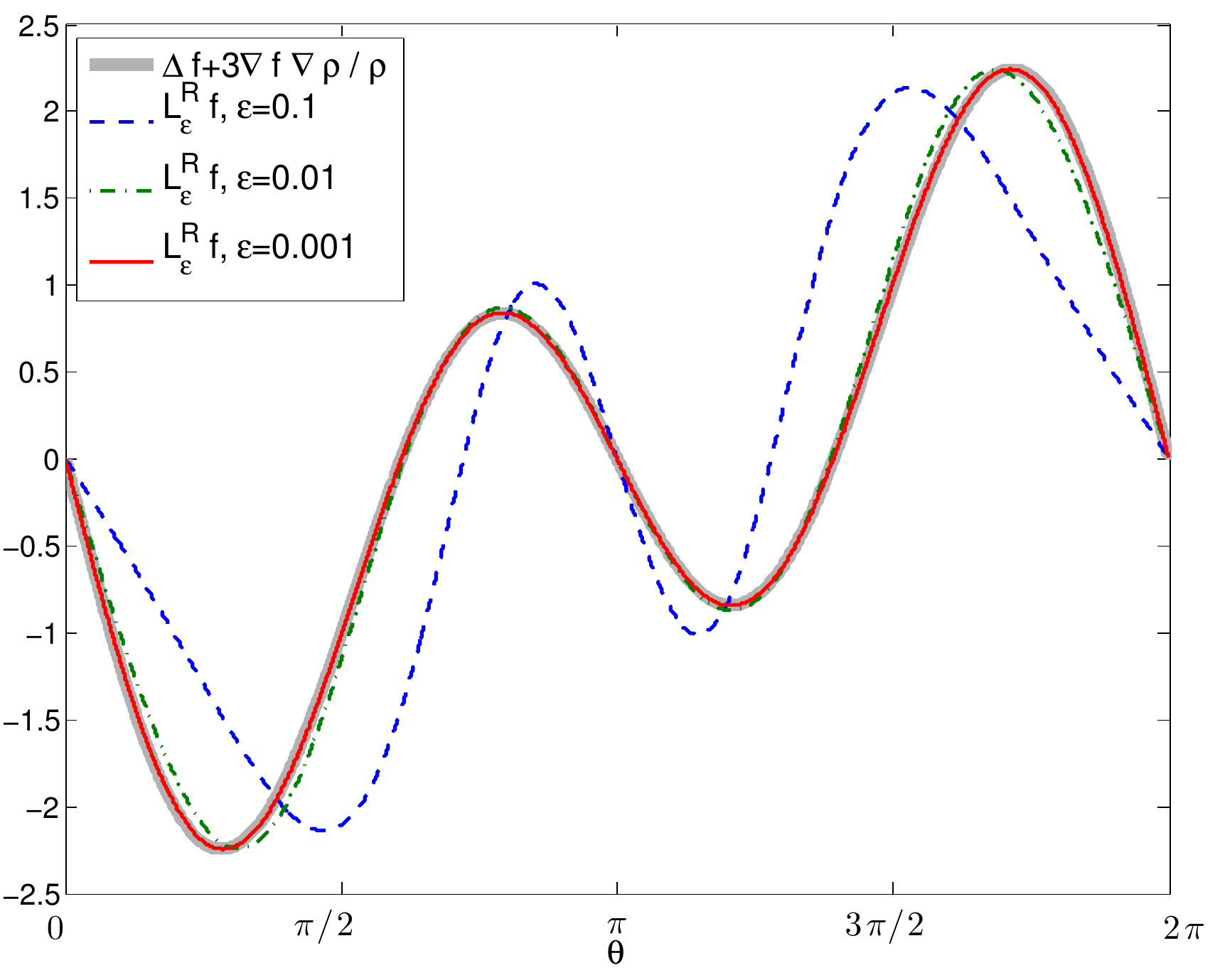}\includegraphics[width=0.4\textwidth]{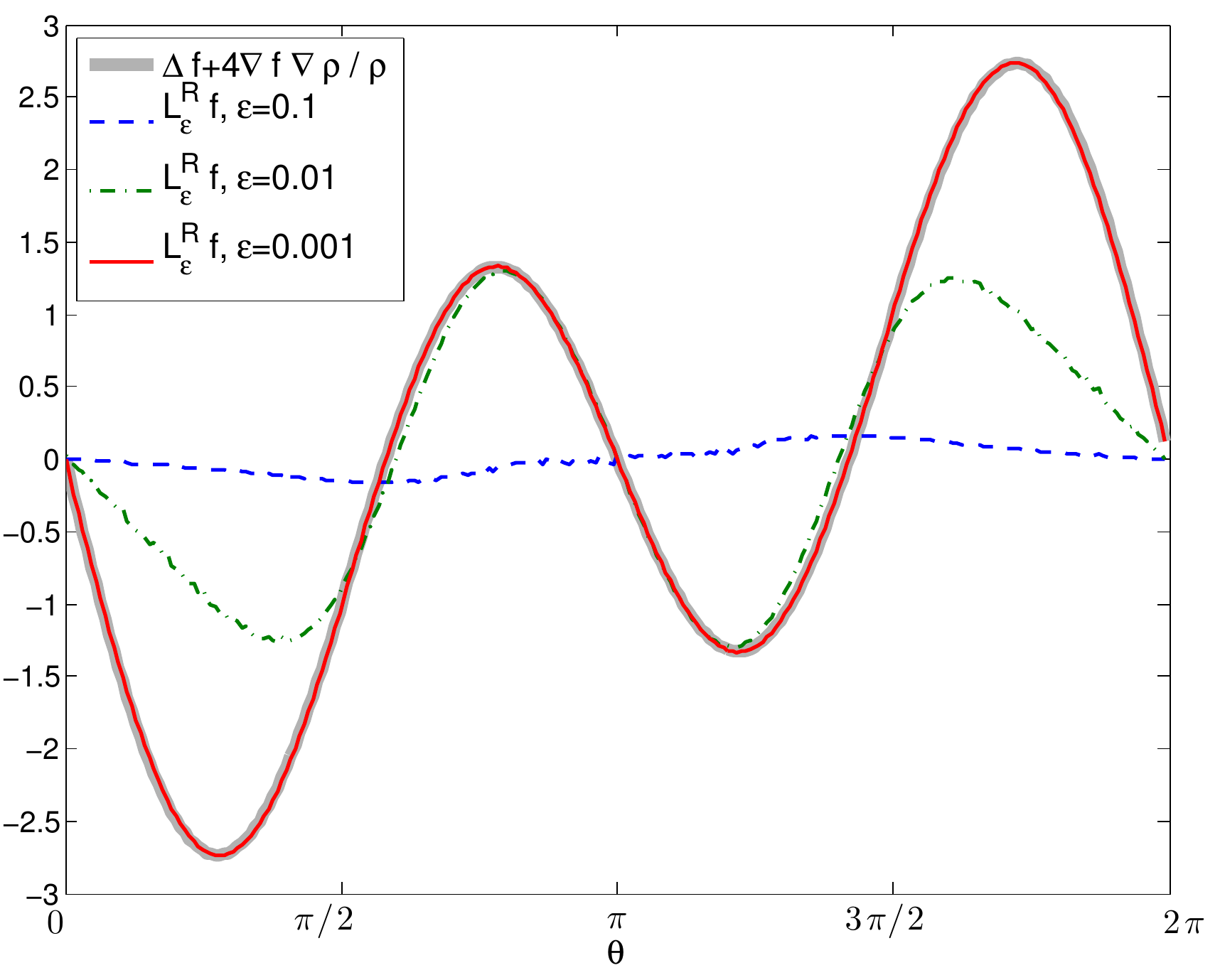}

\includegraphics[width=0.4\textwidth]{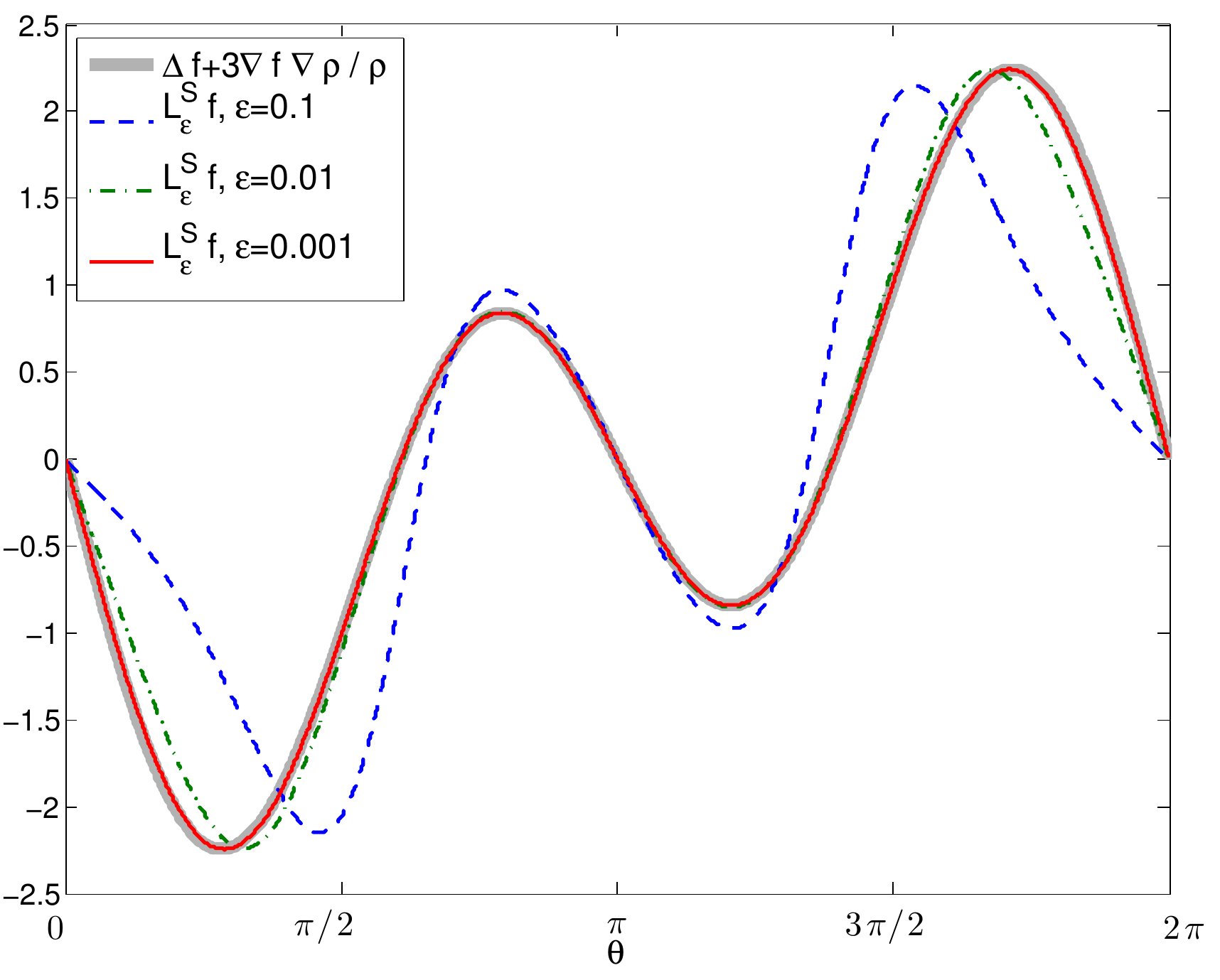}\includegraphics[width=0.4\textwidth]{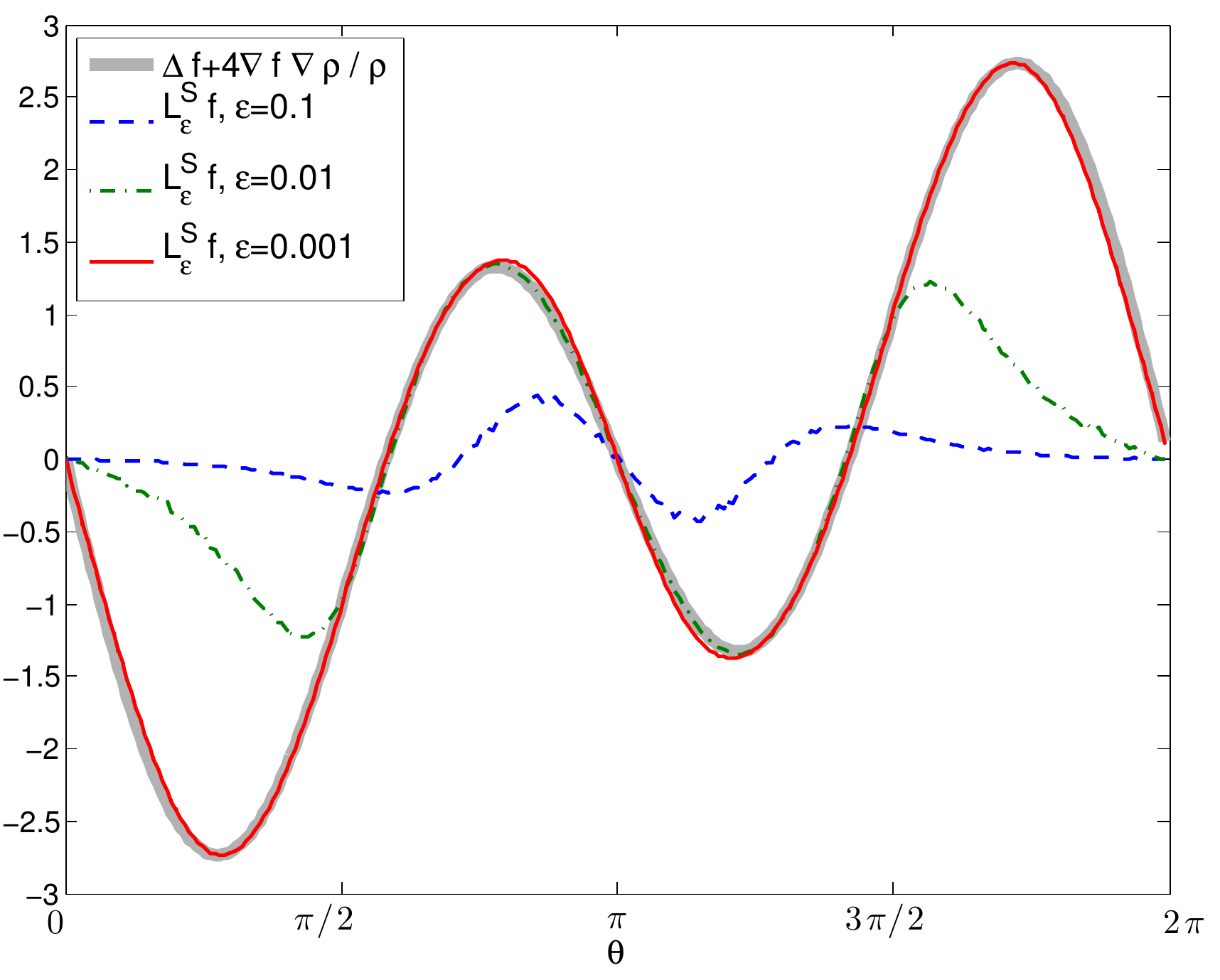}
\caption{\label{leftrightcomp} Operators $L_{\epsilon}^L$ (top), $L_{\epsilon}^R$ (middle), and $L_{\epsilon}^S$ (bottom) with variable bandwidth $\rho(\theta)=\exp(\cos(\theta))$ applied to $f(\theta) = \sin(\theta)$ (left column) and $f(\theta,\phi) = \sin(\theta)$ (right column).  Functions and operators are constructed on a uniform grid of 3000 points on the unit circle (left) and a uniform grid of 62,500 points on a flat torus in $\mathbb{R}^4$ (right).  Each operator is constructed for $\epsilon=0.1$ (blue), $0.01$ (green), and $0.001$ (red).  Left, top: $\Delta f = -\sin(\theta)$ (grey) compared to $L_{\epsilon}^Lf$. Left, middle: $\Delta f+(d+2)\frac{\nabla \rho}{\rho}\cdot\nabla f = -\sin(\theta)-3\sin(\theta)\cos(\theta)$ (grey) compared to $L_{\epsilon}^Rf$.  Left, bottom: $\Delta f+(d+2)\frac{\nabla \rho}{\rho}\cdot\nabla f = -\sin(\theta)-3\sin(\theta)\cos(\theta)$ (grey) compared to $L_{\epsilon}^Sf$.  Right, top: $\Delta f = -\sin(\theta)$ (grey) compared to $L_{\epsilon}^Lf$. Right, middle: $\Delta f+(d+2)\frac{\nabla \rho}{\rho}\cdot\nabla f = -\sin(\theta)-4\sin(\theta)\cos(\theta)$ (grey) compared to $L_{\epsilon}^Rf$.  Right, bottom: $\Delta f+(d+2)\frac{\nabla \rho}{\rho}\cdot\nabla f= -\sin(\theta)-4\sin(\theta)\cos(\theta)$ (grey) compared to $L_{\epsilon}^Sf$.  Note that on the torus we use a sparse matrix construction of the operators where only the 500 nearest neighbors of each point are allowed nonzero entries, this degrades the result for large $\epsilon$ but has no effect as $\epsilon$ becomes small due to the exponential decay of the kernel.}
\end{figure}


\subsection{Symmetric bandwidth for non-uniformly sampled data}\label{nonuniform}

Using the expansion of the symmetric variable bandwidth kernel from the previous section, we can now extend the result to the case of non-uniform sampling following the strategy of Coifman and Lafon in \cite{diffusion}.  Assume a positive sampling measure with density function $q(x)$ on $\mathcal{M}$, then when we compute Monte-Carlo approximations of kernel operators we will find,
\[ \lim_{N\to\infty} \frac{1}{N} \sum_{i=1}^N K^S_{\epsilon}(x,x_i)f(x_i) = \int_{\mathcal{M}} K^S_{\epsilon}(x,y)f(y)q(y) \, dV(y) = \epsilon^{d/2} G^S_{\epsilon}(fq)(x). \]
The previous equation implies that the direct application of our kernel $K^S_{\epsilon}$ will be biased by $q(x)$.  Thus we define the biased operator,
\[ G^S_{\epsilon,q}(f) \equiv G^S_{\epsilon}(fq). \]  
In order to remove the bias we first estimate $q(x)$ in the sense of a kernel density estimate by setting $f(x) = 1$.  Applying the result \eqref{symmetricExpansion} from the previous section and introducing notation, 
\begin{align}
\mathcal{L}^Sf = \rho^2\frac{\Delta f}{f} + (d+2)\rho \nabla \rho \cdot \frac{\nabla f}{f}\label{ELS} 
\end{align}
we have,
\begin{align} G^S_{\epsilon,q}(1) &= G^S_{\epsilon}(q) = \rho^{d} \left(m_0q + \epsilon m_2 ( \tilde\omega q +  \rho^2\Delta q + (d+2)\rho \nabla \rho \cdot \nabla q)\right) + \mathcal{O}(\epsilon^2) \nonumber \\
&= m_0\rho^{d}q \left(1 + \epsilon  m  \left( \tilde\omega +  \mathcal{L}^S q\right)\right) + \mathcal{O}(\epsilon^2) \label{GSeq1}
\end{align}
As in the standard Diffusion Map formulation \cite{diffusion}, we introduce the ``de-biasing" parameter $\alpha$ and note that,
\[ G^S_{\epsilon,q}(1)^{\alpha} = (m_0\rho^{d}q)^{\alpha} \left(1 + \alpha\epsilon  m  \left( \tilde\omega +  \mathcal{L}^S q\right)\right) + \mathcal{O}(\epsilon^2). \]
In particular, the sampling bias in the integral operator is removed through a right normalization,
\begin{align} G^S_{\epsilon,q,\alpha}(f) \equiv G^S_{\epsilon,q}\left( \frac{f\rho^{d\alpha}}{G^S_{\epsilon,q}(1)^{\alpha}} \right).\label{gseqa} \end{align}
Note that while this normalization appears to make the kernel non-symmetric, in Section \ref{algorithm} we present a numerical technique which will allow us to maintain the symmetry of the kernel matrix for the purpose of finding eigenvalues. Applying the result \eqref{symmetricExpansion} from the previous section we have,
\begin{align} G^S_{\epsilon,q,\alpha}(f) &= G^S_{\epsilon}\left( \frac{fq\rho^{d\alpha}}{G^S_{\epsilon,q}(1)^{\alpha}} \right) \nonumber \\
&= m_0\rho^{d(1+\alpha)}f q G^S_{\epsilon,q}(1)^{-\alpha}  \left(1 + \epsilon  m  \left( \tilde\omega +  \mathcal{L}^S(f \rho^{d\alpha}q G^S_{\epsilon,q}(1)^{-\alpha})\right)\right) + \mathcal{O}(\epsilon^2)  \nonumber \\
&= m_0\rho^{d(1+\alpha)}f q (m_0\rho^{d}q)^{-\alpha} \left(1 - \alpha\epsilon  m  \left( \tilde\omega +  \mathcal{L}^S q\right)\right)  \left(1 + \epsilon  m  \left( \tilde\omega +  \mathcal{L}^S(f q (m_0q)^{-\alpha})\right)\right) + \mathcal{O}(\epsilon^2)  \nonumber \\
&= f\rho^{d}(m_0q)^{1-\alpha} \left(1 - \alpha\epsilon  m  \left( \tilde\omega +  \mathcal{L}^S q\right) + \epsilon  m  \left( \tilde\omega +  \mathcal{L}^S(f q (m_0q)^{-\alpha})\right)\right) + \mathcal{O}(\epsilon^2)  \nonumber \\
&= f\rho^{d}(m_0q)^{1-\alpha} \left(1 + \epsilon  m  \left( (1-\alpha) \tilde\omega - \alpha \mathcal{L}^S q +  \mathcal{L}^S(f q (m_0q)^{-\alpha})\right)\right) + \mathcal{O}(\epsilon^2)  \label{GSeqaf}
\end{align}
Now applying left normalization we find,
\begin{align}\label{symmetricExpansionNonUniform} \frac{G^S_{\epsilon,q,\alpha}(f)}{G^S_{\epsilon,q,\alpha}(1)} &= \frac{f\rho^{d}(m_0q)^{1-\alpha} \left(1 + \epsilon  m  \left( (1-\alpha) \tilde\omega -\alpha\mathcal{L}^S q +  \mathcal{L}^S(f q (m_0q)^{-\alpha})\right)\right)}{\rho^{d}(m_0q)^{1-\alpha} \left(1 + \epsilon  m  \left( (1-\alpha) \tilde\omega - \alpha \mathcal{L}^S q +  \mathcal{L}^S(q (m_0q)^{-\alpha})\right)\right)} + \mathcal{O}(\epsilon^2)\nonumber \\
&= f \left(1+ \epsilon m \left( \mathcal{L}^S(fq (m_0q)^{-\alpha}) - \mathcal{L}^S(q (m_0q)^{-\alpha}) \right) \right)+ \mathcal{O}(\epsilon^2)
\end{align}
Extracting the order-$\epsilon$ term we have the operator,
\begin{align}
L^S_{\epsilon,\alpha}f(x) &\equiv \frac{1}{\epsilon m \rho(x)^2}\left( \frac{G^S_{\epsilon,q,\alpha}f(x)}{G^S_{\epsilon,q,\alpha}1(x)} - f(x) \right) = \frac{f}{\rho^2}\left(\mathcal{L}^S(fq (m_0q)^{-\alpha}) - \mathcal{L}^S(q (m_0q)^{-\alpha})\right) + \mathcal{O}(\epsilon) \nonumber \\
&= \frac{f}{\rho^2}\left( \rho^2\frac{\Delta (fg)}{fg} + (d+2)\rho \nabla \rho \cdot \frac{\nabla (fg)}{fg} - \rho^2\frac{\Delta g}{g} - (d+2)\rho \nabla \rho \cdot \frac{\nabla g}{g} \right)  + \mathcal{O}(\epsilon) \nonumber \\
&=  \frac{\Delta (fg)}{g} - \frac{f\Delta g}{g} + (d+2) \frac{\nabla \rho}{\rho} \cdot \frac{\nabla (fg)}{g}  - (d+2) f \frac{\nabla \rho}{\rho} \cdot \frac{\nabla g}{g}  + \mathcal{O}(\epsilon)  \nonumber \\
&=  \Delta f + 2\nabla f \cdot \frac{\nabla g}{g}+ (d+2) \nabla f \cdot \frac{\nabla \rho}{\rho}  + \mathcal{O}(\epsilon)\label{symmetricOperatorNonUniform} 
\end{align}
where $g \equiv m_0^{-\alpha}q^{1-\alpha}$ is introduced for convenience.  Note that $\frac{\nabla g}{g} = (1-\alpha)\frac{\nabla q}{q}$ so we can simplify the previous expression to,
\begin{align}\label{gradientFlow}  
L^S_{\epsilon,\alpha}f(x)  &= \Delta f + 2(1-\alpha)\nabla f \cdot \frac{\nabla q}{q}+ (d+2) \nabla f \cdot \frac{\nabla \rho}{\rho}  + \mathcal{O}(\epsilon) 
\end{align}
This completes the proof of the first error bound in Theorem~\ref{mainresult}. Note that \eqref{gradientFlow} shows how the variable bandwidth function $\rho$ effects the operator defined by the kernel.  When $\rho = 1$ is constant, we recover the result of \cite{diffusion}, namely a gradient flow with potential function $U = -2(1-\alpha)\log q$ defined by the sampling density $q$.  The formula \eqref{gradientFlow} reveals that we can use a variable bandwidth kernel to approximate the generator of a gradient flow for an arbitrary potential function $U$ by choosing bandwidth function $\rho = e^{-U/(d+2)}$ so that $(d+2)\frac{\nabla \rho}{\rho} = -\nabla U$ is the vector field defined by the gradient of the potential function $U$.  Setting $\alpha=1$ as in \cite{diffusion}, we remove the effect of the sampling density $q$ on the operator, and recover the desired generator $\Delta f - \nabla U \cdot \nabla f$.  

Finally, if we make the choice $\rho = q^{\beta}$, we find,
\begin{align}\label{fullSymmetricResult}
L^S_{\epsilon,\alpha,\beta}f(x) = \Delta f + c_1\nabla f \cdot \frac{\nabla q}{q}  + \mathcal{O}(\epsilon) \end{align}
where $L^S_{\epsilon,\alpha,\beta}$ is exactly $L^S_{\epsilon,\alpha}$ in \eqref{gradientFlow} with $\rho$ replaced by $q^\beta$
and $c_1 = 2-2\alpha + d\beta + 2\beta$.  In Figure~\ref{symmetricNonUniform}, we numerically verify the expansion in \eqref{fullSymmetricResult} on a circle sampled according to the distribution $q(\theta) = \exp(\cos(\theta))$.  Since $d=1$, setting $\beta = -1/2$ and $\alpha = 1/4$ we find $c_1=0$ and we recover the Laplacian on the circle.  Setting $\beta = -1/2$ and $\alpha=-1/4$ we find $c_1 = 1$ which yields the Kolmogorov operator for the potential $U(\theta) = -\log(q(\theta)) = -\cos(\theta)$ on the circle.

In practical applications the sampling density, $q$, will usually not be known, however we can first use any kernel to estimate the sampling density, see \cite{RosenblattFBK,ParzenFBK,ScottVBK,ScottVBK2} as well as the numerical details in Section \ref{algorithm}. Of course, this means that we will actually have $\rho = q^{\beta} + \mathcal{O}(\epsilon)$ (for $N$ sufficiently large).  While this approximation will affect the expansion \eqref{gseqa} with the symmetric kernel $K^S_{\epsilon}$ in \eqref{vbkernel}, it is easy to see that all of these effects are canceled by the left-normalization in \eqref{symmetricExpansionNonUniform}. Since the estimate in \eqref{fullSymmetricResult} is already order-$\epsilon$ this result is not affected by order-$\epsilon$ errors in the sampling density estimate which is used for the bandwidth function.

\begin{figure}
\centering
\includegraphics[width=0.4\textwidth]{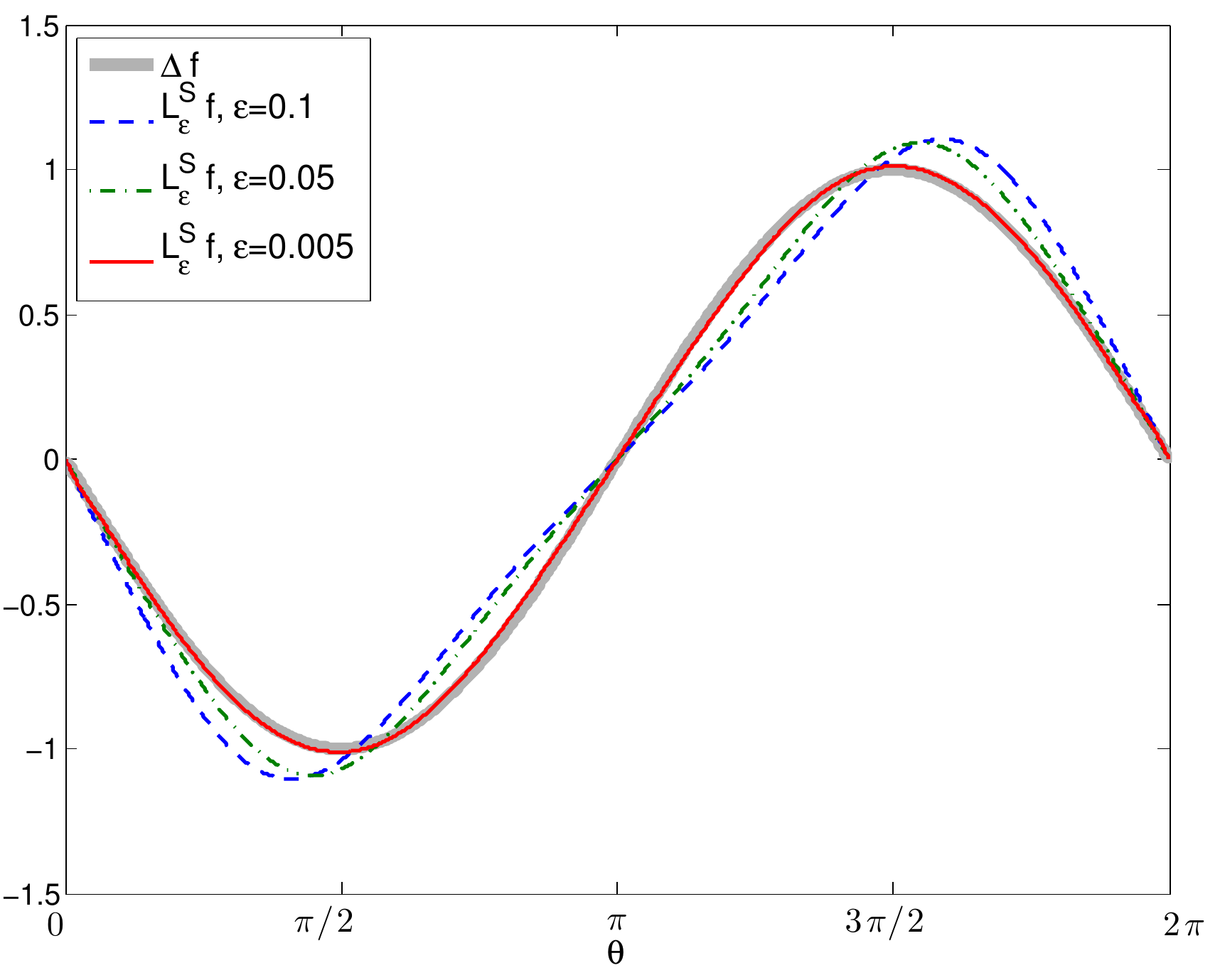}\includegraphics[width=0.4\textwidth]{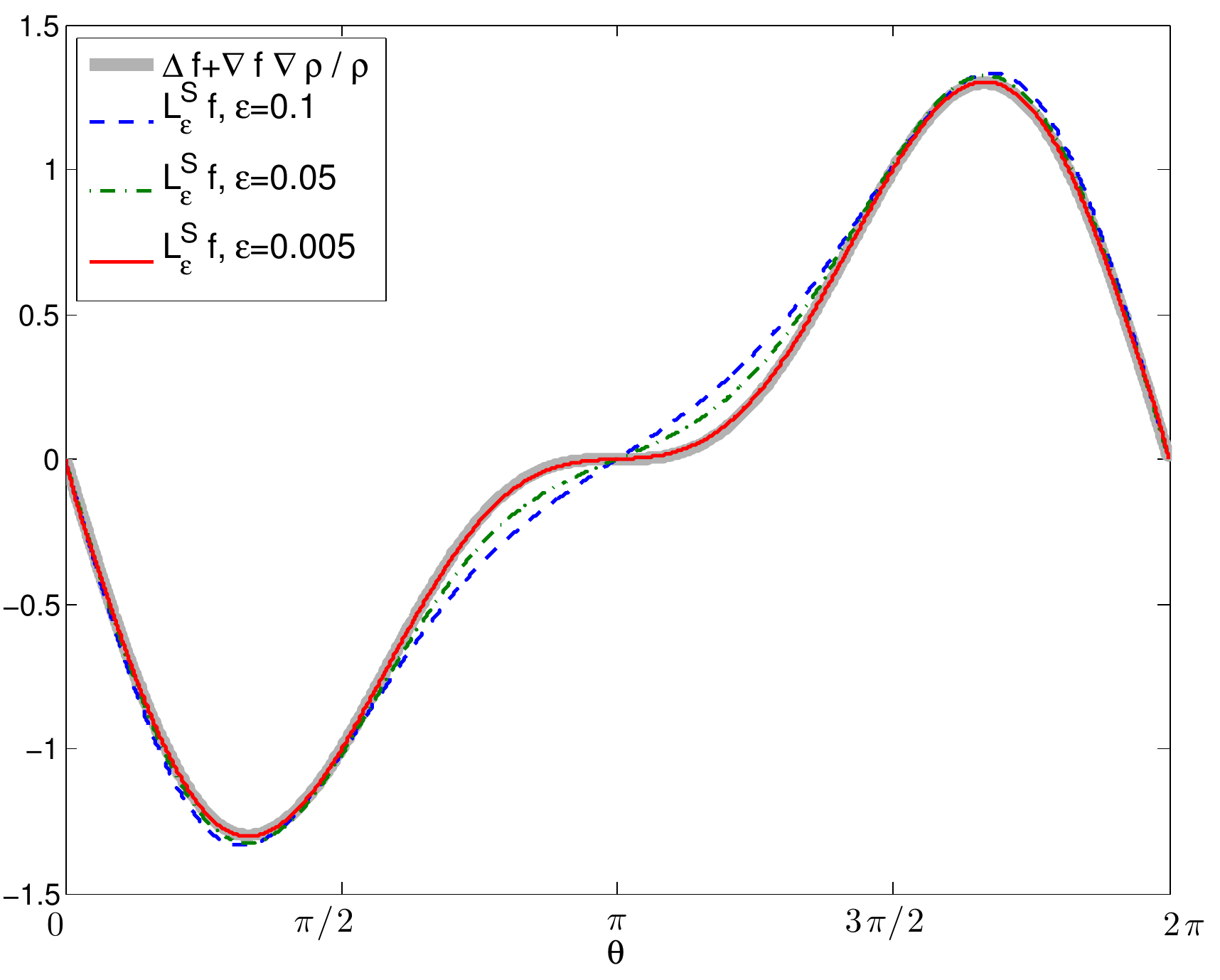}
\caption{\label{symmetricNonUniform} Operators with variable bandwidth $\rho(x)\rho(y)=(q(x)q(y))^{\beta}$ with $\beta=-1/2$ are applied to $f(x) = \sin(x)$ where $x$ parameterizes a unit circle in the plane $\mathbb{R}^2$.  Functions and operators are constructed on a set of 8000 points on the circle, sampled from the density $q(x) = \exp(\cos(x))$.  Each operator is constructed for $\epsilon=0.1$ (blue), $0.01$ (green), and $0.005$ (red).  Left: $\Delta f(x) = -\sin(x)$ (grey) compared to $L^S_{\epsilon}f$ with $\alpha=1/4$, $\beta=-1/2$, note that this is the Laplacian operator.  Right: $\Delta f(x)+\frac{\nabla q(x)}{q(x)}\cdot\nabla f(x) = -\sin(x)-\sin(x)\cos(x)$ (grey) compared to $L^S_{\epsilon}f$ with $\alpha=-1/4$, $\beta=-1/2$ notice that this is the backward Kolmogorov operator for gradient flow with potential $U=-\log (q)=-\cos x$.}
\end{figure}


\section{Convergence Rates for Discrete Operators}\label{discreteOps}

The goal of this section is to analyze the accuracy of the discrete estimates of the continuous kernel operators defined above.  Here, we follow the analysis of Singer \cite{SingerEstimate} and generalize the error estimates to the case of variable bandwidth kernels and non-uniform sampling.  Let $\{x_j\}_{j=1}^N$ be independently sampled according to the density $q(x)$ on the manifold $\mathcal{M} \subset \mathbb{R}^n$ (note that $\mathcal{M}$ is any Riemannian manifold and is not assumed to be compact).  For fixed $x=x_i$ from the data set, define the random variables, 
\begin{align}
F_i(x_j) &= \frac{K^S_{\epsilon}(x_i,x_j)f(x_j)}{\left(\epsilon^{-d/2}N^{-1} q_{\epsilon}^S(x_j) \right)^{\alpha}}, \hspace{50pt} G_i(x_j) = \frac{K^S_{\epsilon}(x_i,x_j)}{\left(\epsilon^{-d/2}N^{-1}q_{\epsilon}^S(x_j)\right)^{\alpha}} \nonumber
\end{align}
where 
\begin{align}
q^S_{\epsilon}(x_j) \equiv \sum_{l} K^S_{\epsilon}(x_j,x_l)/\rho(x_j)^d.  \label{qS}
\end{align}
The functionals $F_i$ and $G_i$ are used in the numerical algorithm to approximate the operator \eqref{symmetricOperatorNonUniform},
\begin{align} L^S_{\epsilon,\alpha}f(x_i) \equiv \frac{1}{\epsilon m \rho(x)^2}\left( \frac{G^S_{\epsilon,q,\alpha}f(x)}{G^S_{\epsilon,q,\alpha}1(x)} - f(x) \right) = \frac{1}{\epsilon m \rho(x_i)^2}\left( \frac{\mathbb{E}[F_i]}{\mathbb{E}[G_i]} - f(x_i) \right) \approx \frac{1}{\epsilon m \rho(x_i)^2}\left( \frac{\sum_j F_i(x_j)}{\sum_j G_i(x_j)} - f(x_i) \right), \label{approx_lsea}\end{align}
where the continuous expectations are defined as,
\[ \mathbb{E}[F_i] \equiv \epsilon^{d/2}G^S_{\epsilon,q,\alpha}f(x) \equiv  \int_{\mathcal{M}} F_i(y)q(y)\, dV(y),\quad\quad \mathbb{E}[G_i] \equiv  \epsilon^{d/2} G^S_{\epsilon,q,\alpha}1(x) \equiv \int_{\mathcal{M}} G_i(y)q(y)\, dV(y) \]
so that the continuous operator $L_{\epsilon,\alpha}^S$ in \eqref{symmetricOperatorNonUniform} agrees with the previous theory in \ref{nonuniform}.  Notice that the factor $\epsilon^{-d/2}N^{-1}$ appearing in $F_i$ and $G_i$ does not need to be known or included in the actual algorithm, since ultimately we will be interested in the ratio $\sum_{j}F_i(x_j) / \sum_{j}G_i(x_j)$ and the factor cancels exactly.  Similarly, while the algorithm and statement of Theorem \ref{mainresult} divide each functional by $q_{\epsilon}^S(x_i)^{\alpha}$, this factor cancels in $L_{\epsilon,\alpha}^S$ in \eqref{symmetricOperatorNonUniform} since the expectations are taken with respect to $x_j$.  Finally, since the density, domain, and number of sample points are the same, the normalization factors Monte-Carlo summations, $\sum_j F_i(x_j)$ and $\sum_j G_i(x_j)$, are identical and are therefore left out.  

In the subsequent sections we will find the error in replacing the continuous expectations with the discrete sums.  If we consider the approximation in \eqref{approx_lsea} to be an estimator for $L^S_{\epsilon,\alpha}f(x_i)$ in \eqref{symmetricOperatorNonUniform} then the bias of the estimator is,
\[ \mathbb{E}\left[ \frac{1}{\epsilon m \rho(x_i)^2}\left( \frac{\mathbb{E}[F_i]}{\mathbb{E}[G_i]} - \frac{\sum_j F_i(x_j)}{\sum_j G_i(x_j)}  \right) \right]. \]
However, since this expectation is difficult to evaluate, we instead follow the analysis of Singer in \cite{SingerEstimate}, which bounds the probability of a large bias error by estimating,
\[ P\left( \frac{1}{\epsilon m \rho(x_i)^2} \left( \frac{\sum_{j\neq i}F_i(x_j)}{\sum_{j\neq i}G_i(x_j)} - \frac{\mathbb{E}[F]}{\mathbb{E}[G]} \right) > a \right). \]
In order to estimate this error, we first need to control the error of the denominators in the functionals $F_i$ and $G_i$.

\subsection{Sampling error in the renormalization factor}\label{samplingErr1}

To analyze the denominator terms of $F_i$ and $G_i$, we define
\begin{align}
H_j(x_l) \equiv \epsilon^{-d/2}K^S_{\epsilon}(x_j,x_l)/\rho(x_j)^d.\nonumber
\end{align}  
From~\ref{symmform} we have,
\begin{align} \mathbb{E}[H_j(\cdot)] &= \frac{\epsilon^{-d/2}}{\rho(x_j)^{d} } \int_{\mathcal{M}} K^S_{\epsilon}(x_j,y)q(y)\, dV(y) =  \frac{G_{\epsilon,q}^S(1)}{\rho(x_j)^{d}} = m_0 q(x_j) \left(1+\epsilon m (\tilde{\omega}(x_j) + \mathcal{L}^S q(x_j) ) \right) + \mathcal{O}(\epsilon^{2}),  \label{EHj}
\end{align}
where we use the expansion \eqref{GSeq1}.   

We first note that in order for the random variables $H_j(x_l)$ to be identically distributed, we must neglect the term $l=j$.  This term is typically included in the implementation of the algorithm, however, the error made by neglecting it is estimated from the expansion,
\begin{align} \frac{1}{N}\sum_l H_j(x_l) &= \frac{1}{N-1+1} \sum_{l\neq j} H_j(x_l) + \frac{\epsilon^{-d/2}}{N} = \frac{(N-1)^{-1}}{1+(N-1)^{-1}} \sum_{l\neq j} H_j(x_l) + \frac{\epsilon^{-d/2}}{N} \nonumber \\
&=\frac{1}{N-1}\sum_{j\neq l} H_{j}(x_l)\Big(1-\frac{1}{N-1} + \mathcal{O}(\frac{1}{(N-1)^{2}}\Big) + \frac{\epsilon^{-d/2}}{N} \nonumber
\end{align}
which shows that the error is 
\begin{align}
\frac{1}{N}\sum_l H_j(x_l) - \frac{1}{N-1}\sum_{j\neq l} H_{j}(x_l) = \mathcal{O}(N^{-1}\epsilon^{-d/2},N^{-2}).\label{mcerror}
\end{align} 
In the remainder of this section, we will use this error bound to replace the summation over all $l$ with the summation over $l\neq j$, since in Theorem \ref{mainresult} we use the definition $q^S_\epsilon(x_i)$ in \eqref{qS} which includes the diagonal term.

We now analyze the error between the discrete Monte-Carlo approximation, $\frac{1}{N-1}\sum_{j\neq l} H_j(x_l)$, and the continuous expectation, $\mathbb{E}[H_j]$.  Letting, $Y_l = H_j(x_l) - \mathbb{E}[H_j]$, we note that $\mathbb{E}[Y_l] = 0$ and for $l\neq j$,
\begin{align} \textup{var}(Y_l) &= \mathbb{E}[Y_l^2] = \mathbb{E}[H_j(x_l)^2] - \mathbb{E}[H_j]^2 \nonumber \\
&= \hat{m}_0 \epsilon^{-d/2}\rho(x_j)^{-d}q(x_j) - m_0^2 q(x_j)^2 + \mathcal{O}(\epsilon^{1-d/2}), \nonumber \\
&=  \hat{m}_0 \epsilon^{-d/2}\rho(x_j)^{-d}q(x_j)  + \mathcal{O}(1), \nonumber
\end{align}
where $\hat{m}_0 \equiv \int_{\mathbb{R}^d} h(\|z\|^2)^2dz$.
We note that $H_j(x_l)$ is bounded due to the exponential decay of $K_{\epsilon}^S$, so by the Chernoff inequality we have, for $a$ sufficiently small,
\begin{align} P\left( \frac{1}{N-1} \left| \sum_{l\neq j} H_j(x_l) - (N-1)\mathbb{E}[H_j]  \right| > a \right) &= P\left( \left| \sum_{l\neq j} Y_l \right| > a(N-1) \right) \nonumber \\
&\leq 2\exp\left( \frac{-a^2(N-1)}{4\hat{m}_0 \epsilon^{-d/2}\rho(x_j)^{-d}q(x_j)}\right). \label{chernoffbound1}
\end{align}
We note that the above bound, commonly known as the Chernoff bound, is actually a less sharp version of a previous bound due to Bernstein.  Note the crucial fact that $q(x_j)$ appears in the denominator, so that as $q\to 0$ the probability of error in the estimate decays.  

Recall that our goal is to expand the ratio $\mathbb{E}[F_i]/\mathbb{E}[G_i]$ up to order-$\epsilon^2$.  Since $q_{\epsilon}^S$ in \eqref{qS}, which is a discrete estimate of the sampling density up to a scalar constant, will appear in the denominators of $F_i$ and $G_i$, we require $N^{-1}\epsilon^{-d/2}q_{\epsilon}^S(x_j) = \frac{1}{N}\sum_{l}H_j(x_l)$ to agree with the continuous limits $\mathbb{E}[H_j]$ up to order-$\epsilon^2$.  Thus we require,
\begin{align} \left| \frac{\epsilon^{-d/2}}{N}q_{\epsilon}^S(x_j) - \mathbb{E}[H_j] \right| &= \left|\frac{1}{N} \sum_{l}H_j(x_l) - \mathbb{E}[H_j] \right| = \left|\frac{1}{N-1} \sum_{l\neq j}H_j(x_l) - \mathbb{E}[H_j] \right| + \mathcal{O}\left(\frac{\epsilon^{-d/2}}{N},\frac{1}{N^2}\right) = \mathcal{O}\left(\epsilon^{2}, \frac{\epsilon^{-d/2}}{N},\frac{1}{N^2}\right) \nonumber \end{align}
with high probability, where the first error term is due to \eqref{mcerror}.  Notice, that balancing the first two error terms requires 
\begin{align}
\epsilon = \mathcal{O}(N^{-1/(2+d/2)}),\label{req1}  
\end{align}
and balancing the first and third error terms requires $\epsilon = \mathcal{O}(N^{-1})$ which is smaller than \eqref{req1} and therefore we neglect the third term. If we assume that $a = \mathcal{O}(\epsilon^{2})$ in \eqref{chernoffbound1}, then we can write $a = \hat a \epsilon^{2}$ where $\hat a = \mathcal{O}(1)$ and we will achieve the desired accuracy with high probability when the exponent of the Chernoff inequality, $a^2 N \epsilon^{d/2}\rho(x_j)^d/q(x_j) = \hat a^2 N \epsilon^{4+d/2}\rho(x_j)^d/q(x_j)$, is large.  In other words, when,
\begin{align}
 \frac{q(x_j)^{1/2}\rho(x_j)^{-d/2}}{N^{1/2}\epsilon^{2+d/4}} = \hat{a} = \mathcal{O}(1),\label{errorEst2}
\end{align}
we attain the desired accuracy with high probability. Notice, that if the numerator, $q^{1/2}\rho^{-d/2} = \mathcal{O}(1)$ this requires $\epsilon = \mathcal{O}(N^{-1/(4+d/2)})$  which dominates the previous requirement in \eqref{req1}, $\epsilon = \mathcal{O}(N^{-1/(2+d/2)})$.  This also shows that the error of order $\mathcal{O}(\epsilon^{-d/2}/N)$ in \eqref{mcerror} from neglecting the diagonal term is negligible compared to the error due to the variance of $Y_l$. This completes the proof of the second term of the error bound in Theorem~\ref{mainresult}.

\subsection{Bounding the statistical bias in the discrete estimate}\label{samplingErr2}

Using the above estimate, we can now consider the case where the summations in the denominators can be replaced by the continuous expectations, so that when $\frac{q(x_j)^{1/2}\rho(x_j)^{-d/2}}{N^{1/2}\epsilon^{2+d/4}} = \mathcal{O}(1)$, we have,
\begin{align}
F_i(x_j) &= \frac{K^S_{\epsilon}(x_i,x_j)f(x_j)}{\left(\mathbb{E}[H_j] + \mathcal{O}(\epsilon^2)\right)^{\alpha}} = \frac{K^S_{\epsilon}(x_i,x_j)f(x_j)}{m_0^\alpha q(x_j)^{\alpha}}(1-\alpha\epsilon m (\tilde \omega(x_j) - \mathcal{L}^s q(x_j))) + \mathcal{O}(\epsilon^2) \nonumber \\
G_i(x_j) &= \frac{K^S_{\epsilon}(x_i,x_j)}{\left( \mathbb{E}[H_j] + \mathcal{O}(\epsilon^2)\right)^{\alpha}} = \frac{K^S_{\epsilon}(x_i,x_j)}{m_0^\alpha q(x_j)^{\alpha}}(1-\alpha\epsilon m(\tilde \omega(x_j) - \mathcal{L}^s q(x_j))) + \mathcal{O}(\epsilon^2). \nonumber
\end{align}
From the expansion in \eqref{GSeqaf}, we deduce, 
\begin{align}
\mathbb{E}[F_i] &=\epsilon^{d/2} G^S_{\epsilon,q,\alpha}(f) = \epsilon^{d/2} f \rho^d(m_0 q)^{1-\alpha}\left(1+ \epsilon m ((1-\alpha)\tilde{\omega} -\alpha \mathcal{L}^s q + \mathcal{L}^s(fq(m_0 q)^{-\alpha})) \right) + \mathcal{O}(\epsilon^{2+d/2}) \nonumber \\
\mathbb{E}[G_i] &= \epsilon^{d/2} G^S_{\epsilon,q,\alpha}(1)  = \epsilon^{d/2}\rho^d(m_0 q)^{1-\alpha}\left(1+ \epsilon m ((1-\alpha)\tilde{\omega} -\alpha \mathcal{L}^s q + \mathcal{L}^s(q(m_0 q)^{-\alpha})) \right) + \mathcal{O}(\epsilon^{2+d/2}) \nonumber. 
\end{align}
Therefore, we can deduce 
\begin{align}
\mathbb{E}[F_i^2] &= \epsilon^{d/2} f^2\rho^d q^{1-2\alpha}m_0^{-2\alpha}\hat{m}_0 \left(1+ \epsilon\tilde{\omega}(\hat{m}-2\alpha m) -2\epsilon\alpha m \mathcal{L}^s q + \epsilon\hat{m}\mathcal{L}^s(f^2 q(m_0 q)^{-2\alpha})) \right) + \mathcal{O}(\epsilon^{2+d/2}), \nonumber \\
\mathbb{E}[G_i^2] &=\epsilon^{d/2} \rho^d q^{1-2\alpha}m_0^{-2\alpha}\hat{m}_0 \left(1+ \epsilon\tilde{\omega}(\hat{m}-2\alpha m) -2\epsilon\alpha m \mathcal{L}^s q + \epsilon\hat{m}\mathcal{L}^s(q(m_0 q)^{-2\alpha})) \right) + \mathcal{O}(\epsilon^{2+d/2}), \\
\mathbb{E}[F_iG_i] &= \epsilon^{d/2} f\rho^d q^{1-2\alpha}m_0^{-2\alpha}\hat{m}_0 \left(1+ \epsilon\tilde{\omega}(\hat{m}-2\alpha m) -2\epsilon\alpha m \mathcal{L}^s q + \epsilon\hat{m}\mathcal{L}^s(fq(m_0 q)^{-2\alpha})) \right) + \mathcal{O}(\epsilon^{2+d/2}), \nonumber
\end{align}
where $\hat{m}_0 \equiv \int_{\mathbb{R}^d} h(\|z\|^2)^2dz$ and $\hat{m}_2 \equiv \int_{\mathbb{R}^d} z_1^2 h(\|z\|^2)^2dz.$ 

Following the analysis of \cite{SingerEstimate} we want to compute,
\[ P\left( \frac{1}{\epsilon m \rho(x_i)^2} \left( \frac{\sum_{j\neq i}F_i(x_j)}{\sum_{j\neq i}G_i(x_j)} - \frac{\mathbb{E}[F_i]}{\mathbb{E}[G_i]} \right) > a \right) = P\left(\sum_{j\neq i} Y_j > a (N-1) \mathbb{E}[G_i]^2 \epsilon m \rho(x_i)^2 \right) \]
where $Y_j = \mathbb{E}[G_i]F_i(x_j) - \mathbb{E}[F_i]G_i(x_j) + a \epsilon m \rho(x_i)^2\mathbb{E}[G_i](\mathbb{E}[G_i]-G_i(x_j))$.   Note that $\mathbb{E}[Y_j] = 0$ and the variance is given by,
\[ \mathbb{E}[Y_j^2] = \mathbb{E}[G_i]^2\mathbb{E}[F_i^2] + \mathbb{E}[F_i]^2\mathbb{E}[G_i^2] - 2\mathbb{E}[G_i]\mathbb{E}[F_i]\mathbb{E}[F_iG_i] + \mathcal{O}(a \epsilon^{1+3d/2}). \]
It is easy to see that the order-$\epsilon^{3d/2}$ terms in the variance are zero, computing the order-$\epsilon^{1+3d/2}$ terms we note that the $\tilde \omega$ terms and $\mathcal{L}^s q$ terms also all cancel.  Letting $\mathcal{H}f = f\mathcal{L}^s f$  where $\mathcal{L}^s$ is defined in \eqref{ELS} we have,
\[ \mathbb{E}[Y_j^2] = \epsilon^{1+3d/2}\rho^{3d}m_0^{2-6\alpha}\hat{m}_2 \left( q^{2-2\alpha}\mathcal{H}(f^2 q^{1-2\alpha}) +  f^2 q^{2-2\alpha} \mathcal{H}(q^{1-2\alpha}) - 2f q^{2-2\alpha} \mathcal{H}(fq^{1-2\alpha})  \right) + \mathcal{O}(a \epsilon^{1+3d/2}+\epsilon^{2+d/2}). \]  
Since $\mathcal{H}f = f\mathcal{L}^s f = \rho^2 \Delta f + (d+2)\rho \nabla \rho \cdot \nabla f$, so that for arbitrary $f,g$ we have,
\begin{align} \mathcal{L}^S(fg) &= \rho^2\Delta (fg) + (d+2)\rho \nabla \rho \cdot \nabla (fg) \nonumber \\
&= \rho^2(f\Delta g + g\Delta f + 2\nabla f \cdot \nabla g) + (d+2)\rho \nabla \rho \cdot (f\nabla g + g\nabla f) \nonumber \\
&= f\mathcal{H} g +g\mathcal{H} f + 2\rho^2 \nabla f \cdot \nabla g. \label{ppty}
\end{align}
Using the definition of $\mathcal{H}$ and \eqref{ppty}, we can simplify the variance as,
\begin{align}\mathbb{E}[Y_j^2] &= 2\epsilon^{1+3d/2}\rho^{3d}m_0^{2-6\alpha}\hat{m}_2 \left(\rho^2 q^{2-2\alpha} \nabla f\cdot \nabla (fq^{1-2\alpha}) - \rho^2 f q^{2-2\alpha} \nabla f \cdot \nabla(q^{1-2\alpha})  \right) + \mathcal{O}(a\epsilon^{1+3d/2}+\epsilon^{2+d/2}) \nonumber \\
&= 2\epsilon^{1+3d/2}\rho^{3d}m_0^{2-6\alpha}\hat{m}_2\left( \rho^2 q^{3-4\alpha} \nabla f\cdot \nabla f \right) + \mathcal{O}(a\epsilon^{1+3d/2}+\epsilon^{2+d/2}) \nonumber \\
&= 2m_0^{2-6\alpha}\hat{m}_2\epsilon^{1+3d/2}  q^{3-4d\alpha} \rho^{2+3d} ||\nabla f||^2 + \mathcal{O}(a\epsilon^{1+3d/2}+\epsilon^{2+d/2}). 
\end{align}
Finally, since $Y_l$ is bounded due to the exponential decay of the kernel, by the Chernoff bound we have,
\begin{align} P\left(\sum_{j\neq i} Y_j > a (N-1) \mathbb{E}[G_i]^2 \epsilon m \rho^2 \right) &\leq 2\exp\left( \frac{ -a^2(N-1)^2\mathbb{E}[G_i]^4 \epsilon^2 m^2 \rho^4 }{4(N-1)\textup{var}(Y_j)} \right) \nonumber \\
&= 2\exp\left( \frac{ -a^2(N-1)m^2m_0^4 \epsilon^{2d+2}(\rho^{4+4d}q^{4-4\alpha}+\mathcal{O}(\epsilon))}{8m_0^{2-6\alpha}\hat{m}_2\epsilon^{1+3d/2}  q^{3-4d\alpha} \rho^{2+3d} ||\nabla f||^2 + \mathcal{O}(a\epsilon^{1+3d/2}+\epsilon^{2+d/2})} \right) \nonumber \\
&= 2\exp\left( \frac{ -a^2(N-1)c \rho^{d+2}q^{1+4d\alpha-4\alpha}}{4 \epsilon^{-1-d/2} ||\nabla f||^2} \right), \nonumber
\end{align}
where $c = m^2m_0^{2+6\alpha}/(8\hat{m}_2)$. So when $a=\mathcal{O}(\epsilon)$, we can write $a = \hat{a}\epsilon$  where $\hat{a}=\mathcal{O}(1)$ and we can solve for $\hat{a}$ to find the expected magnitude of errors to be,
\begin{align}\label{errorEst3}
\frac{||\nabla f(x_i)||q^{-(1/2-2\alpha+2d\alpha)}\rho^{-(d/2+1)}}{\sqrt{N}\epsilon^{1/2+d/4}} = \hat{a} = \mathcal{O}(1),
\end{align}
and this completes the proof of the third error bound in Theorem~\ref{mainresult}.

\comment{
By choosing bandwidth $\rho = q^{\beta}$ the bound becomes,
\[ P\left(  \frac{1}{\epsilon m \rho(x_i)^2}  \left( \frac{\sum_{j\neq i}F_i(x_j)}{\sum_{j\neq i}G_i(x_j)} - \frac{\mathbb{E}[F]}{\mathbb{E}[G]} \right) > a \right) \leq 2\exp\left( \frac{ -a^2(N-1)m^2 q^{1-4\alpha+4d\alpha+d\beta+2\beta}}{4 \epsilon^{-1-d/2} c ||\nabla f||^2} \right) \]
so when the exponent is $\mathcal{O}(1)$ we can solve for $a$ to find the expected magnitude of errors to be,
\begin{align}\label{errorEst3} \mathcal{O}\left(\frac{||\nabla f||q^{-c_2}}{N^{1/2}\epsilon^{1/2+d/4}}\right), \end{align}
where $c_2 = 1/2-2\alpha+2d\alpha +d\beta/2+\beta$.  Notice that the error will be proportional to $q^{-c_2}$ so for $c_2>0$ the errors may be unbounded as $q\to 0$.

Combining the error estimates \eqref{fullSymmetricResult}, \eqref{errorEst2}, and \eqref{errorEst3} we have,
\[L^S_{\epsilon,\alpha}f(x_i) \equiv  \frac{1}{\epsilon m \rho(x_i)^{2}}\left(\frac{\sum_{j}F_i(x_j)}{\sum_{j}G_i(x_j)}-f(x_i)\right)  = \Delta f(x_i) + c_1\nabla f(x_i) \cdot \frac{\nabla q(x_i)}{q(x_i)}+ \mathcal{O}\left(\epsilon, \frac{q(x_i)^{(1-d\beta)/2}}{\sqrt{N}\epsilon^{2+d/4}}, \frac{||\nabla f(x_i)||q(x_i)^{-c_2}}{\sqrt{N}\epsilon^{1/2+d/4}} \right), \]
which proves Theorem \ref{mainresult}.
}


\end{document}